\documentclass{amsart}
\usepackage{amsbsy}
\usepackage{amsmath}
\usepackage{amstext}
\usepackage{amsthm}
\usepackage{amscd}
\usepackage{amssymb}
\usepackage{enumitem}
\usepackage{url}
\usepackage{tikz-cd}
\usepackage{bm}
\usepackage{graphicx}

\makeatletter
\newcommand*{\rom}[1]{\expandafter\@slowromancap\romannumeral #1@}
\makeatother

\theoremstyle{plain}
\newtheorem{thm}{Theorem}[section]
\newtheorem{cor}[thm]{Corollary}

\newtheorem{lem}[thm]{Lemma}
\newtheorem{prop}[thm]{Proposition}

\theoremstyle{definition}
\newtheorem{defn}[thm]{Definition}

\theoremstyle{remark}
\newtheorem{rem}[thm]{Remark}
\newtheorem{expl}[thm]{Example}

\numberwithin{equation}{section}

\newcommand{\modulo}[1]{{\,(\,\operatorname{mod}\,#1)}}
\newcommand\GL{{\mathrm {GL}}}
\newcommand\SL{{\mathrm {SL}}}
\newcommand\tensor{{\> \otimes \> }}
\newcommand\Gal{{\mathrm {Gal}}}

\newcommand{\Qb}{{\mathbb Q}}
\newcommand{\Rb}{{\mathbb R}}
\newcommand{\Cb}{{\mathbb C}}
\newcommand{\Eb}{{\mathbb E}}
\newcommand{\Ib}{{\mathbb I}}
\newcommand{\Tbb}{{\mathbb T}}

\newcommand{\Zb}{{\mathbb Z}}

\newcommand\Pbb{{\mathbb P}}
\newcommand\Vb{{\mathbb V}}
\newcommand\Tb{{\mathbf T}}

\newcommand{\wbb}{{\mathbf w}}
\newcommand{\xbb}{{\mathbf x}}

\newcommand{\Bc}{{\mathcal B}}

\newcommand{\Hc}{{\mathcal H}}
\newcommand{\Lc}{{\mathcal L}}
\newcommand{\Oc}{{\mathcal O}}

\newcommand{\Fc}{{\mathcal F}}

\newcommand{\mG}{{\mathfrak m}}
\newcommand{\cG}{{\mathfrak c}}

\newcommand{\al}{{\alpha}}

\newcommand{\qb}{{\bf q}}
\newcommand{\Qbb}{{\bf Q}}

\newcommand{\aG}{{\mathfrak{a}}}

\newcommand{\bG}{{\mathfrak{b}}}

\newcommand{\twobytwo}[4]{{\begin{bmatrix} #1 & #2 \\ #3 & #4
\end{bmatrix}}}
\newcommand{\cl}[1]{{\overline{#1}}}
\newcommand{\ds}{\displaystyle}

\newcommand{\hide}[1]{{}}
\newcommand{\cha}[1]{{\left\langle #1\right\rangle}}
\newcommand{\twobytwotiny}[4]{{\left[\begin{smallmatrix} #1 & #2 \\ #3 & #4
\end{smallmatrix}\right]}}

\newcommand{\rdot}{\!\cdot\!}

\newcommand{\pb}{{\mathbf p}}

\newcommand{\Ebb}{{\bm \psi}}

\theoremstyle{plain}
\newtheorem{thmL}{Theorem}

\newtheorem{conjL}{Conjecture}

\begin{document}

\title[Continued fractions and modular Symbols]{Dynamics of Continued Fractions and Distribution of Modular Symbols}
\author{Jungwon Lee}
\email{jungwon@unist.ac.kr; jungwon.lee@warwick.ac.uk}
\address{Department of Mathematical Sciences, Ulsan National Institute of Science and Technology, Ulsan, Korea}
\curraddr{Mathematics Institute, University of Warwick, Coventry CV4 7AL, UK}

\author{Hae-Sang Sun}
\email{haesang@unist.ac.kr}
\address{Department of Mathematical Sciences, Ulsan National Institute of Science and Technology, Ulsan, Korea}

\date{\today}

\begin{abstract}
We formulate a dynamical approach to the study of distribution of modular symbols, motivated by the work of Baladi-Vall\'ee. We introduce the modular partitions of continued fractions and observe that the modular symbols are special cases of modular partitions. We prove the limit Gaussian distribution and residual equidistribution for modular partitions as random variables on the set of rationals whose denominators are up to a fixed positive integer, by studying the spectral properties of transfer operator associated to the underlying dynamics. The approach leads to a few applications. We show an average version of conjectures of Mazur-Rubin on statistics for modular symbols of rational elliptic curves. We further observe that the equidistribution of mod $p$ values of modular symbols leads to mod $p$ non-vanishing result for special modular $L$-values twisted by Dirichlet characters.
\end{abstract}

\thanks{A part of this work was supported by the TJ Park Science Fellowship of POSCO TJ Park Foundation. It was also supported by Basic Science Research Program through the National Research Foundation of Korea(NRF) funded by the Ministry of Education(NRF-2017R1A2B4012408, NRF-2019R1A2C108860913, NRF-2020R1A4A1016649)}

\subjclass[2010]{11F67, 37C30}

\keywords{Mod $p$ non-vanishing of special $L$-values, Modular symbols, Mazur-Rubin conjecture, Continued fractions, Skewed Gauss map, Transfer operators}

\maketitle
\setcounter{tocdepth}{1}
\tableofcontents

\section{Introduction and statements of results}

The statistics of continued fraction has been a rich source of research. For instance, it is a longstanding conjecture that the distribution of length of continued fractions over the rational numbers follows the Gaussian distribution. 
More precisely, for a rational number $r \in (0,1)$, write $[0;m_1, m_2, \cdots, m_\ell]$ for the continued fraction expansion of $r$ where $m_1,\cdots,m_{\ell-1}$ are integers greater than $0$ and $m_\ell$ is an integer greater than $1$, and $\ell=\ell(r)$ is the length of the expansion. We consider a set 
$\Sigma_M:= \left\{ \frac{a}{M}\, \big|\,\ 1 \leq a < M, \ (a,M)=1 \right\}.$
One can regard $\Sigma_M$ as a probability space with a uniform distribution and $\ell$ as a random variable on $\Sigma_M$. The unsettled conjecture is that the variable $\ell$ follows asymptotically the Gaussian distribution as $M$ goes to infinity. 

The first prominent result goes back to Hensley \cite{hensley}. He obtained a partial result on the problem in an average setting, in other words, instead of $\Sigma_M$, he proved the conjecture for a larger probability space 
$$\Omega_M=\bigcup_{n\leq M}\Sigma_n.$$ 
Later, Baladi--Vall\'ee \cite{bv} showed the average version in full generality with an optimal error based on the dynamical analysis of Euclidean algorithm.

In this paper, we study the statistics of generalisations of the variable $\ell$, so-called modular partition functions.

\subsection{Modular partition functions}

Let ${P_i}/{Q_i}$ be the $i$-th convergent of $r$, i.e., 
$$\frac{P_i}{Q_i}=[0;m_1, \cdots, m_i],\,P_0=0,\,Q_0=1.$$ 
For $1\leq i\leq \ell$, we define $2\times 2$ integral matrices 
\begin{align*}
g_i(r):=\twobytwotiny{P_{i-1}}{P_i}{Q_{i-1}}{Q_i}\in\GL_2(\Zb) \mbox{ and }
g(r):=g_{\ell}(r).
\end{align*}
The matrices satisfy a recurrence relation $g_{i+1}(r)=g_i(r) \twobytwotiny{0}{1}{1}{m_{i+1}}.$

Let $\Gamma$ be a subgroup of $\SL_2(\Zb)$. 
For a right coset $u \in \Gamma \backslash \GL_2(\Zb)$ and a rational $r\in (0,1)$, a natural quantity to consider is $\#\{1\leq i\leq \ell\,|\, \Gamma g_i(r)\in u\}$. We observe that it is written as $\sum_{i=1}^\ell \Ib_u(\Gamma g_i(r))$ where $\Ib_u(v)=1$ if $u=v$ and $0$ otherwise. Extending it to a function $\psi$ on $\Gamma\backslash\GL_2(\Zb)$, let us define a more general quantity
$$\aG_{\psi}(r):=\sum_{i=1}^\ell\psi(\Gamma g_i(r)).$$

In order to define an $\SL_2$-version, let us introduce $${\rm j}:=\twobytwo{1}{0}{0}{-1}.$$ In this paper, we assume that 
$$[\SL_2(\Zb):\Gamma]\mbox{ is finite and }\Gamma\mbox{ is normalised by }{\rm j}.$$
For $g\in\GL_2(\Zb)$, we define
\begin{align*}
\widehat{g}:=\begin{cases}g&\mbox{ if }\det(g)=1\\
{\rm j}g&\mbox{ otherwise}
\end{cases}\mbox{ and }
\widetilde{g}:=\begin{cases}g&\mbox{ if }\det(g)=1\\
g{\rm j}&\mbox{ otherwise}
\end{cases}.
\end{align*}
For a function $\psi$ on $\Gamma\backslash\SL_2(\Zb)$, we define 
\begin{align*}
\bG_\psi(r):=\sum_{i=1}^\ell \psi(\Gamma\widehat{g_i}(r))\mbox{ and }\cG_\psi(r):=\sum_{i=1}^\ell \psi(\Gamma\widetilde{g_i}(r)).
\end{align*}
The functions ${\aG}_\psi$, ${\bG}_\psi$, and ${\cG}_\psi$ are called \emph{modular partition functions} or \emph{modular cost functions}. 

One of the main goals in the present paper is to determine the moment generating functions of random variables $\bG_\psi$ and $\cG_\psi$ on $\Omega_M$. In particular, there are two applications: the recent conjecture of Mazur--Rubin on the distribution of modular symbols and the non-vanishing modulo $p$ of special $L$-values of modular forms.

\subsection{Main results}
Let $I$ be the interval $[0,1]$. For the later applications, we study more general probability spaces. 
For a map $\varphi$ on the right cosets of $\Gamma$, denote the functions on $I \cap \Qb$ given as
\begin{align*}
\varphi: r\mapsto \varphi\big(\Gamma{g}(r)\big), \ r\mapsto \varphi\big(\Gamma\widehat{g}(r)\big),\mbox{ or }r\mapsto \varphi\big(\Gamma\widetilde{g}(r)\big)
\end{align*}
by the same symbol $\varphi$ according to the context unless any confusion arises.
For an open sub-interval $J\subseteq I$ and a non-trivial non-negative function $\varphi$, let $$\Omega_{M,\varphi,J}$$  be a probability space $\Omega_M\cap J$ with a density function $(\sum_{r\in\Omega_M\cap J}\varphi(r))^{-1}\varphi$
as long as the denominator is non-zero.  For a random variable $\mathfrak{g}$ on a probability space $X$, we denote by $\Pbb[\mathfrak{g}|X]$, $\Eb[\mathfrak{g}|X]$, and $\Vb[\mathfrak{g}|X]$ the probability, mean, and variance of $\mathfrak{g}$ on $X$, respectively.


In order to state the main results we also need the following:

\begin{defn}\label{repn:psi}
	Let $\Bbbk$ be an abelian group and $\psi: {\Gamma\backslash\SL_2(\Zb)}\rightarrow \Bbbk$. 
	\begin{enumerate}
		\item If there exists a $\Bbbk$-valued function $\beta$ on $\Gamma\backslash\SL_2(\Zb)$ such that
		$$\psi(u)=\beta(u)-\beta\Big(u\cdot \twobytwo{-m}{1}{1}{0}\Big)$$
		for all $u\in\Gamma\backslash\SL_2(\Zb)$ and integers $m\geq 1$, then $\psi$  is called a {\it $\bG$-coboundary} over $\Bbbk$, {\it associated with} $\beta$. Let $\mathcal{B}_{\bG}(\Gamma,\Bbbk)$ be the abelian group of all $\bG$-coboundaries over $\Bbbk$.
		\item If there exists a $\Bbbk$-valued function $\beta$ on $\Gamma\backslash\SL_2(\Zb)$ such that
		$$\psi(u)+\psi\Big(u \twobytwo{-n}{1}{1}{0}{\rm j}\Big)=\beta(u)-\beta\Big(u \twobytwo{-n}{1}{1}{0}\twobytwo{-m}{1}{1}{0}\Big)$$
		for all $u\in\Gamma\backslash\SL_2(\Zb)$ and integers $m,n\geq 1$, then $\psi$ is called a {\it $\cG$-coboundary}  over $\Bbbk$, {\it associated with} $\beta$. Let $\mathcal{B}_{\cG}(\Gamma,\Bbbk)$ be the abelian group of all $\cG$-coboundaries over $\Bbbk$.
	\end{enumerate}
	
\end{defn}

\begin{expl}\label{expl:cob}
	For a prime $p$ and $\Gamma=\Gamma_0(p)$, we set $$u_1=\Gamma \twobytwo{*}{*}{1}{0},\, u_2=\Gamma\twobytwo{*}{*}{0}{1}=\Gamma.$$ One can show that $\Ib_{u_1}$ is neither a $\bG$- nor $\cG$-coboundary over $\Rb$. Observe that  $\Ib_{u_1}(v\cdot \twobytwotiny{-m}{1}{1}{0})=\Ib_{u_2}(v)$	for all $v\in\Gamma\backslash\SL_2(\Zb)$ and $m\in\Zb$. Hence, $\psi=\Ib_{u_1}-\Ib_{u_2}$ is a $\bG$-coboundary associated with $\Ib_{u_1}$. As $\psi(u{\rm j})=\psi({\rm j}u)$ for all $u$, we can also show that $\psi$ is a $\cG$-boundary associated with $\Ib_{u_1}$. Hence, $\phi=\Ib_{u_1}-\frac{1}{2}\Ib_{u_2}$ is neither a $\bG$- nor $\cG$-coboundary over $\Rb$. A numerical example is presented in Figure \ref{fig:gauss}.
\end{expl}



\subsubsection{Joint Gaussian distribution}

One of the main results is  that a vector of modular partition functions follows the Gaussian distribution asymptotically.

\begin{thmL} \label{vector:far0}
Let $J$ be a non-empty open sub-interval of $(0,1)$, $\varphi$ a non-trivial non-negative function on $\Gamma\backslash\SL_2(\Zb)$; and $\mathfrak{g}=\bG$ or $\cG$.
\begin{enumerate}[leftmargin=*]
\item Let ${\bm \psi}:\Gamma\backslash\SL_2(\Zb)\rightarrow \Rb^d$ with 
${\bm \psi}=(\psi_1,\cdots,\psi_d)$. Set $\mathfrak{g}_{\bm \psi}:=(\mathfrak{g}_{\psi_1},\cdots,\mathfrak{g}_{\psi_d})$. For each ${\bm \psi}$, there exists  ${\rm H}_{\bm \psi}\in M_d(\Rb)$ (See \S \ref{sec:joint:gauss} for definition) such that:
\begin{enumerate}[leftmargin=*]
	\item  ${\rm H}_{\bm \psi}$ is non-singular if and only if $\psi_1,\cdots,\psi_d$ are $\Rb$-linearly independent modulo $\mathcal{B}_\mathfrak{g}(\Gamma,\Rb)$. 
	\item When ${\rm H}_{\bm \psi}$ is non-singular, the distribution of ${\mathfrak{g}_{\bm\psi}}$ on $\Omega_{M,\varphi,J}$ is asymptotically Gaussian as $M\rightarrow\infty$. More precisely, there exists $\mu_{\bm \psi}\in \Rb^d$ such that
for any $\xbb\in \Rb^{d}$, 
\begin{align*} 
\Pbb &\left[ \frac{{\mathfrak{g}_{\bm\psi}}- \mu_{\bm \psi} \log M}{\sqrt{\log M}}  \leq \xbb\, \Big\lvert\, \Omega_{M,\varphi,J} \right]\\
 &\phantom{blaaank}= \frac{	1}{(2 \pi)^{d/2} \sqrt{\det {\rm H}_{\bm \psi}}} \int_{\mathbf{t} \leq \xbb} \exp \left(-\frac{1}{2} \mathbf{t}^T {\rm H}_{\bm \psi}^{-1} \mathbf{t}  \right) d \mathbf{t} + O \left(\frac{1}{\sqrt{\log M}} \right)  
\end{align*}
where $\mathbf{t} \leq \xbb$ means $t_j \leq x_j$ for all $1 \leq j \leq d$ and the implicit constant is uniform in $\xbb$.
	\end{enumerate}
\item Let $d=1$. For $\psi:\Gamma\backslash\SL_2(\Zb)\rightarrow \Rb$ and $C_\psi={\rm H}_{\psi}$, there exists  $D_{\psi,\varphi,J}$ such that the variance satisfies 
		\[  \Vb[\mathfrak{g}_\psi|\Omega_{M,\varphi,J}]=C_\psi \log M +D_{\psi,\varphi,J} +O(M^{-\gamma}) \]
		for a  $\gamma >0$. In particular, $C_\psi=0$ if and only if $\psi$ is a $\mathfrak{g}$-coboundary over $\Rb$. 

\item Let $d=1$ and $k\geq 3$. There exists a polynomial $Q_{J,\varphi,k}$ of degree at most $k$ such that
		$$\Eb[\mathfrak{g}_\psi^{ k}\,|\,\Omega_{M,\varphi,J}] = Q_{J,\varphi,k}(\log M)+O((\log M)^{k}M^{-\gamma}).$$
\end{enumerate}
\end{thmL}

\begin{expl}\label{cob:examp}
For $\psi$ in Example \ref{expl:cob}, note that $\bG_\psi(r)=\sum_{i=1}^\ell \Ib_{u_1}(\Gamma\widehat{g}_i)-\Ib_{u_1}(\Gamma\widehat{g}_i\cdot\twobytwotiny{-m_{i}}{1}{1}{0})$, which is equal to $\Ib_{u_1}(\Gamma\widehat{g}_\ell)-\Ib_{u_1}(\Gamma)$. Since $\twobytwotiny{a}{b}{c}{d}\in{u_1}$ if and only if $d\equiv0\modulo{p}$, we conclude that
	$$\bG_{\psi}(r)=\Ib_{u_1}(\Gamma\widehat{g}_\ell)=\begin{cases}
		0&\mbox{ if } Q(r)\not\equiv 0\modulo{p}\\
		1&\mbox{ if }Q(r)\equiv 0\modulo{p}
	\end{cases}.$$
In particular, it does not follow the Gaussian distribution asymptotically.
\end{expl}

\hide{----
\begin{expl}\label{expl:gauss}
Let us use the notations in Example \ref{expl:cob}. Since $\Ib_{u_1}$ is not a $\bG$-coboundary over $\Rb$, we conclude that $\phi=\Ib_{u_1}-\frac{1}{2}\Ib_{u_2}$ is not a $\bG$-coboundary over $\Rb$ too. A numerical example is presented in Firgure \ref{fig:gauss}.
\end{expl}
-----}

\begin{figure}
  \includegraphics[width=8cm]{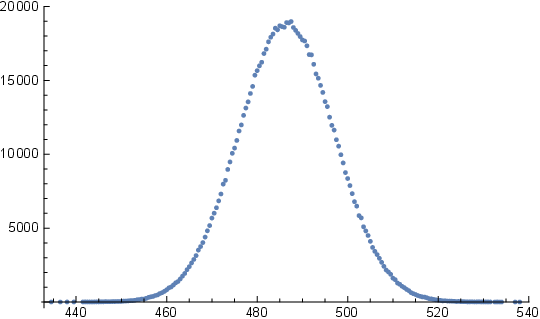}
  \caption{Distribution of $\bG_\phi$ in Example \ref{expl:cob} on the random samples of size $10^{6}$ chosen from $\Omega_{2^{10000}}$ for $\Gamma=\Gamma_0(5)$}
  \label{fig:gauss}
\end{figure}

\begin{rem}
Numerical evidences suggest that an analogue of Theorem \ref{vector:far0} for the variable $\aG$ is plausible. However, due to the fact that the relevant transfer operator for $\aG$ is not topologically mixing (See Remark \ref{not:T:adm}), the arguments in the present paper do not work for the variable $\aG$. The second-named author plans to provide an approach to deal with this problem in future work.
\end{rem}

\subsubsection{Residual equidistribution}

Another main result is the equidistribution of integer valued modular partition functions in residue classes of a fixed modulus. Let $\Omega_{M,J}:=\Omega_M\cap J$.

\begin{thmL}\label{intro:equid:mod:q}
Let $\mathfrak{g}=\bG$ or $\cG$ and $Q\geq 3$ be an integer. Then:
\begin{enumerate}[leftmargin=*]
\item  Let $J$ be a non-empty open sub-interval of $(0,1)$ and ${\bm \psi}=(\psi_1,\cdots,\psi_d):\Gamma\backslash\SL_2(\Zb)\rightarrow \Zb^d$. If $\psi_1,\cdots,\psi_d\modulo{Q}$ are $\Zb/Q\Zb$-linearly independent modulo $\mathcal{B}_\mathfrak{g}(\Gamma,\Zb/Q\Zb)$, then for all ${\bf g}\in (\Zb/Q\Zb)^{d}$ we have
\begin{align*}
	\Pbb[\,{\mathfrak{g}_{\bm\psi}}\equiv {\bf g} \modulo{Q}\,|\,\Omega_{M,J}]&=Q^{-d}+O(M^{-\delta}) 
\end{align*}
for some $\delta>0$. 
\item Let $Q$ be relatively prime to $[\SL_2(\Zb):\Gamma]$ and $\psi:\Gamma\backslash\SL_2(\Zb)\rightarrow\Zb$. If  we have for a non-empty open sub-interval $J$ that
		\begin{align*}
			\lim_{M\rightarrow\infty}\Pbb[\,\mathfrak{g}_\psi\equiv g \modulo{Q}\,|\,\Omega_{M,J}]&=Q^{-1} 
		\end{align*}
for all $g\in \Zb/Q\Zb$, then $\psi\modulo{q}$ is not a $\mathfrak{g}$-coboundary over $\Zb/q\Zb$ for each prime divisor $q$ of $Q$.
\end{enumerate}
\end{thmL}

\begin{rem}
From Example \ref{cob:examp}, we see that the random vectors $(\bG_{\Ib_u})_{u\in\Gamma\backslash\SL_2(\Zb)}$ and $(\cG_{\Ib_u})_{u\in\Gamma\backslash\SL_2(\Zb)}$ are not equidistributed modulo $Q$ for $Q\geq 3$.
\end{rem}

A specialisation $d=1$ and $\psi\equiv1$ gives us a new result that the length $\ell$ of the continued fractions on $\Omega_{M,J}$ is residually equidistributed. 

\begin{cor}
	For $g\in\Zb/Q\Zb$, we have
	$$\Pbb[\ell\equiv g\modulo{Q}|\Omega_{M,J}]=Q^{-1}+o(1).$$
\end{cor}

\subsection{Applications of main results}

First, we introduce an application of Theorem \ref{vector:far0}.

\subsubsection{Conjecture of Mazur-Rubin}\label{intro:mazur:rubin}

In order to understand the growth of Mordell-Weil ranks of a rational elliptic curve in large abelian extensions,  Mazur and Rubin \cite{mazur:rubin} described heuristically the behavior of special values of twisted modular $L$-functions by presenting the conjecture on statistics for modular symbols based on numerical computations.

Let $\Gamma_0(N)=\SL_2(\Zb)\cap \twobytwotiny{\Zb}{\Zb}{N\Zb}{\Zb}$. Let $f$ be a newform for $\Gamma_0(N)$ and of weight 2 with the Fourier coefficients $a_f(n)$. Let $\chi$ be a Dirichlet character of conductor $M$. 
We denote by $L(s,f,\chi)$ the twisted modular  $L$-function, which is given as the meromorphic continuation of the Dirichlet series with the coefficients $a_f(n)\chi(n)$. 
Let $\Qb_f$ be the field generated by the coefficients $a_f(n)$ over $\Qb$. It is known that $\Qb_f$ is  real. There are suitable periods $\Omega_f^\pm$ such that the following normalised special $L$-values are algebraic, more precisely, we have
\[
L_f(\chi):=\frac{G(\overline{\chi}) L(1,f,\chi)}{\Omega_f^\pm}\in\Qb_f(\chi)
\]
where $G(\overline{\chi})$ denotes the Gauss sum and $\pm$ corresponds to the sign $\chi(-1)=\pm 1$.

The \emph{modular symbols} are period integrals of the form
\[ 
\mG^\pm_f(r) := \frac{1}{\Omega_f^\pm}\left\{ \int_r^{i \infty} f(z)dz \pm \int_{-r}^{i \infty} f(z)dz   \right\} \in \Qb_f
\]
for $r\in\Qb$. 
We regard $\mG_f^\pm$ as a random variable. Set $\mG_E^\pm=\mG_{f_E}^\pm$ for the newform  $f_E$ corresponding to an elliptic curve $E$ over $\Qb$. The periods $\Omega_{f_E}^\pm$ can be chosen as the N\'eron periods $\Omega_E^\pm$. Mazur--Rubin  \cite{mazur} proposed:

\begin{conjL}[Mazur--Rubin] \label{mazur:rubin}
Let $E$ be an elliptic curve over $\Qb$ of conductor $N$.  Then:
\begin{enumerate}[leftmargin=*]
\item\label{m:r:conj:gauss} The random variable $\mG_E^\pm$ on $\Sigma_M$ follows the asymptotic Gaussian distribution as $M$ goes to infinity. 
\item For a divisor $d$ of $N$, there exist two constants $C_E^\pm$ and $D_{E,d}^\pm$, called the \emph{variance slope} and the \emph{variance shift}, respectively such that 
\[ \lim_{M \rightarrow \infty \atop (M,N)=d} \Vb[\mG_E^\pm|\Sigma_M]-C_E^\pm \log M= D_{E,d}^\pm. \]
\end{enumerate}
\end{conjL}

Petridis--Risager \cite{petridis:risager} obtained the $\Omega_M$-version of the statement (1) for general cuspforms $f$ of cofinite Fuchsian groups and the statement (2) for congruence subgroup $\Gamma_0(N)$ with a square-free integer $N$. They could give an explicit formula for the constant $C_f^\pm$ as well as $D_{f,d}^\pm$ in terms of the special values of a symmetric square $L$-function of $f$. 
They further established an interval version of (1), that is, for any interval $J \subseteq {I}$, the variable $\mG_f^\pm$ on $\Omega_M \cap J$ follows the Gaussian distribution asymptotically.  Their approach is based on the sophisticated theory of non-holomorphic Eisenstein series twisted by the moments of modular symbols. Their work has been generalised to arbitrary weights by Nordentoft \cite{nord}. 

In this paper, we present another proof of the average version of Conjecture \ref{mazur:rubin} for a newform of weight 2 for $\Gamma_0(N)$ and an arbitrary $N$ as a specialisation of the result (Theorem A) on the modular partition functions. 

\begin{thmL}\label{result:symbol}
Let $f$ be a newform for $\Gamma_0(N)$ and of weight $2$. Then:
\begin{enumerate}[leftmargin=*]
	\item The random variable $\mG_f^\pm$ on $\Omega_{M,\varphi,J}$ follows asymptotic Gaussian distribution as $M$ goes to infinity. More precisely, there exist $\sigma_f^\pm$ and  $C_f^\pm>0$ such that
	\begin{align*} 
		\Pbb &\left[ \frac{\mG_f^\pm- \sigma_{f}^\pm \log M}{\sqrt{C_f^\pm\log M}}  \leq x\, \Big\lvert\, \Omega_{M,\varphi,J} \right]= \frac{1}{2 \pi } \int_{-\infty}^x \exp \left(-\frac{1}{2} t^2\right) dt + O \left(\frac{1}{\sqrt{\log M}} \right).  
	\end{align*}
	Here the implicit constant is independent of $x$.
	
	\item \label{var:shift} The \emph{variance slope} $C_f^\pm$ is independent of $\varphi$ and there exists  \emph{variance shift} $D_{f,\varphi,J}^\pm$ such that
	\[ \Vb[\mG_f^\pm|\Omega_{M,\varphi,J}]=C_f^\pm \log M + D_{f,\varphi,J}^\pm + O(M^{-\gamma}).\]
	\item Let $k\geq 3$. There exists a polynomial $Q_{J,\varphi,k}$ of degree at most $k$ such that
	$$\Eb[(\mG_f^{\pm})^{ k}\,|\,\Omega_{M,\varphi,J}] = Q_{J,\varphi,k}(\log M)+O((\log M)^{k}M^{-\gamma}).$$
\end{enumerate}
\end{thmL}

Theorem \ref{result:symbol} directly implies the result of Petridis--Risager or the average version of Conjecture \ref{mazur:rubin} with specific choices of  $\varphi$. For (1), we take $\varphi=1$. For a divisor $d$ of $N$, we define $\varphi_d \left( \twobytwotiny{\alpha}{\beta}{\gamma}{\delta}  \right)=1$ when $(\delta, N)=d$ and 0 otherwise. Note that $\varphi_d$ is well-defined on $\Gamma_0(N) \backslash \SL_2(\Zb)$. The particular choice $\varphi=\varphi_d$ yields that $\Omega_M$-version of Conjecture \ref{mazur:rubin}.(2) is a special case of our result.

Theorem \ref{result:symbol} is a specialisation of Theorem \ref{vector:far0}: More precisely, there is a function $\psi_f^\pm$ on $\Gamma_0(N)\backslash\SL_2(\Zb)$ such that it is not a coboundary over $\Rb$ and $\mG_f^\pm$ follows the distributions of $\bG_{\psi_f^+}$ and $\cG_{\psi_f^-}$ (See \S\ref{subsec:manin:trick} and \S\ref{subsec:normal:mod:sym}). 

\begin{rem}
From the specialisations $\bG_{\psi_f^+}$ and $\cG_{\psi_f^-}$, we observe that the asymptotic normality of modular symbols is originated essentially not from the modularity of $f$,  but from the dynamics of continued fractions. 
The modularity in our paper plays a role only in showing that $\psi_f^\pm$ is not a coboundary (See \S\ref{subsec:manin:trick}); it is also a crucial ingredient in calculating the mean (Diamantis {\it et al.} \cite{d:h:k:l}, Sun \cite{sun}), the variance slope and shift (Petridis--Risager \cite{petridis:risager}, Blomer {\it et al.} \cite{b.f.k.m.m.s}).
\end{rem}

\begin{rem}
One may wonder if $\bG_{\psi_f^-}$ can be used to study the modular symbols instead of $\cG_{\psi_f^-}$. In fact, the answer is negative. It is the action of ${\rm j}$ that prohibits $\mathfrak{m}_f^-$ from being expressed in terms of $\bG_{\psi_f^-}$. We refer to Remark \ref{rem:not:bG} and \ref{man:mar:error}.
\end{rem}

\begin{rem}
Bettin--Drappeau \cite{bettin:drappeau} showed the asymptotic Gaussian distribution of modular symbols for level $1$ and arbitrary higher weights (See Remark \ref{k12:N1}). We speculate that by adopting their arguments, our work can be extended to arbitrary weights.
\end{rem}

\begin{rem}
Even though computable in polynomial time (Lhote \cite{lhote}), no closed forms for the variance slope and shift for the length $\ell$ is known from the dynamical approach. It is an interesting question whether the expressions of Petridis--Risager for $C_f^\pm$ and $D_{f,\varphi}^\pm$ are hints for this open problem.
\end{rem}

In the next section, we discuss an application of Theorem \ref{intro:equid:mod:q}.

\subsubsection{Non-vanishing mod $p$ of modular $L$-values}
Non-vanishing of twisted $L$-values seems to genuinely rely on the equidistribution or density results of special algebraic cycles (See Vatsal \cite{vatsal:imc}). The first prominent example goes back to Ferrero--Washington \cite{ferrero:washington} and Washington \cite{washington} for mod $p$ non-vanishing of special Dirichlet $L$-values. A key lemma used in their proof precisely comes from $p$-adic analogue of the classical density result due to Kronecker in ergodic theory. 
One of the main motivations of the present paper is to suggest a new dynamical approach towards the study of modular $L$-values with Dirichlet twists.

We can choose a suitable period $\Omega_f^\pm$ so that the corresponding algebraic parts $L_f(\chi)$ are $p$-integral with the minimum $p$-adic valuation when, for example, the mod $p$ Galois representation $\cl{\rho}_{f,p}$ is irreducible, $p$ does not divide $2N$, and $N\geq 3$ (See $\S$\ref{sec:opt:per}). In these circumstances, the $p$-integral $L$-values are expected to be generically non-vanishing modulo $p$. One also obtains the $p$-integrality of 
${\mG}_E^\pm$ when the residual Galois representation $\overline{\rho}_{E,p}$ of $E$ is irreducible; and  $E$ has good and ordinary reduction at $p$. 

For a Dirichlet character $\chi$ of modulus $M$, we define a variant of the special $L$-value by
$$\Lambda_f(\chi):=\sum_{a \in (\Zb/M\Zb)^\times} \overline{\chi}(a) \cdot \mG_f^\pm \left( \frac{a}{M} \right).$$
This $L$-value is closely related to the special $L$-value: They can differ by an Euler-like product over the prime divisors of conductor of $\chi$. In particular, when $\chi$ is primitive, one has $L_f(\chi)=\Lambda_f(\chi).$ 
We obtain a version of the mod $p$ non-vanishing result from our dynamical setup.

\begin{thmL}\label{non:van:thm:intro}
Let $N\geq 3$ and $p\nmid 2N$. Let $f$ be an elliptic newform for $\Gamma_0(N)$ such that $\cl{\rho}_{f,p}$ is irreducible. Then we have
\[ \#\bigcup_{n\leq M}\left\{\chi\in \widehat{(\Zb/n\Zb)^\times}\,\bigg| \, \Lambda_f(\chi)\not\equiv0 \ \left(\mathrm{mod} \ {\frak p}^{1+v_{\mathfrak{p}}(\phi(n))}\right) \right\} \gg M\]
where $\frak p$ is a prime over $p$ in $\overline{\Qb}_p$ and $v_{\mathfrak{p}}(\phi(n))$ is the $\mathfrak{p}$-adic valuation of the Euler totient $\phi(n)$.
\end{thmL}

A similar quantitative mod $p$ non-vanishing of Dirichlet $L$-values is studied in Burungale--Sun \cite{burungale:sun}: Let $\lambda$ be a Dirichlet character of modulus $N$ and $(p, NM)=1$ with $(N,M)=1$. Removing the condition $p\nmid \phi(M)$, their result can be formulated as follows.
$$\#\{\chi\in \widehat{(\Zb/M\Zb)^\times}\,|\, L(0,\lambda\chi)\not\equiv0\modulo{\mathfrak{p}^{1+v_{\mathfrak{p}}(\phi(M))}}\}\gg M^{1/2-\epsilon}.$$
Let us remark that even though Theorem \ref{non:van:thm:intro} is not strong enough as the result of Burungale--Sun, it is the first result of this type for modular $L$-values with Dirichlet twists as far as we know. In fact, this non-vanishing result is a consequence of one of our main results on another Mazur--Rubin conjecture \cite{mazur} as follows.

\begin{conjL}[Mazur--Rubin]
Assume that $\overline{\rho}_{E,p}$ is irreducible and $E$ has good and ordinary reduction at $p$. Then, for any integer $a$ modulo $p$
\[ \lim_{M \rightarrow \infty} \Pbb[\, {\mG}_E^\pm \equiv a \ (\mathrm{mod} \ p)|\Sigma_M ]=\frac{1}{p} . \]
\end{conjL}

Our result on the residual equidistribution of modular symbols is:

\begin{thmL}\label{res:equid:mod}
Assume that $\overline{\rho}_{E,p}$ is irreducible and $E$ has good and ordinary reduction at $p$. Then for any $e\geq 1$ and any integer $a$ modulo $p$,
\[ \Pbb[ {\mG}_E^\pm \equiv a \ (\mathrm{mod} \ p^e)|\Omega_{M,J} ]=\frac{1}{p^e} + O(M^{-\delta})  \]
for some $\delta>0$.
\end{thmL}

This is a specialisation of Theorem \ref{intro:equid:mod:q}. More precisely, there are integer valued functions $\zeta_{E}^\pm$ on $\Gamma_1(N)\backslash\SL_2(\Zb)$ such that their reduction modulo $p^e$ are not coboundaries over $\Zb/p^e\Zb$ and $\mG_E^\pm\modulo{p^e}$ follow the distributions of $\bG_{\zeta_{E}^+}$ and $\cG_{\zeta_E^-}$ (See \S\ref{subsec:res:dist:mod}). 

\begin{rem}
Constantinescu--Nordentoft \cite{const:nord} obtained a discrete version of Petridis--Risager \cite{petridis:risager}, of which consequences include Theorem \ref{res:equid:mod}.
\end{rem}

Mazur, in a private communication, raised a question whether the Gaussian (or archimedean) and residual distributions of the modular symbols are correlated or not. 
We answer the question in Theorem \ref{mod:symb:correl} which is a consequence of more general discussion in $\S$\ref{sec:noncorr}.

\subsection{Dynamics of continued fractions: work of Baladi--Vall\'ee}

We now describe our approach. It is deeply motivated by the work of Baladi--Vall\'ee \cite{bv} on dynamics of continued fraction. Let us briefly outline their result and strategy for the proof.

Baladi--Vall\'ee established the quasi-power behavior of a moment generating function $\Eb[\exp(w\ell )|\Omega_M]$, which ensures the asymptotic Gaussian distribution of $\ell$ (See Theorem \ref{hk:hwang}). More precisely, they studied a Dirichlet series whose coefficients are essentially given by the moment generating function $\Eb[\exp(w\ell)|\Sigma_n]$:
\[ L(s,w)=\sum_{n \geq 1} \frac{c_n(w)}{n^s} , \ c_n(w)=\sum_{r \in \Sigma_n} \exp(w \ell(r))  \] 
for two complex variables $s,w$ with $\Re s>1$ and $|w|$ being sufficiently small.  The desired estimate then follows from the Tauberian argument on $L(s,w)$. To this end, they settled the analytic properties of the poles of the Dirichlet series $L(s,w)$ and uniform estimates on its growth in a vertical strip. Their crucial observation is that the weighted transfer operator plays a central role in settling the necessary properties of $L(s,w)$. 

Let $T:{I} \rightarrow {I}$ denote the Gauss map which is given by 
$T(x)=\frac{1}{x}-\left\lfloor \frac{1}{x} \right\rfloor$
for $x\neq 0$ and $T(0)=0$.
They considered the weighted transfer operator associated with the Gauss dynamical system $({I}, T)$, defined by 
$$H_{s,w}f(x):=\sum_{y:T(y)=x}  \frac{\exp(w)}{|T'(y)|^s} \cdot f(y)$$ 
for two complex variables $s$ and $w$. 
A key relation they established is that
\begin{equation} \label{keyrel:bv}
L(2s,w)=F_{s,w} (Id-H_{s,w})^{-1}1(0)
\end{equation} 
where $Id$ is the identity operator and $F_{s,w}$ is the final operator defined as $H_{s,w}$ whose summation is restricted to the indices $y(x)=\frac{1}{m+x}$ with $m \geq 2$. The crucial properties of Dirichlet series for the Tauberian argument thus directly follow from the spectral properties of transfer operator. In particular, the estimate on the growth of $L(s,w)$ in a vertical strip comes from the Dolgopyat--Baladi--Vall\'ee estimate on an operator norm of $H_{s,w}^n$, $n\geq 1$.

Our idea is to follow their framework by finding a certain dynamical system and corresponding transfer operator that naturally describe the analytic properties of Dirichlet series associated to the modular partition functions.

\begin{rem}\label{k12:N1}
Bettin--Drappeau \cite{bettin:drappeau} generalised the work of Baladi--Vall\'ee in a different direction and obtained distributional results on crucial examples of quantum modular forms. 
 \end{rem}

\subsection{Dynamical System for modular partitions}\label{sect:dyn:mod:par}

Let us describe the dynamics and transfer operators for modular partition functions. 

Let $\varphi$ be a function on the right cosets of $\Gamma$ and $J$ a non-empty open sub-interval of $I$.  To study the moment generating function of $\mathfrak{g}_{\bm \psi}$  on $\Omega_{M,\varphi,J}$, for $\wbb\in\Cb^d$ we set 
$c_{n}(\wbb):=\sum_{r \in \Sigma_{n}\cap J} \varphi(r)\exp(\wbb\cdot\mathfrak{g}_{\bm \psi}(r))$. 
Obviously,
$\Eb[\exp(\wbb\cdot\mathfrak{g}_{\bm\psi})|\Omega_{M,\varphi,J}]=\frac{\sum_{n\leq M}c_n(\wbb)}{\sum_{n\leq M}c_n({\bf 0})}.$
In order to study $c_n(\wbb)$, we consider the generating function, namely a Dirichlet series: For $s \in \mathbb{C}$, set
$L^\mathfrak{g}(s,\wbb):=\sum_{n \geq 1} \frac{c_n(\wbb)}{n^s}$. 
A strategy is to apply Tauberian arguments to $L^\mathfrak{g}(s,\wbb)$ with their behaviors in a critical strip of $\Cb$, which are expected to be consequences of dynamical analysis of the modular partition functions $\mathfrak{g}_{\bm \psi}$. For  $\psi=\wbb\cdot{\bm\psi}$, we get $\wbb\cdot\mathfrak{g}_{\bm\psi}=\mathfrak{g}_\psi$. Hence, for the dynamical analysis, we consider transfer operators with parameter $\psi$ instead of $\wbb$.

\subsubsection{Random variable $\aG$}

Let us define an operator $\Tb$ on ${I} \times \Gamma \backslash \GL_2(\Zb)$ by 
\[ \Tb(x,v) := \left( T(x), v\twobytwo{-m_1(x)}{1}{1}{0} \right)
\] 
where $m_1(x)$ denotes the first digit of continued fraction expansion of $x$. We call $\Tb$ the \emph{skewed Gauss map}. 

Let $\Psi$ be a bounded function on ${I} \times \Gamma \backslash \GL_2(\Zb)$. For $s \in \Cb$ and $\psi:\Gamma\backslash\GL_2(\Zb)\rightarrow\Cb$, we consider a weighted {\it transfer operator} associated to the dynamical system $({I} \times \Gamma \backslash \GL_2(\Zb) , \Tb)$ defined by 
\[ \Lc_{s,\psi} \Psi(x,u):= \sum_{(y,v) \in \Tb^{-1}(x,u)} \frac{\exp\left[\psi(v)\right]}{\left| T'(y)\right|^{s}} \cdot \Psi(y,v).       \]
Let $\Fc_{s,\psi}$ be the {\it final operator} defined as $\Lc_{s,\psi}$ whose summation is restricted to the indices  $(y,v)=(\frac{1}{m+x},u\twobytwotiny{0}{1}{1}{m})$ with $m \geq 2$. To study the space $\Omega_{M,\varphi,J}$, we also introduce interval operators $\mathcal{D}_{s,\psi}^J$ (See \S\ref{subsec:int:op}).
Our crucial observation is that the Dirichlet series for $\aG$ admits an alternative expression in terms of the weighted transfer operators (See Theorem \ref{keyrelation}): The quasi-inverse $(\mathcal{I}-\mathcal{L}_{s,\psi})^{-1}$ is well-defined when $(\Re(s),\Re(\wbb))$ is close to $(1,{\bf 0})$ (See Theorem \ref{bound:dolgopyat}). Then, for an interval $J\subset I$, we have
\begin{align} \label{eq:keyrel}
L^{\aG}(2s,\wbb)&= \mathcal{B}_{s,\psi}^J(1\tensor\varphi)(0,\Gamma)+\mathcal{D}_{s,\psi}^J(\mathcal{I}-\mathcal{L}_{s,\psi})^{-1} \Fc_{s,\psi}(1\tensor\varphi)(0,\Gamma)
\end{align}
for $\psi=\wbb\cdot{\bm \psi}$ and an auxiliary analytic operator $\mathcal{B}^J_{s,\psi}$.

\begin{rem}
When $J=(0,1)$, $\Gamma=\SL_2(\Zb)$, and $\varphi=1$, the expression (\ref{keyrel:bv}) can be recovered from the above expression of $L^\aG(s,\wbb)$. 
\end{rem}

\subsubsection{Random variable $\bG$}\label{intro:subsec:bG}

First of all, note that there is a natural right action of $\GL_2(\Zb)$ on $\Gamma\backslash\SL_2(\Zb)$ given by
$$(\Gamma h)\cdot g:=\Gamma\,\widehat{hg}.$$ With this right action, we consider the spaces
\begin{align*}
	I_\Gamma := {I}\times \Gamma\backslash \SL_2(\Zb)
\end{align*}
Let us define the {\it skewed Gauss map} $\widehat{\Tb}$ on $I_\Gamma$ by 
\[ 
\widehat{\Tb}(x,v):= \left( T(x), v\cdot\twobytwo{-m_1(x)}{1}{1}{0} \right).
\]

 
 Similarly as $\aG$, for a function $\varphi$ on $\Gamma\backslash\SL_2(\Zb)$ we define a weighted {\it transfer operator} $\widehat{\Lc}_{s,\varphi}$ associated to the dynamical system $(I_\Gamma, \widehat{\Tb})$,
the {\it final operator} $\widehat{\Fc}_{s,\varphi}$, and the {\it interval operator} $\widehat{\mathcal{D}}_{s,\varphi}^J$ (See \S\ref{subsec:int:op}). We are also able to obtain a version of the statement (\ref{eq:keyrel}), i.e., an analogous expression for $L^\bG(s,\wbb)$ in terms of $\widehat{\Lc}_{s,\wbb\cdot{\bm \psi}}$, $\widehat{\Fc}_{s,\wbb\cdot{\bm \psi}}$, $\widehat{\mathcal{D}}_{s,\wbb\cdot{\bm \psi}}^J$, and $\widehat{\mathcal{B}}_{s,\wbb\cdot{\bm \psi}}^J$  (See Theorem \ref{keyrelation}), which partly is a consequence of the existence of the right action.


\begin{rem}
	This type of the skew-product Gauss map has already been studied in Manin--Marcolli \cite{manin:marc} in a different context, to study the Gauss--Kuzmin operator and limiting behavior of modular symbols. We refer to Remark \ref{man:mar:error}.
\end{rem}

\subsubsection{Random variable $\cG$} 
Unlike the map $g\mapsto \widehat{g}$, the map $g\mapsto \widetilde{g}$ does not induce a right action of $\GL_2(\Zb)$ on $\Gamma\backslash\SL_2(\Zb)$.
Even though one can easily define $\cG$-analogues of $\widehat{\Tb}$ and $\widehat{\Lc}_{s,\wbb}$, say $\widetilde{\Tb}$ and  $\widetilde{\Lc}_{s,\wbb}$,  the Dirichlet series ${L}^\cG({s,\wbb})$ no longer admit a similar expression to (\ref{eq:keyrel}), especially in terms of $\widetilde{\Lc}_{s,\wbb}$, mainly  due to the absence of a suitable right action of $\GL_2(\Zb)$.

Instead, we first observe that the maps $\widehat{\Tb}^2$ and $\widetilde{\Tb}^2$ are all the same as $\Tb^2|_{I_\Gamma}$, whose second component is now the canonical right action of $\SL_2(\Zb)$ on $\Gamma\backslash\SL_2(\Zb)$. Then one can define another weighted transfer operator $\mathcal{M}_{s,\varphi}$ associated to the system $(I_\Gamma, \widetilde{\Tb}^2)$.  After defining analogues of the previous operators, namely final operator $\widetilde{\mathcal{F}}_{s,\varphi}$, interval operator $\widetilde{\mathcal{D}}_{s,\varphi}^J$, and auxiliary operator $\widetilde{\mathcal{B}}_{s,\varphi}^J$, we are able to express the Dirichlet series $L^\cG(s,\wbb)$ in terms of those operators as before (See Theorem \ref{keyrelation}).

\begin{rem}\label{rem:cob}
	Note that a function $\psi:\Gamma\backslash\SL_2(\Zb)\rightarrow \Bbbk$ is a $\bG$-coboundary over $\Bbbk$ if and only if there exists a $\Bbbk$-valued function $\beta$ on $\Gamma\backslash\SL_2(\Zb)$ such that $\psi=\beta-\beta\circ\pi_2\widehat{\Tb}$. And that $\psi$ is a $\cG$-coboundary if and only if $\psi+\psi\circ \pi_2\widetilde{\Tb}=\beta-\beta\circ\pi_2\Tb^2$ for some $\beta$.
\end{rem}

\subsection{Spectral analysis of transfer operators}

For variable $(s,\psi)$ whose real part $(\Re s, \Re \psi)$ is close to $(1,\bf 0)$, the transfer operators for $\bG$ and $\cG$ act boundedly on ${C^1}(I_\Gamma)$ and admit a spectral gap with the dominant eigenvalues $\lambda_{s,\psi}$. Then by analogues of the identity (\ref{eq:keyrel}), the poles $s$ of Dirichlet series with a fixed $\psi$ in a certain vertical strip are in a bijection with the values $s$ with $\lambda_{s,\psi}=1$. Hence the necessary analytic properties (Proposition \ref{main:dynamics}) of $L^{\mathfrak{g}}(s,\wbb)$ to apply	 the Tauberian theorem follow from the spectral properties of the transfer operators: 
For general $\Gamma$, the dominant eigenvalue of the transfer operator is simple. The topological mixing property of $\widehat{\Tb}$ ensures the uniqueness of the eigenvalue. See \S\ref{skew:system} for more details.


\subsubsection*{Structure of paper} In $\S$\ref{subsec:gp:th}, we collect several group theoretic results relevant to topological mixing of $\widehat{\Tb}$ and coboundary condition of modular partition functions. In the sections $\S$\ref{sec:mod:part} and $\S$\ref{sec:dist:mod}, a series of number theoretic results on the distribution of modular partition functions are deduced by the Tauberian argument from the behaviors of the Dirichlet series. Their proofs will be presented in the last section $\S$\ref{pf:prop:2.2}. Two transitional sections $\S$\ref{sec:transition} and $\S$\ref{sect:tr:op} are devoted to transform the number theoretic assertions to dynamical ones.  In the sections $\S$\ref{sect:spec:transf}, $\S$\ref{sect:dolgopyat}, and $\S$\ref{sec:pressure}, dynamical analyses of the corresponding transfer operators are presented.  


\subsubsection*{Acknowledgements}

We are grateful to Val\'erie Berth\'e for careful comments and patiently answering many questions regarding dynamical analysis. We thank Ashay Burungale, Seonhee Lim and Brigitte Vall\'ee for instructive discussions about the topic. We thank Viviane Baladi for several precise comments and encouragement. We are indebted to Sary Drappeau for the proof of Lemma \ref{gen:mat:lem} and Fran\c{c}ois Ledrappier, Asbj\o rn Nordentoft for pointing out errors in the earlier version. Finally we are grateful to anonymous referees for helpful suggestions.

\section{$\GL_2(\Zb)$-action on $\Gamma\backslash\SL_2(\Zb)$}\label{subsec:gp:th}

Throughout, we fix a subgroup $\Gamma$ of $\SL_2(\Zb)$ of finite index.

\subsection{Right action of $\GL_2(\Zb)$}\label{def:action:sec}

Let us set ${\rm J}:=\cha{{\rm j}}$ and ${\rm G}\Gamma:=\cha{\Gamma,{\rm j}}$ as the subgroup of $\GL_2(\Zb)$ generated by $\Gamma$ and ${\rm j}$. Both the right cosets ${\rm G}\Gamma\backslash \GL_2(\Zb)$ and the double cosets $\Gamma \backslash \GL_2(\Zb)/{\rm J}$ are identified with $\Gamma\backslash \SL_2(\Zb)$ by the maps 
\begin{align}
\Gamma\backslash\SL_2(\Zb)&\simeq {\rm G}\Gamma\backslash \GL_2(\Zb),\,u=\Gamma h\mapsto {\rm G}{u}:={\rm G}\Gamma h,\label{GG:idf}\\
\Gamma\backslash\SL_2(\Zb)&\simeq\Gamma \backslash \GL_2(\Zb)/{\rm J},\,v=\Gamma h\mapsto {v}{\rm J}:=\Gamma{h}{\rm J}\label{GJ:idf}.
\end{align}
The right action of $\GL_2(\Zb)$ on $\Gamma\backslash\SL_2(\Zb)$ discussed in \S\ref{intro:subsec:bG} actually comes from 
the natural action on ${\rm G}\Gamma\backslash\GL_2(\Zb)$ via the identification (\ref{GG:idf}).
\begin{rem}\label{rem:right:action}
On the other hand, there is no right action of $\GL_2(\Zb)$ originates from (\ref{GJ:idf}). We still can observe that the map $\Gamma h {\rm J}\mapsto \Gamma hg {\rm J}$ is a permutation on $\Gamma\backslash\GL_2(\Zb)/{\rm J}$ for a $g\in\GL_2(\Zb)$, in other words, the map $t_g:\Gamma h\mapsto \Gamma\widetilde{hg}$
is a permutation on $\Gamma\backslash\SL_2(\Zb)$. Further, for $u\in \Gamma\backslash\SL_2(\Zb)$, it can be observed that
\begin{align}\label{tilde:g12:right}
t_{g_2}(t_{g_1}(u))=u g_1g_2\mbox{ if }g_1\in\SL_2(\Zb).
\end{align}
\end{rem}

Let $u$ be a right coset of $\Gamma$ in $\SL_2(\Zb)$ and $g\in \GL_2(\Zb)$. It is easy to see that
$\widehat{g}\in u$ if and only if ${\rm G}\Gamma g={\rm G}u$ and that $\widetilde{g}\in u$ if and only $\Gamma g{\rm J}=u{\rm J}$. We extend a function $\psi$ on $\Gamma\backslash\SL_2(\Zb)$ to $\Gamma\backslash\GL_2(\Zb)$ such that
$\widehat{\psi}:=\psi\circ {\rm G}^{-1}\circ p_1$ and $\widetilde{\psi}:=\psi\circ {\rm J}^{-1}\circ p_2$
where $p_i$ are the canonical surjections $p_1:\Gamma\backslash\GL_2(\Zb)\rightarrow {\rm G}\Gamma\backslash\GL_2(\Zb)$ and $p_2:\Gamma\backslash\GL_2(\Zb)\rightarrow \Gamma\backslash\GL_2(\Zb)/{\rm J}$. It is easy to see 
$$\widehat{\psi}(\Gamma g)=\psi(\Gamma\widehat{g})\mbox{ and }\widetilde{\psi}(\Gamma g)=\psi(\Gamma\widetilde{g}).$$
Hence, from the definition of $\bG_\psi$ and $\cG_\psi$, one obtains that
\begin{align}\label{2nd:repn:bu}
\bG_\psi(r)=\sum_{i=1}^\ell \widehat{\psi}(\Gamma g_i(r)) \mbox{ and }\cG_\psi(r)&=\sum_{i=1}^\ell \widetilde{\psi}(\Gamma g_i(r)).
\end{align}

Let us first prove several preliminary results on the special linear group.

\subsection{$T$-mixing}\label{subsec:tadmiss}

In this section, a matrix of the form $\twobytwotiny{-m}{\pm 1}{1}{0}$ is called a {\it digit matrix}. The following lemma and proposition are useful when we discuss the topological properties of $\widehat{\Tb}$.

\begin{lem}\label{gen:mat:lem}
Let $\epsilon=\pm 1$ be fixed. Then:
	\begin{enumerate}[leftmargin=*]
		\item \label{item:gen:digit:mat} For any $v\in\Gamma\backslash\SL_2(\Zb)$, we have
		\begin{align}\label{gen:digit:mat}
			\Gamma\backslash\SL_2(\Zb)=\left\{v\cdot\twobytwo{-m_1}{\epsilon}{1}{0}\twobytwo{-m_2}{\epsilon}{1}{0}\cdots \twobytwo{-m_\ell}{\epsilon}{1}{0}\,\bigg|\,\ell\geq 0,m_i\in\Zb_{\geq1}\right\}
		\end{align}
		where the element for $\ell=0$ corresponds to $v$.
		\item \label{gen:id:mat} There exists $K\geq 1$ such that for each integer $k\geq K$, we can find integers $m_1,\cdots,m_k\geq 1$ such that
		$$\Gamma\cdot\twobytwo{-m_1}{\epsilon}{1}{0}\twobytwo{-m_2}{\epsilon}{1}{0}\cdots \twobytwo{-m_k}{\epsilon}{1}{0}=\Gamma.$$
	\end{enumerate}
\end{lem}

\begin{proof}
Consider first the case of $\epsilon=1$.	Let us denote the R.H.S. of (\ref{gen:digit:mat}) by $S$. Let $a=\twobytwotiny{-1}{1}{1}{0}$ and $b=\twobytwotiny{-2}{1}{1}{0}$. Let $u\in S$. As $\Gamma$ is of finite index, there exist integers $p,q\geq 1$ such that for all $u$, we get $u\cdot a^p=u$ and $u\cdot b^q=u$ 
and hence, $u\cdot a^{-1}, u\cdot b^{-1}\in S$. In sum, we conclude that for any $g\in\GL_2(\Zb)$  generated by $a$ and $b$, we have $u\cdot g\in S$. On the other hand, observe that $ab^{-1}=\twobytwotiny{1}{1}{0}{1}$, $a^{-1}b=\twobytwotiny{1}{0}{1}{1}$, $ab^{-1}a=\twobytwotiny{0}{1}{1}{0}$.
	It is well-known that $\GL_2(\Zb)$ is generated by these three elements, hence by $a$ and $b$. Since $\Gamma\backslash\SL_2(\Zb)=\Gamma\cdot\GL_2(\Zb)$, we obtain the statement (\ref{item:gen:digit:mat}).

For the second statement, observe that
$ab^{-1}a^2b^{-1}a=\twobytwotiny{1}{0}{0}{1}$  and $b^{-1}a^3b^{-1}={\rm j}.$
When applied to a coset, the first product above can be regarded as the product of $2(q-1)+4$ number of digit matrices. The second one is the product of $2(q-1)+3$ number of digit matrices. Then, there exists a number $K$ such that any integer $k\geq K$ can be written as $k=(2(q-1)+4)s+(2(q-1)+3)t$ with $s,t\geq 1$. Then, we have
 $\Gamma\cdot (ab^{-1}a^2b^{-1}a)^s (b^{-1}a^3b^{-1})^t=\Gamma\cdot {\rm j}^t=\Gamma$ 
since $\Gamma\cdot {\rm j}=\Gamma$. 

For the case of $\epsilon=-1$, set $c=\twobytwotiny{-Q}{-1}{1}{0}$  where $Q$ is an integer $>0$ such that $u\twobytwotiny{1}{-Q}{0}{1}=u$ for all $u$. Then, $uc=u\twobytwotiny{0}{-1}{1}{0}$ for each $u$.
We also set $d=\twobytwotiny{-1}{-1}{1}{0}$, and $e=\twobytwotiny{-2}{-1}{1}{0}$. Then, $de^{-1}=\twobytwotiny{1}{1}{0}{1}$. Since $\SL_2(\Zb)$ is generated by $\twobytwotiny{0}{-1}{1}{0}$ and $\twobytwotiny{1}{1}{0}{1}$, we obtain the first statement by a similar argument as above. Note that $u=uc^4$ and $u=ud^3$ for all $u$. As before, we obtain the second statement.
\end{proof} 

\begin{rem}
A version of the statement (\ref{item:gen:digit:mat}) also can be found in Manin--Marcolli \cite[Theorem 0.2.1]{manin:marc}.
\end{rem}


\begin{prop}\label{prop:cong:Tadm}
Fix $\epsilon=\pm1$. There exists an $M>0$ such that for any $u\in \Gamma\backslash\SL_2(\Zb)$ and any $\ell\geq M$, one obtains
$$\Gamma\backslash\SL_2(\Zb)=\left\{u\cdot\twobytwo{-m_1}{\epsilon}{1}{0}\cdots\twobytwo{-m_\ell}{\epsilon}{1}{0}\,\Big|\,m_i\in\Zb_{\geq 1}\right\}.$$	
\end{prop}

\begin{proof}
Let $\Gamma_0$ be the kernel of the homomorphism from $\SL_2(\Zb)$ to the permutation group on $\Gamma\backslash\SL_2(\Zb)$ induced from the right action. Then $\Gamma_0$ is normal, of finite index, and is normalised by ${\rm j}$. 
Since the statement for $\Gamma$ follows from one for $\Gamma_0$, we may assume that $\Gamma$ is normal. 

First fix representations of $\Gamma\backslash\SL_2(\Zb)$ in (\ref{gen:digit:mat}), i.e., product-representations by digit matrices; For a right coset $u$, we can find a product ${\rm m}(u)$ of digit matrices such that $u=\Gamma\cdot {\rm m}(u)$. Let $\ell(u)$ be the number of the digit matrices that consists of ${\rm m}(u)$ and $L:=\max_u \ell(u)$. We claim that for any $n\geq L+K$ and any two right cosets $u,v$ of $\Gamma$, there are $n$ digit matrices such that their product, say ${\rm m}_{uv}$, satisfies $u=v\cdot {\rm m}_{uv}$.

	First of all, we show the claim for $v=\Gamma$. Let ${\rm w}_k$ be the product of $k$ digit matrices in Lemma \ref{gen:mat:lem}.(\ref{gen:id:mat}). For any $n\geq L+K$, let us set ${\rm m}_n[u]:={\rm w}_{n-\ell(u)}{\rm m}(u)$. Observe that $u=\Gamma\cdot {\rm m}_n[u] $ and ${\rm m}_n[u]$ is the product of $n$ digit matrices. 
	
	Let $v$ be a general right coset. Let us set $u=\Gamma g$ and $v=\Gamma h$ for $g,h\in \SL_2(\Zb)$. Let $n\geq L+K$.  Then observe that $$
	\Gamma g\cdot {\rm m}_n[\Gamma g^{-1}h]=g\Gamma\cdot {\rm m}_n[\Gamma g^{-1}h]=g\Gamma g^{-1}h=\Gamma h$$ 
	as $\Gamma$ is normal in $\SL_2(\Zb)$. 	
\end{proof}

\subsection{Coboundary functions}

Let $\Bbbk$ be an abelian group. In this section, we characterize all the coboundary functions over $\Bbbk$. We fix $\beta:\Gamma\backslash\SL_2(\Zb)\rightarrow\Bbbk$ corresponding to a coboundary $\psi$, i.e., $\psi(u)=\beta(u)-\beta(u\cdot\twobytwotiny{-m}{1}{1}{0})$ if $\mathfrak{g}=\bG$ and $\psi(u)+\psi(u\twobytwotiny{-m}{1}{1}{0}{\rm j})=\beta(u)-\beta(u\twobytwotiny{-m}{1}{1}{0}\twobytwotiny{-n}{1}{1}{0})$ if $\mathfrak{g}=\cG$.
Let us denote 
$${\rm L}:=\twobytwo{1}{0}{\Zb}{1}
.$$
Note that since $\Gamma\backslash\SL_2(\Zb)$ is a finite set, the natural right action of ${\rm L}$ on the cosets $\Gamma\backslash\SL_2(\Zb)$ factors through $\twobytwotiny{1}{0}{\Zb}{1}\rightarrow \twobytwotiny{1}{0}{\Zb/Q\Zb}{1}$ for an integer $Q>1$.

We collect several properties that $\beta$ satisfies:

\begin{prop}\label{misc:beta:prop}
Let $\psi:\Gamma\backslash\SL_2(\Zb)\rightarrow \Bbbk$ be a function. 
\begin{enumerate}[leftmargin=*]
\item If $\psi$ is a $\mathfrak{g}$-coboundary, then $\beta$ is ${\rm L}$-invariant. In particular, 
	\begin{align}\label{inv:beta}
		\beta(u\cdot\twobytwotiny{m}{1}{1}{0})=\beta(u\cdot\twobytwotiny{n}{1}{1}{0})\mbox{ for all $m,n\in\Zb$.}
	\end{align}	
\item\label{psl:invj} Let $\psi$ be a $\bG$-coboundary. If $\psi(u)=\psi(- u)$ and $\psi(u{\rm j})=\psi({\rm j}u)$ for all $u$, then $\beta(u{\rm j})=\beta({\rm j}u)$ for all $u$.
\end{enumerate}

\end{prop}

\begin{proof}
We first consider (1). Consider $\mathfrak{g}=\bG$.  
We can say that $\psi(u)=\beta(u)-\beta(u\cdot\twobytwotiny{-m}{1}{1}{0})$ for all $m\in\Zb$ and $u$. Observe also that $\twobytwotiny{\Zb}{1}{1}{0}=\twobytwotiny{0}{1}{1}{0}\twobytwotiny{1}{0}{\Zb}{1}$. Then the coboundary condition on $\psi$ implies that $\beta$ is invariant under $\rm L$.	

Let $\psi$ be a $\cG$-coboundary. Similarly as above, for all $m,n\in\Zb$, we get $\psi(u)+\psi\big(u \twobytwotiny{-m}{1}{1}{0}{\rm j}\Big)=\beta(u)-\beta\big(u \twobytwotiny{-m}{1}{1}{0}\twobytwotiny{-n}{1}{1}{0}\big)$. Setting $m=0$, we get $\beta(u\twobytwotiny{1}{0}{-n}{1})=\beta(u)-\psi(u)-\psi(u\iota)$. Hence, $\beta$ is ${\rm L}$-invariant. 

Now we show (2). Setting $\alpha(u):=\beta(u)-\beta(-u)$, we get $\alpha(u)=\alpha(u\cdot\twobytwotiny{-m}{1}{1}{0})$ for all $u$ and $m$. By Proposition \ref{prop:cong:Tadm}, we know $\alpha$ is constant, in particular, $\alpha(u)=\alpha(-u)$. But $\alpha(-u)=-\alpha(u)$. Hence, $\alpha=\bf 0$, i.e., 
\begin{align}\label{beta:-beta}
\beta(u)=\beta(-u)\mbox{ for all $u$}.
\end{align} 
The given expression for $\psi$ can be written as 
$\beta(u{\rm j})-\beta({\rm j}u)=\beta(u{\rm j}\cdot\twobytwotiny{-m}{1}{1}{0})-\beta({\rm j}u\cdot\twobytwotiny{-m}{1}{1}{0}).$
Using ${\rm j}\twobytwotiny{m}{1}{1}{0}{\rm j}=-\twobytwotiny{-m}{1}{1}{0}$  and (\ref{beta:-beta}) with (\ref{inv:beta}), the last expression equals to
$\beta(u{\rm j})-\beta({\rm j}u)=\beta(u\cdot\twobytwotiny{-m}{1}{1}{0}{\rm j})-\beta({\rm j}u\cdot\twobytwotiny{-m}{1}{1}{0}).$
From this we conclude that $u\mapsto\beta(u{\rm j})-\beta({\rm j}u)$ is constant. Considering ${\rm j}u{\rm j}$, we obtain the statement.
\end{proof}

The following is crucial for determining $\mathcal{B}_\mathfrak{g}(\Gamma,\Bbbk)$.

\begin{prop}\label{trans:inv:cob}
Let  $\psi$ be a $\mathfrak{g}$-coboundary over $\Bbbk$ for an ${\rm L}$-invariant $\beta$. Then \item 	$\psi$ is zero if and only if $\beta$ is a constant. In this case, $\beta$ can be chosen as zero.

\end{prop}

\begin{proof}
First let $\mathfrak{g}=\bG$. If $\psi$ is zero, then by Proposition \ref{prop:cong:Tadm}, we can find $m_1$, $\cdots$, $m_\ell$ for a sufficiently large $\ell$ and $v$ such that $u\cdot\twobytwotiny{-m_1}{1}{1}{0}\cdots \twobytwotiny{-m_k}{1}{1}{0}=v$. Hence, $\beta$ is a constant. The converse is trivial.

Let $\mathfrak{g}=\cG$. Let $\beta$ be a constant function. Then, $\psi(u)+\psi\big(u \twobytwotiny{-m}{1}{1}{0}{\rm j}\Big)=0$ for all $m$, i.e., $\psi(u)=-\psi\big(u\twobytwotiny{-m}{-1}{1}{0}\Big)=0$ for all $m\geq 1$. By Proposition \ref{prop:cong:Tadm}, we can find $m_1$, $\cdots$, $m_\ell$ for a sufficiently large odd $\ell$ such that $u\twobytwotiny{-m_1}{-1}{1}{0}\cdots \twobytwotiny{-m_\ell}{-1}{1}{0}=u$. So, we get $\psi=0$. Conversely, suppose that $\psi$ is zero. Then, similarly as above, we can show that $\beta$ is a constant. 
\end{proof}

Now, the boundary functions are completely characterized. 

\begin{cor}\label{clssf:bndry}
\begin{enumerate}[leftmargin=*]
\item Let $U=\Gamma\backslash\SL_2(\Zb)/{\rm L}\setminus \{\Gamma\}$. There is an isomorphism  
\begin{align*}
	\mathcal{B}_\mathfrak{g}(\Gamma,\Bbbk)\simeq \Bbbk^{U}.
\end{align*}
\item \label{torsion:bndry} For an integer $Q>1$, the $Q$-torsion subgroup of $\Bc_\mathfrak{g}(\Gamma,\Rb/2\pi\Zb)$ is equal to $\Bc_\mathfrak{g}(\Gamma,2\pi Q^{-1}\Zb/2\pi\Zb)$, which is isomorphic to $\Bc_\mathfrak{g}(\Gamma,\Zb/Q\Zb)$ by the map $\psi\mapsto \frac{Q}{2\pi}\psi\modulo{Q}$.
\end{enumerate}
\end{cor}

\begin{proof}
The map $\psi\mapsto \beta|_U-\beta(\Gamma)$ is an isomorphism. For the second statement, observe that the $Q$-torsion subgroup of $(\Rb/2\pi\Zb)^U$ is just $(2\pi Q^{-1}\Zb/2\pi\Zb)^U$.
\end{proof}

\section{Modular partitions of continued fractions}\label{sec:mod:part}

In this section, we present the proofs of limit joint Gaussian distribution and the residual equidistribution of modular partition functions of $\bG$ and $\cG$ over $\Omega_{M,\varphi,J}$. 
We consider the same setting as \S\ref{sect:dyn:mod:par}. Throughout, we fix a non-empty open interval $J\subseteq I$ and a non-trivial function $\varphi:\Gamma\backslash\SL_2(\Zb)\rightarrow \Rb_{\geq 0}$ unless mentioned explicitly. Recall that ${\rm j}$ is assumed to normalise $\Gamma$. Let us set
$$\mathfrak{g}:=\bG\mbox{ or }\cG.$$ 
We define a map on $\Sigma_M$  given by 
$$r=[0;m_1,\cdots,m_\ell] \longmapsto r^*:=[0;m_\ell,\cdots,m_1].$$ For a function $\Psi$ on $I_\Gamma$, let us denote functions on $\Qb\cap{I}$
\begin{align*}
	\Psi^*: r\mapsto \Psi(r^*,\Gamma \widehat{g}(r))\mbox{ or }\Psi^*: r\mapsto \Psi(r^*,\Gamma \widetilde{g}(r))
\end{align*}
by the same symbol $\Psi^*$ according to the choice of $\mathfrak{g}$. Let us define the Dirichlet series associated to $\mathfrak{g}$ as
\begin{align*}
L_{\Psi,J}(s,\wbb)&:=\sum_{n \geq 1} \frac{d_n(\wbb)}{n^s}\mbox{ with }
d_n(\wbb)=\sum_{r \in \Sigma_{n}\cap J} \Psi^*(r) \exp(\wbb\cdot\mathfrak{g}_{\bm \psi}(r))
\end{align*} 
for $s \in \Cb$, $\wbb \in \Cb^d$, and ${\bm \psi}=(\psi_1,\cdots,\psi_d)$. 
The average of the coefficients $d_n(\wbb)$ can be studied using the following  truncated Perron's formula.

\begin{thm}[Perron's Formula, Titchmarsh {\cite[Lemma 3.12]{titchmarsh}}] \label{perron}
Suppose $a_n$ is a sequence and $A(x)$ is a non-decreasing function such that $|a_n|=O(A(n))$. Let $F(s)=\sum_{n \geq 1} \frac{a_n}{n^s}$ for $\sigma:=\Re s>\sigma_a$, the abscissa of absolute convergence of $F(s)$. Then for all $D>\sigma_a$ and $T>0$, one has
\begin{align*}
\sum_{n \leq x} a_n= \frac{1}{2\pi i} \int_{D-iT}^{D+iT} F(s) \frac{x^s}{s}ds &+ O\left(\frac{x^D |F|(D)}{T} \right) +O \left(\frac{A(2x)x\log x}{T} \right) \\ &+ O \left( A(N) \mathrm{min} \left\{ \frac{x}{T|x-N |},1 \right\} \right),
\end{align*}
where  $|F|(\sigma)=\sum_{n \geq 1} \frac{|a_n|}{n^\sigma}$ for $\sigma > \sigma_a$ and $N$ is the nearest integer to $x$. 
\end{thm}

In order to shift the contour, we use the following properties of Dirichlet series in the vertical strip. In the next, we say a $\Rb$-valued function is a $\mathfrak{g}$-coboundary over $\Rb/2\pi\Zb$ if its composition with the surjection $\Rb\rightarrow\Rb/2\pi\Zb$ is.

\begin{prop} \label{main:dynamics}
	Let ${\bf v}\in \Rb^{d}$. There exists $0<\alpha_1\leq\frac{1}{2}$ such that for any $\widehat{\alpha}_1$ with $0<\widehat{\alpha}_1<\alpha_1$, there exists a neighborhood $W$ of $i\mathbf{v}$ in $\Cb^d$ such that:  
\begin{enumerate}[leftmargin=*]
\item	If ${\bf v}\cdot{\bm \psi}\in \Bc_\mathfrak{g}(\Gamma,\Rb/2\pi\Zb)$, then $L_{\Psi,J}(2s,\wbb)$ has a unique simple pole at $s=s(\wbb)$ in the strip $|\Re s-1| \leq \alpha_1$ for each $\wbb\in W$ with the properties:
	\begin{enumerate}[leftmargin=*]
		\item $s(\wbb)$ is analytic in $W$  and $s(i{\bf v})=1$.\label{def:s:0}
		\item $\Re s(\wbb) > 1-(\alpha_1 -\widehat{\alpha}_1)$.\label{s:strip} 
		\item The Hessian of ${s}(\wbb)$ is non-singular at $\wbb=i\mathbf{v}$ if and only if $\psi_1$, $\cdots$, $\psi_d$ are $\Rb$-linearly independent modulo $\Bc_\mathfrak{g}(\Gamma,\Rb)$.\label{s:hessian}
		\item \label{residue:LPsi}The residue $E_{\bf v}(\wbb)$ at $s(\wbb)$ is analytic on $W$ with 
		$$E_{\bf v}(i{\bf v})=\frac{e^{i\beta(\Gamma)}6|J| }{\pi^2\log2}\int_{(0,\frac{1}{2})\times\Gamma\backslash\SL_2(\Zb)}e^{-i\beta}\Psi dm$$
		where $\beta$ is associated with  ${{\bf v}}\cdot{\bm \psi}$. Here $|J|$ is the length of $J$.\label{s:residue}
	\end{enumerate}
\item If ${\bf v}\cdot{\bm \psi}\not\in \Bc_\mathfrak{g}(\Gamma,\Rb/2\pi\Zb)$, then
 $L_{\Psi,J}(2s, {\wbb})$ is analytic in the strip $|\Re s-1| \leq \alpha_1$ for all ${\wbb}\in W$.\label{cob:iybb}
\item \label{growth:strip}
	For $0<\xi<\frac{1}{5}$, there exist $0<\alpha_0\leq\alpha_1$, $0<\rho<1$, and a neighborhood $B$ of $\mathbf{0}$ in $\Rb^d$ such that for any $\Psi\in {C^1}(I_{\Gamma})$ and for all $\wbb\in \Cb^d$ with $\Re(\wbb)\in B$, 
we have $$|L_{\Psi,J}(2s,\wbb)| \ll \max(1,|\Im s|^\xi)$$ 
when $|\Re s-1| \leq \alpha_0$ with $|\Im s| \geq \frac{1}{\rho^2}$ or $\Re(s)=1\pm{\alpha}_0$ with $|\Im s|\leq \frac{1}{\rho^2}$.
\end{enumerate}
\end{prop}

\begin{rem}
Proposition \ref{main:dynamics} is a unification of Lemma 8 and 9 in Baladi-Vall\'ee \cite{bv}, which correspond to the cases (1) $d=1$, ${\bf v}= 0$ and (2) $d=1$, $\wbb=i{\bf v}$ with ${\bf v}\neq {0}$, respectively. Please refer to Remark \ref{corr:gen:thm} and \ref{sp:case:lv:1}.
\end{rem}



We postpone the proof of Proposition \ref{main:dynamics} to the end of present paper after introducing the skewed Gauss map and the associated transfer operator on ${C^1}(I_{\Gamma})$; and settling an explicit relation between the resolvent of the operator and Dirichlet series associated to the modular partition functions in $\S$\ref{sec:transition}--$\S$\ref{sec:pressure}.

The following is one of our main results which leads to both the asymptotic Gaussian behavior and residual equidistribution of the variable ${\mathfrak{g}_{\bm\psi}}$. 


\begin{prop}\label{av:sum:dnw}
	Let ${\bf v}\in \Rb^d$. There exist a constant $0<\delta<2$ and a neighborhood $W$ of $i\bf v$ in $\Cb^d$ such that for  $\Psi\in{C^1}(I_\Gamma)$ and $\wbb\in W$, we have
	\begin{align}\label{av:coeff:dnw}
		\sum_{n \leq M}d_n(\wbb)=R_{M,{\bf v}}(\wbb)+O(M^\delta)
	\end{align}
	where 
	$$R_{M, \bf v}(\wbb):=\begin{cases} \frac{E_{\bf v}(\wbb)}{s(\wbb)}M^{2s(\wbb)}&\mbox{ if }{\bf v}\cdot {\bm \psi}\in \Bc_{\mathfrak{g}}(\Gamma,\Rb/2\pi\Zb)\\
	\phantom{blaank} 0&\mbox{ otherwise}
	\end{cases}.$$
	The implicit constant and $\delta$ are independent of $\wbb$.
\end{prop}

\begin{proof}

Proposition \ref{main:dynamics} enables us to do the contour integration using Cauchy's residue theorem 
	\[ \frac{1}{2\pi i} \int_{\mathcal{U}_T(\wbb)} L_{\Psi,J}(2s,\wbb) \frac{M^{2s}}{2s} d(2s)=R_{M,{\bf v}}(\wbb) \]
	where  $\mathcal{U}_T(\wbb)$ denotes the contour with positive orientation, which is simply a rectangle with vertices $1+\alpha_0+iT$, $1-\alpha_0+iT$, $1-\alpha_0-iT$, and $1+\alpha_0-iT$. 
	
Applying the Perron formula from Theorem \ref{perron} to $L_{\Psi,J}(2s,\wbb)$ for $s$ along the vertical line $1+\alpha_0 \pm iT$, we have  
	\begin{align*}
		\sum_{n \leq M} &d_n(\wbb)= R_{M,{\bf v}}(\wbb)+ O\left(\frac{M^{2(1+\al_0)}}{T} \right) +O(A(M))\\ &+ O \left(\frac{A(2M)M\log M}{T} \right) + O \left( \int_{1-\al_0-iT}^{1-\al_0+iT} |L_{\Psi,J}(2s,\wbb)| \frac{M^{2(1-\al_0)}}{|s|} ds    \right)\\
		& + O \left( \int_{1-\al_0 \pm iT}^{1+\al_0 \pm iT} |L_{\Psi,J}(2s,\wbb)| \frac{M^{2\Re s}}{T} ds  \right) .
	\end{align*} 
	Note that the last two error terms are derived from the contour integral and each of them corresponds to the left vertical line and horizontal lines of the rectangle $\mathcal{U}_T$ respectively. Let us write it as 
	\[ \sum_{n \leq M} d_n(\wbb)= R_{M,{\bf v}}(\wbb) +\mathrm{\rom{1}+ \rom{2}+\rom{3}+\rom{4}+\rom{5}}. \]
	We choose $\widehat{\al}_0$ with $\frac{8}{59} \al_0 < \widehat{\al}_0 < \al_0$ and set $$T= M^{2\al_0+4 \widehat{\al}_0}.$$ 
	Notice that $\frac{E_{\bf v}(\wbb)}{s(\wbb)}$ is bounded in the neighborhood $W$ since $s(i{\bf v})=1$. Then, the error terms are bounded as follows.
	
	The error term $\mathrm{\rom{1}}$ is equal to $O\big(M^{2(1-2\widehat{\al}_0)}\big)$ and by Proposition \ref{main:dynamics}, the exponent of $M$ satisfies $2(1-2\widehat{\al}_0)<2$.

Let us set $|\xbb|:=\max_i|x_i|$ for $\xbb=\Re(\wbb)$. Since $\xbb\cdot{\Ebb}(r)\ll |\xbb|\ell(r)$ and $\ell(r)\ll \log n$ for $r\in\Sigma_n$, for some $c>0$ we obtain
\begin{align} \label{perron:am}
	d_n(\wbb)\ll n^{1+c|\xbb|}.
\end{align}
By (\ref{perron:am}), for any $0<\varepsilon<\frac{\widehat{\al}_0}{2}$, we can take $W$ from Proposition \ref{main:dynamics} small enough to have $c |\xbb| < \varepsilon/2$ so that $A(M)=O(M^{1+\varepsilon/2})$ and $\log M \ll M^{\varepsilon/2}$. Hence, the exponent of $M$ in the error term $\mathrm{\rom{3}}$ is equal to 
	\[ 1+ (1+c|\xbb|) + \frac{\varepsilon}{2}-(2 \al_0+ 4 \widehat{\al}_0)  \leq 2- \frac{23}{4} \widehat{\al}_0 <2 .\]
	
	Similarly the error term $\mathrm{\rom{2}}$ is equal to $O(M^{1+\varepsilon/2})$, so the exponent satisfies \begin{align*} 
		1+\frac{\varepsilon}{2}< 1+\frac{1}{4} \widehat{\al}_0 <2.  \end{align*}
	
	Also for $0<\xi<\frac{1}{5}$, we have $|L_{\Psi,J}(2s,\wbb)| \ll |\Im s|^{\xi}$ by Proposition \ref{main:dynamics}. Hence, the error term $\mathrm{\rom{4}}$ is $O(M^{2(1-\al_0)} T^\xi)$ and the exponent of $M$ is equal to 
	\begin{align*} 
		2(1-\al_0)+ (2\al_0+4 \widehat{\al}_0)\xi <2- \frac{4}{5}(2\al_0-\widehat{\al}_0) <2 .\end{align*}
	
	The last term $\mathrm{\rom{5}}$ is $O(T^{\xi-1} \cdot M^{2(1+\al_0)} (\log M)^{-1})$, hence the exponent of $M$ satisfies 
	\begin{align*} & (2\al_0+4 \widehat{\al}_0)(\xi-1)+2(1+\al_0)-\frac{\varepsilon}{2} < 2-\left(- \frac{2}{5}\al_0+ \frac{59}{20} \widehat{\al}_0 \right)< 2.
	\end{align*}
	In total, setting 
	\[ \delta= \mathrm{max} \left(  2-\frac{23}{4} \widehat{\al}_0, 1+\frac{1}{4}\widehat{\al}_0, 2- \frac{4}{5}(2\al_0-\widehat{\al}_0), 2-\left(- \frac{2}{5}\al_0+ \frac{59}{20} \widehat{\al}_0 \right)   \right), \]
	we conclude the proof.
\end{proof}

\begin{rem}
	We would like to mention that we have used a version of Perron's formula for Proposition \ref{av:sum:dnw}, which is different from the one used in Baladi--Vall\'ee \cite{bv}. 
	The current version directly leads us to get the desired estimate for the moment generating function of smaller spaces, namely $\Sigma_M(\epsilon)$, than $\Omega_M$ without using the extra smoothing process of Baladi--Vall\'ee. See Lee--Sun \cite{lee:sun} for the relevant discussion concerning the length of continued fractions.
\end{rem}

Observe that by Proposition \ref{av:sum:dnw} with $\Psi\equiv1$ and ${\bf v}={\bf 0}$,  
\begin{align}\label{dn0:expression}
	|\Omega_{M,J}|=E_{\bf 0}({\bf 0})M^2+O(M^\delta)\mbox{ with }E_{\bf 0}({\bf 0})=\frac{3|J|}{\pi^2\log 2}.
\end{align}

\subsection{Joint Gaussian distribution: Proof of Theorem \ref{vector:far0}}\label{sec:joint:gauss}

In this subsection, we obtain an explicit quasi-power behavior for moment generating function of the modular partition functions and show the limit joint Gaussian distribution.

\begin{thm} \label{gaussian:vector:plus}
There exist a neighborhood $W$ of $\bf 0$, an analytic function $B_{\varphi,J}$ on $W$, and a constant $0<\gamma<\alpha_1$ with $\alpha_1$  from Proposition \ref{main:dynamics}, such that  $B_{\varphi,J}$ is non-vanishing on $W$ and 
 \[ \Eb[\exp(\wbb\rdot {\mathfrak{g}_{\bm\psi}})|\Omega_{M,\varphi,J}]=\frac{B_{\varphi,J}(\wbb)}{B_{\varphi,J}({\bf 0})}M^{2( s(\wbb)-s(\mathbf{0}))}(1+O(M^{-\gamma})) \]
with $s(\wbb)$ from Proposition \ref{main:dynamics} (1) with ${\bf v}={\bf 0}$ and $\Psi=1\tensor\varphi$. The implicit constant and the constant $\gamma$ are independent of $\wbb\in W$.
\end{thm}

\begin{proof}
Setting $B_{\varphi,J}(\wbb):=E_{\bf 0}(\wbb)/s(\wbb)$, from Proposition \ref{av:sum:dnw} with ${\bf v}={\bf 0}$ and $\Psi=1\tensor\varphi$, we obtain the proof of theorem.
\end{proof}

The following probabilistic result ensures that the asymptotic normality of a sequence of random vectors comes from the quasi-power behavior of their moment generating functions. 

\begin{thm}[Heuberger--Kropf \cite{hk}, Hwang \cite{bv}] \label{hk:hwang}
Suppose that the moment generating function for a sequence $\mathbf{X}_N$ of $m$-dimensional real random vectors on spaces $\Xi_N$ satisfies the quasi-power expression
\[ \Eb[\exp(\wbb\rdot \mathbf{X}_N)\,|\,\Xi_N]=\exp(\beta_N U(\wbb)+V(\wbb))(1+O(\kappa_N^{-1}))  \]
with $\beta_N, \kappa_N \rightarrow \infty$ as $N \rightarrow \infty$, and $U(\wbb), V(\wbb)$ analytic for $\wbb=(w_i) \in \Cb^m$ with $|\wbb|$ being sufficiently small. Assume that the Hessian $\mathbf{H}_U(\mathbf{0})$ of $U$ at $\bf 0$ is non-singular. Then:
\begin{enumerate}[leftmargin=*]
\item The distribution of $\mathbf{X}_N$ is asymptotically normal with the speed of convergence $O(\kappa_N^{-1}+\beta_N^{-1/2})$. In other words, for any $\xbb\in \Rb^m$
\begin{align*} 
\Pbb &\left[ \frac{{\bf X}_N- \nabla U(\mathbf{0}) \beta_N}{\sqrt{\beta_N}}  \leq \xbb\, \Big\lvert\, \Xi_N \right]\\
 &= \frac{1}{(2 \pi)^{m/2} \sqrt{\det\mathbf{H}_U(\mathbf{0})}} \int_{\mathbf{t} \leq \xbb} \exp \left(-\frac{1}{2} \mathbf{t}^T\mathbf{H}_U(\mathbf{0})^{-1}\mathbf{t}  \right) d \mathbf{t} + O \left(\frac{1}{\kappa_N}+\frac{1}{\sqrt{\beta_N}} \right)  
\end{align*}
where $\mathbf{t} \leq \xbb$ means $t_j \leq x_j$ for all $1 \leq j \leq k$ and the $O$-term is uniform in $\xbb$.
\item Let $m=1$. The moments of $X_N$ satisfy
 \begin{align*}
       \Eb[X_N\,|\,\Xi_N] &= \beta_N U'(0)+V'(0)+O(\kappa_N^{-1}), \\ \Vb[X_N\,|\,\Xi_N]  &= \beta_N U''(0)+V''(0)+O(\kappa_N^{-1})\\
\Eb[X_N^k\,|\,\Xi_N] &= P_k(\beta_N)+O(\beta_N^{k}\kappa_N^{-1})
\end{align*}    
for some polynomials $P_k$ of degree at most $k
\geq 3$.
\end{enumerate}
\end{thm}

We are ready to give:

\begin{proof}[Proof of Theorem \ref{vector:far0}]
Let $U(\wbb)=2(s(\wbb)-s(\mathbf{0}))$ and $V(\wbb)=\log \frac{B_{\varphi,J}(\wbb)}{B_{\varphi,J}(\mathbf{0})}$ with $s$ and $B_{\varphi,J}$ from Theorem \ref{gaussian:vector:plus}. By Proposition \ref{main:dynamics} with ${\bf v}={\bf 0}$, both $U$ and $V$ are independent of $M$, analytic for sufficiently small $\wbb$, the Hessian of $U$ at $\bf 0$ is equal to one of $s(\wbb)$, and it is non-singular if and only if $\psi_i$ are $\Rb$-linearly independent modulo $\Bc_{\mathfrak{g}}(\Gamma,\Rb)$. Setting ${\rm H}_{\bm \psi}$ as the Hessian of $U$, Theorem \ref{hk:hwang} enables us to finish the proof.
\end{proof}

\subsection{Residual equidistribution: Proof of Theorem \ref{intro:equid:mod:q}}\label{sec:res:equid}

In this subsection, we give a proof for the residual equidistribution of modular partition functions. First we need:

\begin{thm} \label{modp:equidist:vector}
Let ${\bf v}\in \Rb^d$ and ${\bm \psi}:\Gamma\backslash\SL_2(\Zb)\rightarrow \Rb^d$. There exists $\gamma_1>0$ such that  
\[ \Eb[\,\exp({i{\bf v}}\cdot {\mathfrak{g}_{\bm\psi}})|\Omega_{M,J}]=R({\bf v})+O( M^{-\gamma_1})\]
with  
$$R({\bf v})=\begin{cases}
\ds\frac{\sum_v\exp[i(\beta(\Gamma)-\beta(v))]}{[\SL_2(\Zb):\Gamma]}&\mbox{ if }{\bf v}\cdot{\bm \psi}\in \Bc_{\mathfrak{g}}(\Gamma,\Rb/2\pi\Zb)\\
\phantom{blaaaaank}0&\mbox{ otherwise}
\end{cases}.$$ 
Here in the first case, ${\bf v}\cdot{\bm \psi}$ is associated with $\beta$.
\end{thm}

\begin{proof}
Note that if ${\bf v}\cdot {\bm\psi}\in \Bc_{\mathfrak{g}}(\Gamma,\Rb/2\pi\Zb)$, then by Proposition \ref{main:dynamics} with $\Psi\equiv 1$, $$E_{\bf v}(i{\bf v})=\frac{3|J|\sum_v\exp[i(\beta(\Gamma)-\beta(v))]}{[\SL_2(\Zb):\Gamma]\pi^2\log 2}.$$
Setting $R({\bf v})=\frac{E_{\bf v}(i{\bf v})}{E_{\bf 0}({\bf 0})}$, by Proposition \ref{av:sum:dnw} and (\ref{dn0:expression}), we obtain the statement with $\gamma_1=2-\delta$. 
\end{proof}

We are ready to give:

\begin{proof}[Proof of Theorem \ref{intro:equid:mod:q}]

Recall that ${\mathfrak{g}_{\bm\psi}}(r)\in \Zb^d$ as $r$ varies over $\Omega_M$.  For ${\bf g}\in (\Zb/Q\Zb)^d$, it easy to see
\begin{align*}  
\Pbb&[ {\mathfrak{g}_{\bm\psi}}\equiv {\bf g}\modulo{Q}|\Omega_{M,J}] 
= \frac{1}{Q^d} \sum_{{\bf s}\in (\Zb/Q\Zb)^d} e^{- \frac{2 \pi i}{Q} {\bf s}\cdot {\bf g}} \cdot \Eb \left[\exp \left( \frac{2 \pi i}{Q} {\bf s}\rdot {\mathfrak{g}_{\bm\psi}} \right)\Big\lvert \Omega_{M,J} \right] .
\end{align*}
We then split the summation into two parts: ${\bf s}={\bf 0}$ and ${\bf s}\neq {\bf 0}$. The term corresponding to ${\bf s}={\bf 0}$ is the main term which is $Q^{-d}$. For the sum over ${\bf s\not=0}$, we assert that 
\begin{align}\label{not:cob:s:psi}
\frac{2\pi}{Q}{\bf s\cdot\bm\psi}\not\in \Bc_{\mathfrak{g}}(\Gamma,\Rb/2\pi\Zb).
\end{align}
Note that the condition is independent of any choice of a lift of ${\bf s}$ to $\Zb^d$.
Assume the contrary of (\ref{not:cob:s:psi}). Then, by Corollary \ref{clssf:bndry}.(\ref{torsion:bndry}) we get $\frac{2\pi}{Q}{\bf s\cdot\bm\psi}\in \Bc_{\mathfrak{g}}(\Gamma,2\pi Q^{-1}\Zb/2\pi\Zb)$ and hence ${\bf s\cdot\bm\psi}\in \Bc_{\mathfrak{g}}(\Gamma,\Zb/Q\Zb)$. This contradicts to the condition on ${\bm \psi}$. Hence, from Theorem \ref{modp:equidist:vector}, we obtain
\[ \Eb \left[\exp\left(\frac{2\pi i{\bf s}}{Q}\rdot {\mathfrak{g}_{\bm\psi}}\right)\bigg|\Omega_{M,J}\right]\ll M^{-\gamma_1}. \]
This gives the proof of the first statement.

For the second one, suppose that $\psi$ is a $\mathfrak{g}$-coboundary over $\Zb/q\Zb$ for a prime $q\mid Q$, associated with $\beta$. 
First, we have
\begin{align*}
\Pbb&[ {\mathfrak{g}_{\psi}}\equiv {a}\modulo{q}|\Omega_{M,J}]= 
\frac{1}{q}\sum_{t\in\Zb/q\Zb}e^{-2\pi iat/q}\Eb \left[\exp \left( \frac{2 \pi i}{q} {t}{\mathfrak{g}_{\psi}} \right)\Big\lvert \Omega_{M,J} \right]
\end{align*}
Note that $\frac{2\pi it}{q}\psi$ is also a coboundary associated with $\frac{2\pi i t}{q}\beta$. From Theorem \ref{modp:equidist:vector}, the last expression equals 
\begin{align*}
	\frac{1}{q[\SL_2(\Zb):\Gamma]}\sum_{t}e^{-2\pi iat/q}\sum_{v\in \Gamma\backslash\SL_2(\Zb)}\exp\left[\frac{2\pi it}{q}(\beta(\Gamma)-\beta(v))\right]+o(1).
\end{align*}
This equals $c_a[\SL_2(\Zb):\Gamma]^{-1}+o(1)$ where $c_a=\#\{v\,|\,\beta(v)\equiv\beta(\Gamma)-a\modulo{q}\}$. On the other hand, we get 
$$\Pbb[ {\mathfrak{g}_{\psi}}\equiv {a}\modulo{q}|\Omega_{M,J}]=\sum_{g\equiv a (q)\atop g\in \Zb/Q\Zb}\Pbb[ {\mathfrak{g}_{\psi}}\equiv {g}\modulo{Q}|\Omega_{M,J}]=\frac{Q}{q}\cdot\frac{1}{Q}+o(1)$$ 
for each $a$. Hence, $c_a[\SL_2(\Zb):\Gamma]^{-1}$ are all the same as $q^{-1}$. This is possible only when $q$ is a divisor of $[\SL_2(\Zb):\Gamma]$, which is a contradiction. Hence, we obtain the statement.
\end{proof}

\subsection{Weak correlation between archimedean and residual distributions}\label{sec:noncorr}

In this subsection, we present a result that the Gaussian distribution and residual distribution of modular partition functions are weakly correlated, i.e., non-correlated asymptotically. 

Let $\psi:\Gamma\backslash\SL_2(\Zb)\rightarrow\Zb$ and $Q>1$ be an integer.  For a ${{g}}\in \Zb/Q\Zb$, let $\Omega_{M,J}^{{{g}}}$ be the probability space $\{r\in\Omega_{M,J}\,|\, {\mathfrak{g}_{\psi}}(r)\equiv {{g}}\modulo{Q}\}$ with the uniform density. Let $\underline{\mathfrak{g}}_\psi$ be the normalisation of $\mathfrak{g}_\psi$, i.e., $$\underline{\mathfrak{g}}_\psi:=\frac{\mathfrak{g}_\psi-\mu_\psi\log M}{\sqrt{C_\psi\log M}}$$ 
where $\mu_\psi$ and $C_\psi$ are given in Theorem \ref{vector:far0}. The following result shows that the two distributions of $\mathfrak{g}_\psi$  on $\Omega_{M,J}$ are asymptotically non-correlated for a residual non-coboundary $\psi$:

\begin{thm}\label{weak:correl:cf}
Assume that $\psi\modulo{q}\not\in\Bc_{\mathfrak{q}}(\Gamma,\Zb/q\Zb)$ for each prime $q\mid Q$.	
For each $x\in\Rb$, as $M\rightarrow\infty$, we get:
\begin{enumerate}[leftmargin=*]
\item $\Pbb[ {\underline{\mathfrak{g}}_{\psi}}  \leq x\, \lvert\, \Omega_{M,J}^g]=\Pbb[ {\underline{\mathfrak{g}}_{\psi}}  \leq x\, \lvert\, \Omega_{M,J}]+o(1)$. 
\item 
$\Pbb[ {\underline{\mathfrak{g}}_{\psi}}  \leq x, \mathfrak{g}_\psi\equiv g (Q)\, \lvert\, \Omega_{M,J}]=	\Pbb[ {\underline{\mathfrak{g}}_{\psi}}  \leq x\, \lvert\, \Omega_{M,J}]\cdot	\Pbb[ \mathfrak{g}_\psi\equiv g (Q)\, \lvert\, \Omega_{M,J}]+o(1).$
\end{enumerate}
\end{thm}

\begin{proof}
For (1), it suffices to show that  for $w\in \Cb$ near $ 0$, there exists $\gamma_2>0$ such that $\Eb[\exp(w{\mathfrak{g}_{\psi}})|\Omega_{M,J}^{{g}}]=\Eb[\exp(w{\mathfrak{g}_{\psi}})|\Omega_{M,J}]+O(M^{-\gamma_2}).$ Let us set
$$R(w):=\sum_{r\in\Omega_{M,J}^{g}}\exp(w{\mathfrak{g}_{{\psi}}}(r)).$$
Note that $R(w)/R(0)=\Eb[\exp(w{\mathfrak{g}_{\psi}})|\Omega_{M,J}^{{g}}]$ by definition.

Using the orthogonality of the additive character $t\mapsto \exp(\frac{2\pi i t}{Q})$, we have
\begin{align*}
R(w)= \frac{1}{Q} \sum_{{t\in\Zb/Q\Zb}} e^{- \frac{2 \pi i}{Q} {t}{{g}}} \sum_{r\in\Omega_{M,J}}\exp \left( \Big(\frac{2 \pi i}{Q} {t}+w\Big){\mathfrak{g}_{{\psi}}}(r) \right) .
\end{align*}
Split the sum over $t$ into two parts, $t=0$ and $t\not\equiv { 0}\modulo{Q}$. As done in the proof of Theorem \ref{intro:equid:mod:q}, from $t\not\equiv { 0}\modulo{Q}$ and the hypothesis on $\psi$, we deduce $\frac{2\pi t}{Q}\psi\not\in \Bc_\mathfrak{g}(\Gamma,\Rb/2\pi\Zb)$. Since $w$ is near $0$, from Proposition \ref{av:sum:dnw},  for some $0<\delta<2$, we get 
$$\sum_{r\in\Omega_{M,J}}\exp \left( \Big(\frac{2 \pi t}{Q} {i}+w\Big){\mathfrak{g}_{{\psi}}}(r) \right)=O(M^\delta)\mbox{ if }t\not\equiv 0\modulo{Q}.$$
In sum, $R(w)=\frac{1}{Q} \sum_{r\in\Omega_{M,J}} \exp \left(w{\mathfrak{g}_{{\psi}}}(r) \right)+O(M^\delta)$ and hence, we get $$\frac{R(w)}{R(0)}=\Eb[\exp(w{\mathfrak{g}_{\psi}})|\Omega_{M,J}]+O\left(\frac{M^\delta}{|\Omega_{M,J}|}\right).$$ 
Since $|\Omega_{M,J}|\gg M^2$ by (\ref{dn0:expression}), we obtain the statement with $\gamma_2=2-\delta$.

For (2), a simple calculation gives us
\begin{align*}
	\Pbb[ {\underline{\mathfrak{g}}_\psi}  \leq x, \mathfrak{g}_\psi\equiv g (Q)\, \lvert\, \Omega_{M,J}]=	\Pbb[ {\underline{\mathfrak{g}}_\psi}  \leq x\, \lvert\, \Omega_{M,J}^g]\cdot	\frac{|\Omega_{M,J}^g|}{|\Omega_{M,J}|}.
\end{align*}
From (1), we conclude the proof.
\end{proof}

In particular, we obtain the weak correlation for archimedean and residual distributions of the length of continued fractions for $\Omega_{M,J}$.

\begin{rem}\label{corr:gen:thm}
Two special cases of Proposition \ref{av:sum:dnw}, namely (1) any $\wbb$ near ${\bf v}=\bf 0$ in $\S$\ref{sec:joint:gauss} and (2) $\wbb=i{\bf v}$ with ${\bf v}\neq {\bf 0}$ in $\S$\ref{sec:res:equid}, are sufficient for our main ends of the present paper. Nevertheless, we still need the luxury of generality (any $\wbb$ near $i{\bf v}$) as it is indispensable for the proof of Theorem \ref{weak:correl:cf}.  
\end{rem}

\section{Distribution of modular symbols} \label{sec:dist:mod}

In this section, we show that the modular symbols are non-degenerate specialisation of the modular partition functions in both zero and positive characteristics. Using this, we deduce the distribution results on the modular symbols from those on the modular partitions.

\subsection{Involution on de Rham cohomology}

Let $\mathbb{H}:=\{z\in\Cb\,|\, \Im(z)>0\}$ be the upper-half plane, $\Pbb^1(\Qb):=\Qb\cup \{\infty\}$, and $\mathbb{H}^*:=\mathbb{H} \cup \Pbb^1(\Qb)$. 
Let $\Gamma$ be a congruence subgroup of $\SL_2(\Zb)$ and $X_\Gamma:=\Gamma\backslash \mathbb{H}^*$ the corresponding modular curve.
For two cusps $r, s$ in $\Pbb^1(\Qb)$, we write $\{r,s\}_\Gamma$ for the relative homology class corresponding to the projection to ${X}_\Gamma$ of the geodesic on $\mathbb{H}^*$ connecting $r$ to $s$. For $\Gamma=\Gamma_1(N)$, let us set $\{r,s\}_N:=\{r,s\}_{\Gamma_1(N)}$.

Let us denote by $H_{\mathrm{dR}}^1({X}_\Gamma)$ the first de Rham cohomology of ${X}_\Gamma$. 
We define an operator $\iota$ on $\gamma\in\SL_2(\Zb)$ and $z\in \mathbb{H}^*$ by
$$\gamma^\iota:={\rm j}\gamma {\rm j}\in \SL_2(\Zb) \ \mbox{and} \ z^\iota:=-\cl{z}\in \mathbb{H}^*.$$ 
As $\Gamma$ is assumed to be normalised by ${\rm j}$ (e.g. $\Gamma=\Gamma_1(N)$), the action of $\iota$ yields a well-defined involution on ${X}_\Gamma$.
Let $S_2(\Gamma)$ be the space of cuspforms of weight 2 for $\Gamma$.
The involution $\iota$ then has an action on $H^1_{\mathrm{dR}}({X}_\Gamma)\simeq S_2(\Gamma)\oplus\cl{S_2(\Gamma)}$.
The involution $\iota$ interchanges $S_2(\Gamma)$ and $\cl{S_2(\Gamma)}$.
Moreover, the involution $\iota$ is normal with respect to the cap product
\begin{equation} \label{cap:prod}
\cap : H_1({X}_\Gamma,\Zb)\times H^1_{\mathrm{dR}}({X}_\Gamma) \rightarrow\Cb,\,\,  (\xi,\omega)\mapsto \xi\cap \omega =\int_{\xi}\omega.
\end{equation}
The cap product can be interpreted as follows. For $f \in
S_2(\Gamma)$, $g \in \overline{S_2(\Gamma)}$, and $\{r,s\}_\Gamma\in H_1({X}_\Gamma,\Qb)$, set
\begin{equation} \label{pairing:merel}
\big< \{r,s\}_\Gamma, (f,g) \big> = \int_r^s f(z)dz + \int_r^s g(z)dz^\iota .
\end{equation}  
Then it is known that the pairing $\langle\cdot\,,\,\cdot\rangle$ is non-degenerate (See Merel \cite{merel}). 
Note that the modular symbol $\mG_f^\pm(r)$ can be understood as the above pairing (\ref{pairing:merel}) between a relative homology class $\{r, i\infty\}_N$ and de Rham cohomology class  $(f,\mp f^\iota)$, respectively.

\subsection{Optimal periods}\label{sec:opt:per}

We discuss preliminary results to study the residual distribution of modular symbols.

Let $f$ be a newform of level $N$ and weight $2$. Let ${\bf m}$ be a maximal ideal of the Hecke algebra $\Tbb_N$ such that the characteristic of $\Tbb_N/{\bf m}$ is $p$ and corresponds to $f$. 
There exists a Galois representation $\rho_{\bf m}:\Gal(\cl{\Qb}/\Qb)\rightarrow \GL_2(\Tbb_N/{\bf m})$. 

Let $C_1(N)$ be the set of cusps on $X_1(N)$. Consider $\left(X_1(N), C_1(N)\right)$-relative  homology sequence
\begin{equation}\label{Homology:exact}
 0\rightarrow H_1(X_1(N),\Zb) \rightarrow H_1(X_1(N),C_1(N),\Zb) \rightarrow H_0(C_1(N),\Zb) \rightarrow \Zb \rightarrow 0.
\end{equation}
For a prime $q$ with $q\equiv 1\modulo{Np}$, let $D_q={\rm T}_q -q \cha{q}-1$. The following is observed in Greenberg-Stevens \cite{greenberg:stevens}: 
The operator $D_q$ annihilates  $H_0(C_1(N),\Zb)$ in (\ref{Homology:exact}). Let $\Tbb_{N,{\bf m}}$ denote the completion of the Hecke algebra $\Tbb_N$ at $\bf m$. Since $D_q$ is a unit in $\Tbb_{N,{\bf m}}$ if $\rho_{\bf m}$ is irreducible, we can conclude that $H_1(X_1(N),\Zb)_{\bf m}$ is isomorphic to $H_1(X_1(N),C_1(N),\Zb)_{\bf m}$. 
For a $\Zb_p$-algebra $R$ with the trivial action of $\Tbb_N$, we have a perfect pairing 
\begin{align}\label{perf:pairing:m}
H_1(X_1(N),R)_{\bf m}\times H^1(\Gamma_1(N),R)_{\bf m}\rightarrow R.
\end{align} 
When $R$ is given by $\Cb$, the pairing is realized as the Poincar\'{e} pairing under the isomorphism $\Cb_p\simeq \Cb$.

Let $\Oc$ be an integral extension of $\Zb_p$ including the Fourier coefficients of $f$. Assume $N \geq 3$, $p \nmid 2N$, and $\rho_{\bf m}$ is irreducible.  Then, there is a Hecke equivariant isomorphism 
\begin{align}\label{isom:cohom:modform}
{\delta^\pm}: S_2(\Gamma_1(N),\Oc)_{\bf m}\,\,{\cong}\,\, H^1(\Gamma_1(N),\Oc)_{\bf m}^\pm .
\end{align}
It is the isomorphism mentioned in Vatsal \cite{vatsal}.

Let $\omega_f\in H_1(\Gamma_1(N),\Cb)$ be a cohomology class corresponding to $2\pi if(z) dz$. 
Using the isomorphism (\ref{isom:cohom:modform}) and the theorem of strong multiplicity one, the periods $\Omega^\pm_f\in\Cb_p$ can be chosen (see Vatsal \cite{vatsal}) so that 
\begin{align}\label{optimal:period:def}
\Omega_f^\pm\delta^\pm(f)=\omega_f\pm \omega_f^\iota.
\end{align}
It is known that for a newform $f_E$ corresponding to an elliptic curve $E$ over $\Qb$, the period $\Omega_{f_E}^\pm$ can be chosen as  the N\'eron periods $\Omega_E^\pm$ of $E$.

\subsection{Modular symbol as a modular partition function: Manin's trick}\label{subsec:manin:trick}

We describe how the statistics of continued fractions enters into our discussion on the distribution of modular symbols.

Let $f$ be a newform for $\Gamma_0(N)$ of weight $2$, i.e., cuspform for $\Gamma_1(N)$ with the trivial Nebentypus. Manin \cite{manin} notices that the period integral can be written as
\begin{align*}
\int_0^r f(z)dz= \sum_{i=1}^\ell \int_{\frac{P_{i-1}}{Q_{i-1}}}^{\frac{P_i}{Q_i}} f(z)dz = - \sum_{i=1}^\ell \int_{\widetilde{g_i}(r) \cdot 0}^{\widetilde{g_i}(r) \cdot \infty} f(z)dz
\end{align*}
with $\widetilde{g}_i(r)=\twobytwotiny{P_{i-1}}{-P_i}{Q_{i-1}}{-Q_i}$ if $\det g_i(r)=-1$.
Setting $g_i=g_i(r)$, we also get
\begin{align*}
	\int_0^{-r} f(z)dz=\sum_{i=1}^\ell \int_{-\frac{P_{i-1}}{Q_{i-1}}}^{-\frac{P_i}{Q_i}} f(z)dz= - \sum_{i=1}^\ell \int_{{\rm j}\widetilde{g_i}{\rm j} \cdot 0}^{{\rm j}\widetilde{g_i}{\rm j} \cdot \infty} f(z)dz.
\end{align*}

For $u\in \Gamma_0(N)\backslash \SL_2(\Zb)$, define 
$$\psi_f^\pm(u):=\frac{1}{\Omega_f^\pm}\left(\int_{u \cdot 0}^{u \cdot \infty} f(z)dz \pm \int_{{\rm j}u{\rm j} \cdot 0}^{{\rm j}u{\rm j} \cdot \infty}f(z)dz\right)  \in \Qb_f.$$
By the definition of $\cG_{\psi_f^-}$, the modular symbols are expressed as
\begin{align}\label{manin:trick:tilde}
	\mG_f^-(r)= -\cG_{\psi_f^-}(r).
\end{align}
We observe that $\psi_f^\pm(u{\rm j})=\pm\psi_f^\pm({\rm j}u)$. Hence, we also get $\psi_f^+(\Gamma \widetilde{g}_i)=\psi_f^+(\Gamma\widehat{g}_i)$ and 
\begin{align}\label{manin:trick:tilde2}
\mG_f^+(r)= \frac{2L(1,f)}{\Omega_f^+}-\bG_{\psi_f^+}(r).
\end{align}

Let us use the optimal periods in \S\ref{sec:opt:per} with the same notation $\Omega_f^\pm$ when we study the residual equidistribution of modular symbols. By previous discussion, one obtains ${\mG}_f^\pm(r)\in \mathcal{O}$ for each $r$.
We define $\zeta_f^\pm:\Gamma_1(N)\backslash\SL_2(\Zb)\rightarrow \mathcal{O}$ by 
$$\zeta_f^\pm(u):=\{u\cdot 0,u\cdot\infty\}_N \cap \delta^\pm(f).$$
Note that we have
\begin{align}\label{manin:trick:zeta}
	\mG_f^+(r)= \frac{2L(1,f)}{\Omega_f^+}-\bG_{\zeta_f^+}(r)\mbox{ and }	\mG_f^-(r)= -\cG_{\zeta_f^-}(r).
\end{align}

\begin{rem}\label{rem:not:bG}
	The representation (\ref{manin:trick:tilde}) is no longer true for $\widehat{g}_i(r)$. In fact, we have
	$$\int_{\widehat{g_i}(r) \cdot 0}^{\widehat{g_i}(r) \cdot \infty} f(z)dz=-\int_{\pm\frac{P_{i-1}}{Q_{i-1}}}^{\pm\frac{P_i}{Q_i}} f(z)dz$$
	as $\widehat{g}_i=\twobytwotiny{\pm P_{i-1}}{\pm P_i}{Q_{i-1}}{Q_i}$ for $\det g_i(r)=\pm1$ and hence
	$$\int_0^rf(z)dz-\int_0^{-r}f(z)dz=\sum_{i=1}^\ell (-1)^i \psi_f^-(\Gamma\widehat{g}_i),$$
	which is not equal to $\bG_{\psi_f^-}(r)$, in general.
\end{rem}

\begin{rem}\label{man:mar:error}
As far as we understand, Manin--Marcolli seemed to assert that the modular symbols are expressible in terms of $\bG_{\psi_f^\pm}$. As discussed in Remark \ref{rem:not:bG}, it is doubtful to have such an expression for $\psi_f^-$. A relevant mistake is that they regarded $g_j(r)$ as an element of ${\rm PSL}_2(\Zb)$, which is not the case if $\det g_j(r)=-1$ (See Manin-Marcolli \cite[p.6, Line 12]{manin:marc}).
\end{rem}

\subsection{Gaussian distribution: Proof of Theorem \ref{result:symbol}}\label{subsec:normal:mod:sym}

In this subsection, we give a proof on the limit Gaussian distribution of modular symbols $\mG_f^\pm$ on $\Omega_{M,\varphi,J}$.

\begin{prop} \label{nonvanish:w}
	For any nontrivial $f\in S_2(\Gamma_0(N))$, the function $\psi_f^+$ ($\psi_f^-$, resp.) is not a $\bG$-coboundary ($\cG$-coboundary, resp.) over $\Rb$. 
\end{prop}

\begin{proof}
	
	First assume that $\psi_f^-$ is a $\cG$-coboundary over $\Rb$. In other words, there exists $\beta\in \Rb^{\cha{\Gamma}}$ such that $\psi_f^-(u)+\psi_f^-(u\twobytwotiny{-m}{1}{1}{0}{\rm j})=\beta(u)-\beta(u\twobytwotiny{-m}{1}{1}{0}\twobytwotiny{-n}{1}{1}{0})$ for all $u$ and  $m,n\in\Zb$. Taking $m=n=0$, we get
	$\psi_f^-(u)=-\psi_f^-(u\iota)$ for each $u$ and $\iota=\twobytwotiny{0}{-1}{1}{0}$. Furthermore we have $u\iota\cdot \infty=-u\cdot 0$ and $u\iota\cdot 0=-u\cdot\infty$. Hence, $\psi_f^-(u\iota)=\psi_f^-(u)$. In sum, $2\psi_f^-$ is the zero function. On the other hand, Manin's trick implies that the set of Manin symbols $\{u\cdot 0,u\cdot\infty\}_{\Gamma_0(N)}$ for $u\in \Gamma_0(N)\backslash\SL_2(\Zb)$ generates the first homology group of $X_0(N)$. Since the pairing (\ref{cap:prod}) is non-degenerate, we conclude that $\psi_f^\pm$ are not trivially zero as long as $f$ is non-trivial. This is a contradiction and hence we conclude that $\psi_f^-$ is not a $\cG$-coboundary.


	Assume that $\psi_f^+$ is a $\bG$-coboundary over $\Rb$, i.e., $\psi_f^+(u)=\beta(u)-\beta(u\cdot \twobytwotiny{-m}{1}{1}{0})$ for all $m$ and $u$. 
	Let $h:\Gamma_0(N)\rightarrow H_1(X_0(N),\Zb)$ given by $h(\gamma):=\{0,\gamma\cdot 0\}_{\Gamma_0(N)}$. Let $\gamma\in\Gamma_0(N)$. Note that 
	$\{0,\gamma\cdot 0\}_{\Gamma_0(N)}=\{\infty,\gamma\cdot\infty\}_{\Gamma_0(N)}$. Since $\twobytwotiny{1}{m}{0}{1}\in\Gamma_0(N)$ for each $m\in\Zb$, we get $$\{\infty,\gamma\cdot\infty\}_{\Gamma_0(N)}=\{\twobytwotiny{1}{m}{0}{1}\cdot\infty,\twobytwotiny{1}{m}{0}{1}\gamma\cdot\infty\}_{\Gamma_0(N)}=\{\infty,\twobytwotiny{1}{m}{0}{1}\gamma\cdot\infty\}_{\Gamma_0(N)}.$$ Therefore we have $h(\gamma)=\{0,\twobytwotiny{1}{m}{0}{1}\gamma\cdot 0\}_{\Gamma_0(N)}=h(\twobytwotiny{1}{m}{0}{1}\gamma)$ for each $m\in\Zb$ and hence may assume that $0< \gamma\cdot 0< 1$. Let $\gamma\cdot 0=[0;m_1,\cdots,m_\ell]$ and $g_i=g_i(\gamma\cdot 0)$. Note that $g_i\twobytwotiny{-m_i}{1}{1}{0}=g_{i-1}$ with $g_0=I$. Then, we have 
	\begin{align*}
	h(\gamma)&\cap (\omega_f+\omega_f^\iota)=\int_0^{\gamma\cdot 0}f(z)dz+\int_0^{-\gamma\cdot 0}f(z)dz=\sum_{i=1}^\ell\psi_f^+(\Gamma_0(N) \widehat{g}_i)\\
	&=\sum_{i=1}^\ell \beta(\Gamma_0(N)\cdot g_i)-\beta(\Gamma_0(N)\cdot g_i \twobytwotiny{-m_i}{1}{1}{0})=\beta(\Gamma_0(N)\widehat{g}_\ell)-\beta(\Gamma_0(N)).
	\end{align*}

Observe that $\psi_f^+(u)=\psi_f^+(-u)$ and $\psi_f^+({\rm j}u)=\psi_f^+(u{\rm j})$ for all $u$. By Proposition \ref{misc:beta:prop}.(\ref{psl:invj}), we know $\beta({\rm j}u)=\beta(u{\rm j})$ for all $u$. In particular, $\beta(\Gamma_0(N)\widehat{g}_\ell)=\beta(\Gamma_0(N)\widetilde{g}_\ell)$.
Note that $\gamma\cdot 0=\widetilde{g}_\ell\cdot 0$, i.e., $\widetilde{g}_\ell\in \gamma{\rm L}$. Since $\beta$ is ${\rm L}$-invariant, we get $\beta(\Gamma_0(N)\widetilde{g}_\ell)=\beta(\Gamma_0(N)\gamma)=\beta(\Gamma_0(N))$. In sum we get $h(\gamma)\cap (\omega_f+\omega_f^\iota)=0$ for all $\gamma\in\Gamma_0(N)$. Since it is well-known that $h$ is surjective, we conclude from the non-degeneracy of the pairing (\ref{cap:prod}) that $f=0$, which is a contradiction. 
\end{proof}

We are ready to present:

\begin{proof}[Proof of Theorem \ref{result:symbol}]
By the expressions (\ref{manin:trick:tilde}) and (\ref{manin:trick:tilde2}), the distribution of modular symbols follows from the ones of modular partition functions $\bG_{\psi_f^+}$ and $\cG_{\psi_f^-}$. Now Theorem \ref{result:symbol} follows from Theorem \ref{vector:far0} and Proposition \ref{nonvanish:w}.
\end{proof}

\subsection{Residual distribution: Proofs of Theorem \ref{res:equid:mod} and \ref{non:van:thm:intro}}\label{subsec:res:dist:mod} 

In this subsection, we give proofs on the residual equidistribution of integral random variable ${\mG}_E^\pm$ on $\Omega_{M,\varphi,J}$ and non-vanishing result on the special $L$-values.

First, we need the following:

\begin{prop} \label{nonvanish:w:modp}
	Let $\varpi$ be a uniformizer of $\mathcal{O}$. Assume $N \geq 3$, $p \nmid 2N$, and $\rho_{\bf m}$ is irreducible. Let $f\not\equiv0\modulo{\varpi}$. Then  $\zeta_f^+ \modulo{\varpi}$ ($\zeta_f^- \modulo{\varpi}$, resp.) is not a $\bG$-coboundary ($\cG$-coboundary, resp.) over $\mathcal{O}/(\varpi)$.
\end{prop}

\begin{proof}

First assume that $\zeta_f^-\modulo{\varpi}$ with $\zeta_f^-(u)=\{u\cdot 0,u\cdot i\infty\}_N\cap \delta^\pm(f)$ is a $\cG$-coboundary over $\mathcal{O}/(\varpi)$. As done in the proof of Proposition \ref{nonvanish:w}, for each $u$, we get $\zeta_f^-(u)=-\zeta_f^-(u\iota)$ with the action of $\iota$ on $\delta^\pm(f)$. Moreover, we have $\zeta_f^-(u\iota)=\zeta_f^-(u)$ using the action of $\iota$ on $0$ and $i\infty$. In sum, we get $2\zeta_f^-\equiv0\modulo{\varpi}$. On the other hand, the Manin symbols generate the first homology group of $X_1(N)$. Therefore due to the perfectness of the pairing (\ref{perf:pairing:m}), the congruence $\zeta_f^-\equiv0\modulo{\varpi}$ implies that $\delta^-(f)\equiv0\modulo{\pi}$. However it is forbidden by the hypothesis using (\ref{isom:cohom:modform}). In total, we conclude that $\zeta_f^-$ is not a $\cG$-coboundary over $\mathcal{O}/(\varpi)$.


Assume that $\zeta_f^+\modulo{\varpi}$ is a $\bG$-coboundary over $\Oc/(\varpi)$, i.e., there exists a function $\beta\in (\Oc/(\varpi))^{\cha{\Gamma_1(N)}}$ such that $\zeta_f^+(u)\equiv\beta(u)-\beta(u\cdot \twobytwotiny{-m}{1}{1}{0})\modulo{\varpi}$ for all $m$ and $u\in\Gamma_1(N)\backslash\SL_2(\Zb)$. 
Let us set $k(\gamma):=\{\infty,\gamma\cdot \infty\}_N\cap \delta^+(f)\in\Oc$ for $\gamma\in\Gamma_1(N)$. 
\hide{++++++++++++++++++++++++++++
Using (\ref{optimal:period:def}), we also get $k(\gamma)=(\Omega_f^\pm)^{-1} \{\infty,\gamma\cdot\infty\}_{\Gamma_0(N)}\cap (\omega_f\pm\omega_f^\iota)$. Note that $k(\gamma)=k(\gamma\twobytwotiny{1}{m}{0}{1})$ for $m\in\Zb$. We can find an $m$ so that $0< \gamma'\cdot 0< 1$ for $\gamma':=\gamma\twobytwotiny{1}{m}{0}{1}\in\Gamma_0(N)$. Let $\gamma'\cdot 0=[0;m_1,\cdots,m_\ell]$ and let $g_1,\cdots,g_\ell$ be the convergents of $\gamma'\cdot 0$.  Since $\{0,\gamma'\cdot 0\}_{\Gamma_0(N)}=\{\infty,\gamma'\cdot\infty\}_{\Gamma_0(N)}$, we have 
	\begin{align*}
	k(\gamma)=\frac{1}{\Omega_f^\pm }\{0,\gamma'\cdot 0\}_{\Gamma_0(N)}\cap (\omega_f\pm\omega_f^\iota)=\sum_{i=1}^\ell\zeta_f^+(\Gamma_1(N) \widehat{g}_i). 	
	\end{align*}
Using (\ref{optimal:period:def}) we obtain $\zeta_f^+(u\delta)=\zeta_f^+(u)$ for all $\delta\in\Gamma_0(N)$. By Proposition \ref{Gamma1:Gamma0}, we get $\beta\in \Oc^{\cha{\Gamma_0(N)}}$. 
Similarly as the previous proof, we obtain
	\begin{align*}
	k(\gamma)\equiv\beta(\Gamma_0(N)\widehat{g}_\ell)-\beta(\Gamma_0(N)) \modulo{\varpi}
	\end{align*}
and 
+++++++++++++++++++++++++++}
Similarly as the previous proof, from the observation that $\zeta_f^+(u)=\zeta_f^+(-u)$ and $\zeta_f^+({\rm j}u)=\zeta_f^+(u{\rm j})$ for all $u$, we can conclude that $k(\gamma)\equiv 0\modulo{\varpi}$ for all $\gamma\in\Gamma_1(N)$. Using the non-degeneracy of the pairing (\ref{perf:pairing:m}) and the isomorphism (\ref{optimal:period:def}), we obtain $f\equiv0\modulo{\varpi}$, which is a contradiction. This finishes the proof.
\end{proof}

We are ready to give:

\begin{proof}[Proof of Theorem \ref{res:equid:mod}]
	By the expressions (\ref{manin:trick:zeta}), the residual distribution of modular symbols follows ones of modular partition functions $\bG_{\zeta_f^+}$ and $\cG_{\zeta_f^-}$. Now Theorem \ref{res:equid:mod} follows from Theorem \ref{intro:equid:mod:q} and Proposition \ref{nonvanish:w:modp}.
\end{proof}

We also present: 

\begin{proof}[Proof of Theorem \ref{non:van:thm:intro}]
Let $c$ be a number less than $1-\sqrt{ 1-\frac{6}{\pi^2}(1-\frac{1}{p})}$.
Let $\widehat{(\Zb/n\Zb)}^\times_{\pm}$ be the set of Dirichlet characters modulo $n$ that are even or odd according to the parity $\pm$. Let us set
$$T_M^\pm:=\left\{1<n\leq M \,\bigg|\,\exists\,\chi\in \widehat{(\Zb/n\Zb)}^\times_\pm, \Lambda_E(\chi)\not\equiv0\,(\mathfrak{p}^{1+v_{\mathfrak{p}}(\phi(n))})\right\}.$$
The statement follows once the following inequality is verified:
$\# T_M^\pm\geq cM$ for all sufficiently large $M$.
Let us assume the contrary, i.e., suppose that $\#T_M^\pm< c M$ for infinitely many $M$. Then for each $n\not\in T_M^\pm$ with $1<n\leq M$ and $m\in (\Zb/n\Zb)^\times$, we obtain
\begin{align*}
\sum_{\chi\in \widehat{(\Zb/n\Zb)}^\times_{\pm}}\cl{\chi}(m)\Lambda_E(\chi)\equiv 0\modulo{\mathfrak{p}^{1+v_{\mathfrak{p}}(\phi(n))}}.
\end{align*}
Then for all $r\in \Sigma_{n}$ with $n\not\in T_M^\pm$, we obtain ${\mathfrak{m}}_E^\pm(r)\equiv 0\modulo{p}$. From this, we can conclude 
$$\sum_{1<n\leq M\atop n\not\in T_M^\pm}\phi(n)<\frac{1}{p}\sum_{1<n\leq M}\phi(n)\mbox{, i.e., }\sum_{n\in T_M^\pm}\phi(n)>\left(1-\frac{1}{p}\right)\sum_{1<n\leq M}\phi(n).$$
Since $\# T_M^\pm<cM$, the L.H.S. is smaller than or equal to
$$\sum_{M-cM<n\leq M}n \leq \frac{1}{2}(1-(1-c)^2)M^2.$$
Note that $\lim_{M\rightarrow\infty}\frac{1}{M^2}\sum_{n\leq M}\phi(n)=\frac{3}{\pi^2}$. Hence we obtain $(1-c)^2\leq 1-\frac{6}{\pi^2}(1-\frac{1}{p})$ which is a contradiction to the choice of $c$. 
\end{proof}

\begin{rem}
It seems that  it is currently not doable to deduce an estimate on
$$\#\bigcup_{n\leq M, p\nmid \phi(n)}\left\{\chi\in \widehat{(\Zb/n\Zb)}^\times \,\bigg|\,\Lambda_E(\chi)\not\equiv0\,(\mathfrak{p}),\,\chi(-1)=\pm1\right\}$$
from the previous proof or similar argument since the set $\{n\leq M\,|\, p\nmid \phi(n)\}$ is too thin as its size is asymptotic to $M(\log M)^{-1/(p-1)}$ (See Spearman--Williams \cite{spea:will}).
\end{rem}

\begin{rem}
It is worthwhile to mention about previous research on the residual non-vanishing of $L$-values. The ergodic approach for the Dirichlet $L$-values has been extensively generalised to the study of anti-cyclotomic twists (for example, see Hida \cite{hida},  Burungale--Hida \cite{burungale:hida}, and Vatsal \cite{vatsal2}).  Meanwhile, up until now, there has been no notable analogous progress for the modular $L$-values with cyclotomic or Dirichlet twists except a few cases. 
The first non-vanishing result goes back to Ash--Stevens \cite{ash:stevens} and Stevens \cite{stevens} for a large class of characters. Kim--Sun \cite{kim:sun} recently obtained the non-vanishing result for a positive proportion of characters $\chi$ of $\ell$-power conductors with a prime $\ell \neq p$. However, all of these results are based on the classical arguments and their improvements. It is also worthwhile to mention another ergodic approach for the Dirichlet $L$-values proposed recently by Lee-Palvannan \cite{lee:pal}.
\end{rem}

Using Theorem \ref{weak:correl:cf}, we present an answer to Mazur's question on the weak correlation between archimedean and residual distributions of modular symbols:

\begin{thm}\label{mod:symb:correl}
	Assume that $\cl{\rho}_{E,p}$ is irreducible and $p\nmid N_E$. For $x\in\Rb$ and $a\in \Zb/p^e\Zb$, as $M\rightarrow\infty$, we get
		\[ 
		\Pbb[\underline{\mG}_E^\pm\leq x,\,\mG_E^\pm\equiv a (p^e) |\Omega_{M,J}]= \Pbb[\underline{\mG}_E^\pm\leq x|\Omega_{M,J}]\cdot \Pbb[\mG_E^\pm\equiv a (p^e) |\Omega_{M,J}]+o(1).
		\]
\end{thm}

\section{Skewed Gauss dynamical systems} \label{sec:transition}

The remaining part of this paper will be devoted to explain in details how Proposition \ref{main:dynamics} can be obtained. We first present an underlying dynamical description for the modular partitions motivated by the work of Baladi--Vall\'ee \cite{bv}.

\subsection{Skewed Gauss Map} \label{skew:system}

Let us recall that the skewed Gauss map $\Tb$ on ${I} \times \Gamma \backslash \GL_2(\Zb)$ is given by $\Tb(x,v) = \left( T(x), v\twobytwotiny{-m_1(x)}{1}{1}{0} \right)$ 
and the skewed Gauss map $\widehat{\Tb}$ on $I_\Gamma={I} \times \Gamma \backslash \SL_2(\Zb)$ is given by $\widehat{\Tb}(x,v) = \left( T(x), v\cdot\twobytwotiny{-m_1(x)}{1}{1}{0} \right)$. 

Let $K^\circ(m_1,\cdots,m_\ell)$ be the open fundamental interval associated with the digits $m_i$, in other words,
$$K^\circ(m_1,\cdots,m_\ell):=\{[0;m_1,\cdots,m_\ell+x]\,|\,0<x<1\}.$$
An easy observation is
\begin{align}\label{markov:skew}
\widehat{\Tb}^\ell(K^\circ(m_1,\cdots,m_\ell)\times \{v\})=(0,1)\times \Big\{v\cdot\twobytwotiny{-m_1}{1}{1}{0}\twobytwotiny{-m_2}{1}{1}{0}\cdots \twobytwotiny{-m_\ell}{1}{1}{0}\Big\}.
\end{align}
It can be easily seen that $\Tb$ and $\widehat{\Tb}$ are  measure-preserving and in fact are ergodic with respect to the product measure of the Gauss measure and counting measure on the skewed Gauss dynamical systems $({I} \times \Gamma \backslash \GL_2(\Zb),\Tb)$ and $(I_\Gamma,\widehat{\Tb})$, respectively. However, measure-theoretic properties will not be investigated in this paper as we restrict our attention to topological properties.

For a dynamical system $(X,f)$, the map $f$ is called \emph{topologically transitive} if for any non-empty open subsets $U$ and $V$ in $X$, there exists a positive integer $L$ such that $f^L(U) \cap V \neq \varnothing$; and \emph{topologically mixing} if $f^n(U) \cap V \neq \varnothing$ for all $n \geq L$. Notice that if $f$ is topologically mixing, then it is topologically transitive.

\hide{+++++++++++++++++++++++++++++++++++++
\begin{prop} \label{top:trans}
	Let $\Gamma$ be a subgroup of $\SL_2(\Zb)$ with a finite index.
	\begin{enumerate}[leftmargin=*]
		\item The map $\widehat{\Tb}$ on $I_\Gamma$ is topologically transitive.
		\item \label{dens:two:sets}For any $(x_0,v_0)\in I_\Gamma$, the set $\bigcup_{n\geq 1}\widehat{\Tb}^{-n}(x_0,v_0)$ is dense in $I_\Gamma$.
	\end{enumerate}
\end{prop}

\begin{proof}
(2) Let $V$ be an open subset of $I_\Gamma$. We may assume $V=K^\circ(a_1,\cdots,a_k)\times\{u\}$ for a right coset $u$ and some $a_i$. By Lemma \ref{gen:mat:lem}.(\ref{item:gen:digit:mat}), there exists $m_1,\cdots,m_\ell$ such that 
	$$u\twobytwo{-a_1}{1}{1}{0}\cdots\twobytwo{-a_k}{1}{1}{0}\cdot\twobytwo{-m_1}{1}{1}{0}\cdots\twobytwo{-m_\ell}{1}{1}{0}=v_0.$$
	Then we have
	$$(x_0,v_0)\in \widehat{\Tb}^{k+\ell}(K^\circ(a_1,\cdots,a_k,m_1,\cdots,m_\ell)\times\{u\})\subset \widehat{\Tb}^{k+\ell}(V).$$
	Hence, the set $\bigcup_{n\geq 1}\widehat{\Tb}^{-n}(x_0,v_0)$ is dense. 

(1) The transitivity comes from the first statement.	
\end{proof}
+++++++++++++++++++++++++++++++++++++++++++++++}

\begin{prop} \label{op:alg:simple}
\begin{enumerate}[leftmargin=*]
\item The map $\widehat{\Tb}$ on $I_\Gamma$ is topologically mixing.
\item \label{mixing:dense}For any sequence $(x_n,v_n)$ in $I_\Gamma$, the set $\bigcup_{n\geq 1}\widehat{\Tb}^{-n}(x_n,v_n)$ is dense in $I_\Gamma$.
\end{enumerate}
\end{prop}

\begin{proof}
(1). Take any non-empty open sets $U$ and $V$ in $I_\Gamma$. Then, one can assume that $U$ is of the form $(a,b)\times \{u\}$ for some $0<a<b<1$ and $u \in \Gamma\backslash\SL_2(\Zb)$. Since the Gauss map $T$ satisfies the strong Markov property, i.e., $T([ \frac{1}{m+1}, \frac{1}{m}))=[0,1)$ for all $m \geq 1$, we have $T^n(a,b)={I}$ for all sufficiently large $n$. 
Once we have the full image on the first coordinate, we obtain all the elements in the skewed component at all sufficiently many iterations as well using Proposition \ref{prop:cong:Tadm}. Hence we can conclude that $\widehat{\Tb}^n(U)\cap V=I_\Gamma \cap V \neq \varnothing$ for all sufficiently large $n$.

(2). Let $V$ be an open subset of $I_\Gamma$. We may assume $V=K^\circ(a_1,\cdots,a_k)\times \{u\}$. By Proposition \ref{prop:cong:Tadm}, there exists $L\geq 1$ such that for all $\ell\geq L$, we get  
	$$u\cdot\twobytwotiny{-a_1}{1}{1}{0}\cdots\twobytwotiny{-a_k}{1}{1}{0}\cdot\twobytwotiny{-m_1}{1}{1}{0}\cdots\twobytwotiny{-m_\ell}{1}{1}{0}=v_n$$
some $m_1,\cdots,m_\ell$.
	Then, by (\ref{markov:skew}), we have
	$(x_n,v_n)\in \widehat{\Tb}^{n}(V)$ when $n\geq k+L$.	Hence, we prove the second statement.
\end{proof}

\begin{rem}\label{not:T:adm}
In a similar way, one can show that $\Tb$ is also transitive. However, it is easy to see that $\Tb$ is not topologically mixing for any  subgroup $\Gamma$ of $\SL_2(\Zb)$.
\end{rem}

\subsection{Inverse branches}

Let $\Qbb$ be the set of inverse branches of $\Tb$, that is,
$$\Qbb:=\{\qb_m\,|\, m\in\Zb_{\geq 1}\}$$
where an inverse branch $\qb_m:{I}\times\Gamma\backslash\GL_2(\Zb)\rightarrow {I}\times\Gamma\backslash\GL_2(\Zb)$ is given by
\begin{align}\label{inv:branch}
	\qb_m(x,v)=\left( \frac{1}{m+x}, v\twobytwo{0}{1}{1}{m}\right).
\end{align} 
Let $\mathbf{F}\subseteq \Qbb$ be the final set that consists of branches corresponding to  the final digits of continued fractions. In other words, it is given by 
$$\mathbf{F}:=\{\qb_m\,|\, m\geq 2\}.$$

\subsubsection{Basic setting}
For $n\geq 1$, let us denote by ${\Qbb}^{\circ n}$ the set of inverse branches of the $n$-th iterate $\Tb^n$, which is equal to 
$${\Qbb}^{\circ n}:=\Qbb\circ\cdots\circ\Qbb=\{\qb_{m_n} \circ \qb_{m_{n-1}} \circ \cdots \circ \qb_{m_1}\,|\,m_1,\cdots,m_n\geq 1\}$$
and ${\bf Q}^{\circ 0}:=\{{\bf id}_{I_\Gamma}\}$. Let us also set 
$$\Qbb^{\infty}:=\bigcup_{n \geq 0} \Qbb^{\circ n}.$$
The index $n$ is called the \emph{depth} of the inverse branches. 
For an inverse branch $\qb=\qb_{m_n} \circ \cdots \circ \qb_{m_2} \circ \qb_{m_1}$  of depth $n$ and $i\leq n$, let us set the $i$-th part $\qb^{(i)}$ of $\qb$ as
$$\qb^{(i)}:=\qb_{m_i} \circ \cdots\circ\qb_{m_2} \circ \qb_{m_1}=\Tb^{n-i}\circ\qb\in\Qbb^{\circ i}.$$
For a branch $\qb(x,u)=(y(x),ug)$, let us set $\pi_i\qb$ as the $i$-th component of $\qb$, i.e.,
$$\pi_1\qb(x,u):=y(x)\mbox{ and }\pi_2\qb(x,u):=ug.$$

\subsubsection{Definitions for $\bG$ and $\cG$}
Let $\qb\in \Qbb^{\infty}$ be given as $\qb(x,u)=(y(x),ug)$ for some $y(x)$ and $g\in\GL_2(\Zb)$. Then for $v\in\Gamma\backslash\SL_2(\Zb)$, let us define
$$\widehat{\qb}(x,v):=(y(x),v\widehat{g})=(y(x),v\cdot g).$$
These consist of the set of inverse branches of $\widehat{\Tb}$, denoted by $\widehat{\Qbb}$. We also set 
$$\widetilde{\qb}(x,v):=(y(x),v\widetilde{g}).$$
It can be easily checked from the action of $\GL_2(\Zb)$ on $\Gamma\backslash\SL_2(\Zb)$ that for all $\mathbf{p}$ and $\qb\in \Qbb^{\infty}$, one obtains
\begin{align}\label{rel:inv:branch}
\widehat{\mathbf{p}\circ\qb}=\widehat{\mathbf{p}}\circ\widehat{\qb}\mbox{ and } \widehat{\qb}^{(i)}=\widehat{\qb^{(i)}}=\widehat{\Tb}^{n-i}\circ\widehat{\qb}.
\end{align}
For $\qb\in\Qbb\circ\Qbb$, the map $\pi_2\qb$ is now just a right action of $\SL_2(\Zb)$ by the relation (\ref{tilde:g12:right}). Therefore, if $\qb$ is of even depth, then $\widehat{\qb}=\widetilde{\qb}=\qb|_{I_\Gamma}$ and for all $\pb\in\Qbb^\infty$ we obtain
\begin{align}\label{tilde:decomp:pq}
	\widetilde{\mathbf{p}\circ\qb}=\widetilde{\mathbf{p}}\circ\widetilde{\qb}.
\end{align} 
In particular, for $\qb\in\Qbb^{\circ 2n}$ and $1\leq i\leq n$,
\begin{align*}
\widetilde{\qb}^{(2i)}=\widetilde{\qb^{(2i)}}=(\widehat{\Tb}^2)^{n-i }\circ \widetilde{\qb}\mbox{ and }
\widetilde{\qb^{(2i-1)}}=\widetilde{\Tb}\circ(\widehat{\Tb}^2)^{n-i}\circ\widetilde{\qb}.
\end{align*}

\hide{==============
\begin{proof}
Set $\qb=\qb_{m_n}\circ\cdots\circ\qb_{m_1}$, i.e., $$\qb(x,v)=([0;m_n,\cdots,m_1+x], v\twobytwotiny{0}{1}{1}{m_1}\cdots \twobytwotiny{0}{1}{1}{m_n}).$$
Note that
$$\widehat{\qb}(x,v)=([0;m_n,\cdots,m_1+x], v\cdot\twobytwotiny{0}{1}{1}{m_1}\cdots \twobytwotiny{0}{1}{1}{m_n})$$
and that
$$\widetilde{\qb}(x,v)=
\begin{cases}
	\widehat{\qb}(x,v)&\mbox{ if $n$ is even}\\
	([0;m_n,\cdots,m_1+x], v\twobytwotiny{0}{1}{1}{m_1}\cdots \twobytwotiny{0}{1}{1}{m_n}{\rm j})&\mbox{  otherwise}
\end{cases}.$$
Hence, if $\qb\in \Qbb^{\circ 2n}$, then
\begin{align*}
\widetilde{\qb^{(2i-1)}}(x,v)&=([0;m_{2i-1},\cdots,m_1+x], v\twobytwotiny{0}{1}{1}{m_1}\cdots \twobytwotiny{0}{1}{1}{m_{2i-1}}{\rm j})\\
&=\widetilde{\Tb}([0;m_{2i},\cdots,m_1+x], v\twobytwotiny{0}{1}{1}{m_1}\cdots \twobytwotiny{0}{1}{1}{m_{2i}})\\
&=\widetilde{\Tb}(\qb^{(2i)}(x,v))\\
&=\widetilde{\Tb}(\Tb^{2(n-i)}(\qb(x,v)))
\end{align*}
\end{proof}

======================}

\subsubsection{Specialisation}
It can be easily seen that there is a one-to-one correspondence between 
$\Qb\cap (0,1)$ and $\mathbf{F}\circ \Qbb^{\infty}$ given by
\begin{align*}
	r=[0;m_1,\cdots,m_\ell] \longmapsto \qb_r:=\qb_{m_\ell} \circ \cdots \circ \qb_{m_1}
\end{align*}
with $m_1, \cdots, m_{\ell-1} \geq 1$ and $m_\ell \geq 2$. We obtain
\begin{prop}
For each $r\in\Qb\cap (0,1)$, 
\begin{align}\label{ev:qbr}
\qb_r(0,\Gamma)=(r^*, \Gamma g(r)),\, \widehat{\qb}_r(0,\Gamma)=(r^*,\Gamma\widehat{g}(r)),\,\widetilde{\qb}_r(0,\Gamma)=(r^*,\Gamma\widetilde{g}(r)).
\end{align}
\end{prop}

\begin{proof}
We get $\pi_2\qb_r^{(i)}(\Gamma)=\Gamma g_i(r)$ from the expression
\begin{align*}
	\qb_{m_n} \circ \qb_{m_{n-1}} \circ \cdots \circ \qb_{m_1}(0,\Gamma)=\left(\frac{Q_{n-1}}{Q_n},\Gamma  g\big([0;m_1,\cdots,m_n]\big) \right)
\end{align*}
where $P_n/Q_n =[0;m_1,\cdots,m_n]$ and $Q_{n-1}/Q_n =[0;m_n,m_{n-1},\cdots,m_1]$.
This finishes the proof.
\end{proof}

\subsection{Branch analogues of modular partitions}
\label{sect:bran:mod:part}

In this section, we introduce the branch versions of modular partition functions, which liaise between the Dirichlet series and the corresponding transfer operators in $\S$\ref{sect:tr:op}.


For a function $\varphi$ on $\Gamma\backslash\GL_2(\Zb)$, let us abuse the notation $\aG_\varphi$ to define a branch analogue of $\aG_\varphi(r)$ in an inductive way such that 
$\aG_\varphi(\qb):=\aG_\varphi(\qb^{(n-1)})+\varphi\circ\pi_2\qb$ for each $\qb\in\Qbb^{\circ n}$, $n\geq 1$ and $\aG_\varphi({\bf id}_{I_\Gamma}):=\bf 0$. Similarly, we define
\begin{align*}
\bG_\psi(\qb)&:=\bG_\psi\big({\qb^{(n-1)}}\big)+\widehat{\psi}\circ\pi_2{\qb},\\
\cG_\psi(\qb)&:=\cG_\psi\big({\qb^{(n-1)}}\big)+\widetilde{\psi}\circ\pi_2{\qb}
\end{align*}
for $\qb\in\Qbb^{\circ n}$, $n\geq 1$, and a function $\psi$ on $\Gamma\backslash\SL_2(\Zb)$. We also set $\bG_\psi({{\bf id}}_{I_\Gamma})=\cG_\psi({{\bf id}}_{I_\Gamma})=\bf 0$. We obtain: 
\begin{prop}\label{cgr:cgqb}
For $r\in \Qb\cap (0,1)$, we have $\aG_\varphi(r)=\aG_\varphi(\qb_r)(0,\Gamma)$, $\bG_\psi(r)=\bG_\psi({\qb}_r)(0,\Gamma)$, and $\cG_\psi(r)=\cG_\psi({\qb}_r)(0,\Gamma)$.
\end{prop}

\begin{proof}
It is immediate from the definitions that  
$\aG_\varphi({\qb})=\sum_{i=1}^n \varphi\circ \pi_2{\qb}^{(i)}$, $\bG_\psi({\qb})=\sum_{i=1}^n \widehat{\psi}\circ \pi_2{\qb^{(i)}}$, and $\cG_\psi({\qb})=\sum_{i=1}^n \widetilde{\psi}\circ \pi_2{\qb^{(i)}}$   for $\qb\in\Qbb^{\circ n}$.  From (\ref{2nd:repn:bu}) and (\ref{ev:qbr}), we obtain the statement.
\end{proof}

\begin{rem}\label{rem:br:bGG}
Since $\qb\in\Qbb^\infty$ is determined completely by $\pi_1\qb$, the value $\bG_\psi(\qb)$ also depends on $\widehat{\qb}$ as well.  In fact,  one can define $\widehat{\qb}^{(i)}$ as a product of $\widehat{\qb}_m$'s and hence define $\bG(\widehat{\qb})$ analogously. Since ${\rm G}[\pi_2{\qb^{(i)}}(v)]={\rm G}u$ is equivalent to $\pi_2\widehat{\qb^{(i)}}(v)\in u$, one can conclude that $\bG_\psi(\qb)=\bG_\psi(\widehat{\qb})$.
\end{rem}

\begin{rem}\label{rem:cr:cGG}
One may want to define $\widetilde{\qb}^{(i)}$ as a product of $\widetilde{\qb}_m$'s and hence to define $\cG(\widetilde{\qb})$ analogously. However, due to the absence of an analogue of (\ref{rel:inv:branch}), the $i$-th part $\widetilde{\qb}^{(i)}$ is not equal to $\widetilde{\qb^{(i)}}$, in general. Instead, using (\ref{tilde:decomp:pq}), we can give a new definition for $\cG(\widetilde{\qb})$, which is equal to $\cG(\qb)$. Since the variable $\cG(\qb)$ is enough for our discussion, we are not going to pursuit this direction. Note that $\cG_\psi(\qb)$ is also completely determined by $\widetilde{\qb}$ as well. Hence, we also set $\cG_\psi(\widetilde{\qb}):=\cG_\psi(\qb)$ for $\qb\in\Qbb^\infty$.

\end{rem}


\section{Transfer operators}
\label{sect:tr:op}

A transfer operator is one of the main tools for studying the statistical properties of trajectories of a dynamical system. Ruelle \cite{ruelle} first made a deep observation that the behavior of trajectories of dynamics can be well explained by spectral properties of the transfer operator. In this section, we define weighted transfer operators corresponding to the modular partition functions and several miscellaneous operators necessary to obtain the desired relations between the Dirichlet series and operators. 

We use a notation that
\begin{align*}
\cha{\Gamma}:=\Gamma\backslash\GL_2(\Zb)\mbox{ or }\Gamma\backslash\SL_2(\Zb),
\end{align*}
according to the symbol $\aG$, $\bG$, or $\cG$ in discussion. We also set
$$X=X_\Gamma:={I}\times \cha{\Gamma}.$$
For a set $A$, let $A^{\cha{\Gamma}}$ be the set of all maps from $\cha{\Gamma}$ to $A$. For a $\psi\in \Cb^{\cha{\Gamma}}$, we set
$$\psi=\eta+i\zeta\mbox{ for }\eta,\zeta\in\Rb^{\cha{\Gamma}}.$$

\subsection{Branch operators}\label{subsec:br:op}

In order to represent the Dirichlet series, we first obtain expressions for the iterations of the transfer operator by studying a component of the operator, which corresponds to the continued fraction expansion of a rational number in $(0,1)$.

For $s\in \Cb$ and $\psi\in\Cb^{\cha{\Gamma}}$, the {\it branch operators} for $\mathfrak{g}$ are defined by
\begin{align*}
\mathcal{B}^\qb_{s,\psi}\Psi&:=\exp \left[ \aG_\psi(\qb) \right] |\partial\pi_1\qb|^s  \Psi \circ\qb,\\
\widehat{\mathcal{B}}^\qb_{s,\psi}\Psi&:=\exp \left[ \bG_\psi({\qb}) \right] |\partial\pi_1 \qb|^s  \Psi \circ\widehat{\qb},\\
\widetilde{\mathcal{B}}^\qb_{s,\psi}\Psi&:=\exp \left[ \cG_\psi({\qb}) \right] |\partial\pi_1 \qb|^s  \Psi \circ\widetilde{\qb} 
\end{align*}
for $\qb\in\Qbb^{\infty}$ and $\Psi\in L^\infty(X)$. Here $\partial\pi_1 \qb$ is the derivative of the first component of $\qb$. We have a multiplicative property:

\begin{prop}\label{prop:digit:iterate}
\begin{enumerate}[leftmargin=*]
\item For $\pb_1,\cdots,\pb_n\in\Qbb^\infty$ and $\qb=\pb_n\circ\cdots\circ\pb_2\circ\pb_1$, we have
$$\Bc^{\qb}_{s,\psi}=\Bc^{\pb_1}_{s,\psi}\circ\Bc^{\pb_2}_{s,\psi}\circ\cdots\circ\Bc^{\pb_n}_{s,\psi}\mbox{ and } \widehat{\mathcal{B}}^\qb_{s,\psi}=\widehat{\mathcal{B}}^{\pb_{1}}_{s,\psi}\circ \widehat{\Bc}^{\pb_{2}}_{s,\psi}\cdots \circ\widehat{\Bc}^{\pb_{n}}_{s,\psi}.$$
\item 	For $\pb,\qb\in\Qbb^{\infty}$  with $\qb$ being of even depth, we have
	$$\widetilde{\mathcal{B}}^{\qb\circ\pb}_{s,\psi}=\widetilde{\mathcal{B}}^{\pb}_{s,\psi}\circ \widetilde{\Bc}^{\qb}_{s,\psi}.$$
	In particular, for $\qb=\pb_n\circ \pb_{n-1}\circ\cdots\circ \pb_{1}$ with $\pb_i\in\Qbb^{\infty}$ of even depth for $1\leq i\leq n-1$ and $\pb_n\in\Qbb\sqcup\Qbb^{\circ 2}$, we have
$$\widetilde{\mathcal{B}}^\qb_{s,\psi}=\widetilde{\mathcal{B}}^{\pb_{1}}_{s,\psi}\circ \widetilde{\Bc}^{\pb_{2}}_{s,\psi}\cdots \circ\widetilde{\Bc}^{\pb_{n}}_{s,\psi}.$$
\end{enumerate}
\end{prop}

\begin{proof}
For the first statement it suffices to show that we get $\Bc_{s,\psi}^{\pb}\circ\Bc_{s,\psi}^{\qb}\Psi = \Bc_{s,\psi}^{\qb\circ\pb}\Psi$ for $\pb,\qb\in\Qbb^{\infty}$. This is just a consequence of applying the identity $\aG_\psi(\mathbf{p})+\psi\circ\pi_2\qb\circ \pb=\aG_\psi(\qb\circ \pb)$ and chain rule. Similarly, we obtain the statements for $\bG$ and $\cG$ with (\ref{rel:inv:branch}) and (\ref{tilde:decomp:pq}).
\end{proof}

From Proposition \ref{cgr:cgqb}, we obtain the relation between a term in the Dirichlet series  and an evaluation of a Branch operator, both of which correspond to a rational number:

\begin{cor}\label{Bqr:Dapsi}
For $r\in\Qb\cap(0,1)$ we obtain
\begin{align*}
\Bc^{\qb_r}_{s,\psi}\Psi(0,\Gamma)&=\exp[\aG_\psi(r)]Q(r)^{-2s}\Psi(r^*,\Gamma g(r)),\\
\widehat{\Bc}^{\qb_r}_{s,\psi}\Phi(0,\Gamma)&=\exp[\bG_\psi(r)]Q(r)^{-2s}\Phi(r^*,\Gamma \widehat{g}(r)),\\
\widetilde{\Bc}^{\qb_r}_{s,\psi}\Phi(0,\Gamma)&=\exp[\cG_\psi(r)]Q(r)^{-2s}\Phi(r^*,\Gamma \widetilde{g}(r)).
\end{align*}	
\end{cor}

For later use, we also record:

\begin{prop}\label{birkhoff:prop}
Setting $\widehat{\Upsilon}_{s,\psi}(x,v):=s\log|x|+\psi(v)$, for $\qb\in\Qbb^{\circ n}$ we get
$$\widehat{\Bc}_{s,\psi}^\qb\Psi=\widehat{\Bc}_{1,0}^\qb\Big[\exp\Big(\sum_{i=0}^{n-1}\Upsilon_{2s-2,\psi}\circ \widehat{\Tb}^i\Big)\Psi\Big].$$
 Setting $\widetilde{\Upsilon}_{s,\psi}(x,v):=s\log|x|+\psi(v)+\psi(\pi_2\widetilde{\Tb}(x,v))$, for $\qb\in\Qbb^{\circ 2n}$, we get
$$\widetilde{\Bc}_{s,\psi}^\qb\Psi=\widetilde{\Bc}_{1,0}^\qb\Big[\exp\Big(\sum_{i=0}^{n-1}\widetilde{\Upsilon}_{2s-2,\psi}\circ \Tb^{2i}\Big)\Psi\Big].$$
\end{prop}

\begin{proof}
The first statement follows from the chain rule and the expression $\bG_\psi({\qb})=\sum_{i=0}^{n-1}\psi\circ\pi_2(\widehat{\Tb}^{i}\circ\widehat{\qb})$ for $\qb\in \Qbb^{\circ n}$. The second one follows from 
\begin{align*}
\cG_\psi({\qb})=\sum_{i=1}^{2n} {\psi}\circ \pi_2\widetilde{\qb^{(i)}}
=\sum_{i=0}^{n-1}(\psi+\psi\circ\pi_2\widetilde{\Tb})\circ(\Tb^{2i}\circ\widehat{\qb})
\end{align*}
for $\qb\in\Qbb^{\circ 2n}$. This finishes the proof.  
\end{proof}

\subsection{Transfer operators}

In this section, to put more emphasis on the iterations, we express the transfer operators in terms of the branch operators rather than as traditional interpretations, so-called the density transformers associated to dynamical systems.

For  $s \in \Cb$ and $\psi \in \Cb^{\cha{\Gamma}}$, the transfer operator for the variable $\aG$ is written as
$\Lc_{s,\psi}=\sum_{\qb\in\Qbb}\Bc^{\qb}_{s,\psi}.$ It is a weighted transfer operator associated to the skewed Gauss map $\Tb$. Using the expression (\ref{inv:branch}), we can rewrite the operators in a more explicit way as ${\Lc}_{s,\psi}\Psi(x,v)= \sum_{m \geq 1} \frac{\exp [ \psi(v \twobytwotiny{0}{1}{1}{m} )]}{(m+x)^{2s}}
\Psi ( \frac{1}{m+x}, v \twobytwotiny{0}{1}{1}{m} )$ 
for $\Psi\in L^\infty(X)$. It can be easily observed that this series converges absolutely for $\Re(s)>1/2$. 
The transfer operator for the variable $\bG$ is defined as
\begin{equation*}
	\widehat{\Lc}_{s,\psi}:={\sum}_{\qb\in\Qbb}\widehat{\Bc}^{\qb}_{s,\psi}.
\end{equation*}
It is a weighted transfer operator associated to $\widehat{\Tb}$, which can be written as 
\[ \widehat{\Lc}_{s,\psi}\Psi(x,v)= \sum_{m \geq 1} \frac{\exp \big[ \widehat{\psi} \big(v\twobytwotiny{0}{1}{1}{m} \big) \big]}{(m+x)^{2s}}
\Psi \left( \frac{1}{m+x}, v\cdot \twobytwo{0}{1}{1}{m} \right). 
\]

Our discussions in $\S$\ref{subsec:br:op}, especially the expression (\ref{tilde:decomp:pq}) and Proposition \ref{prop:digit:iterate}, lead us to define
$${\mathcal{M}}_{s,\psi}:=\sum_{\pb\in\Qbb\circ\Qbb}\widetilde{\Bc}^{\pb}_{s,\psi}.$$
It is a weighted transfer operator associated to $\widehat{\Tb}^2=\widetilde{\Tb}^2$ which equals to the restriction of $\Tb^2$ to $I_\Gamma$. 
In an explicit way, we have 
\begin{align*}
&\mathcal{M}_{s,\psi}\Phi(x,v)=\\
&\sum_{m,n\geq 1}\frac{\exp[{\psi}(v\twobytwotiny{0}{1}{1}{m}{\rm j})+{\psi}(v\twobytwotiny{1}{n}{m}{1+mn})]}{(1+mn+mx)^s}\Phi\left(\frac{1}{n+\frac{1}{m+x}},v\twobytwo{1}{n}{m}{1+mn}\right).
\end{align*}
It can be also easily checked that the series converges absolutely for $\Re(s)>1/2$.

\begin{rem}
For the variable $\cG$, one may want to define the transfer operator for $\cG$ such that
$\widetilde{\Lc}_{s,\psi}:=\sum_{\qb\in\Qbb}\widetilde{\Bc}^{\qb}_{s,\psi}.$ 
However, due to the absence of $\cG$-analogue of (\ref{rel:inv:branch}), it seems unlikely that $L^{\cG}_{\Psi,J}(s,\psi)$ is expressible in terms of $\widetilde{\Lc}_{s,\psi}$. 
\end{rem}

\begin{rem}\label{rem:M:L:sq}
Note that $\mathcal{M}_{s,\psi}$ is not equal to $\widetilde{\Lc}_{s,\psi}^2$. However, we have $\mathcal{M}_{s,{\bf 0}}=\widehat{\Lc}_{s,{\bf 0}}^2$.
\end{rem}

\subsection{Final operators}

The final operator for $\aG$ is defined as
$$\Fc_{s,\psi}:=\sum_{\qb\in{\bf F}}\Bc^{\qb}_{s,\psi}.$$
We can also define a final operator for $\bG$ such that
$$\widehat{\Fc}_{s,\psi}:=\sum_{\qb \in {\bf F}}\widehat{\Bc}^{\qb}_{s,\psi}$$
and the final operator for $\cG$ such that
\begin{align*}
	{\widetilde{\mathcal{F}}}_{s,\psi}:=\sum_{\pb\in(\mathbf{F}\circ\Qbb)\sqcup \mathbf{F}}\widetilde{\Bc}^{\pb}_{s,\psi}.
\end{align*}

\subsection{Interval and Auxiliary operators}\label{subsec:int:op}

To deal with the distributions over intervals, we need an operator, so-called an interval operator, which corresponds to most of rational numbers in the interval; To deal with the missing rational points avoided by the interval operator, we devise an auxiliary operator. 

We first give a preliminary result on structures of intervals. Recall that
$$K^\circ(m_1,\cdots,m_n)=\{[0;m_1,\cdots,m_n+x]\,|\,0< x<1\}.$$
Observe that 
\begin{align}\label{decomp:fund:int}
K^\circ(m_1,\cdots,m_n)=\bigsqcup_{k=1}^\infty K^\circ(m_1,\cdots,m_n,k)\sqcup  \big\{[0;m_1,\cdots,m_n,k]\,\big|\,k>1\big\}.
\end{align}
For an integer $n\geq 1$, we define a collection ${\rm A}_n'$ of open fundamental intervals inductively as follows: 
\begin{enumerate}[leftmargin=*]
\item Let ${\rm A}_1'$ be the collection of (consecutive) open fundamental intervals of depth $1$ that are included in $J$. 
\item Let ${\rm A}_j'$ be defined for $1\leq j\leq n$. Then, ${\rm A}_{n+1}'$ is the collection of open fundamental intervals of depth $n+1$ that are included in $J\setminus \bigcup_{j=1}^n \bigcup_{K\in {\rm A}_j'}K$.  
\end{enumerate}
Obviously ${\rm A}_n'\neq \varnothing$ for some $n$.  

The following is useful when we discuss the convergence of interval and auxiliary operators.

\begin{prop}\label{prop:An}
Let $J=(a,b)\subseteq (0,1)$. Let $a=[0;u_1,u_2,\cdots]$ and $b=[0;v_1,v_2,\cdots]$ be the (possibly finite) continued fraction expansions. 
When $n$ is even, 
$${\rm A}_n'\subseteq\{K^\circ(u_1,\cdots,u_{n-1},k)\,|\,k\geq u_{n}+1\}\cup \{K^\circ(v_1,\cdots,v_{n-1},k)\,|\,1\leq k\leq v_{n}\}.$$
When $n$ is odd,
$${\rm A}_n'\subseteq\{K^\circ(u_1,\cdots,u_{n-1},k)\,|\,1\leq k\leq u_{n}\}\cup \{K^\circ(v_1,\cdots,v_{n-1},k)\,|\,k\geq v_{n}+1\}.$$
\end{prop}

\begin{proof}

Let $c\in (0,1)$ have a (possibly finite) continued fraction expansion $c=[0;m_1,m_2,\cdots]$ and $\frac{P_n}{Q_n}$ be the $n$-th convergent of $c$. It is well-known that for any $n,m\geq 1$, we have $\frac{P_{2n}}{Q_{2n}}\leq c\leq \frac{P_{2m-1}}{Q_{2m-1}}.$

Let $n>1$ be an even integer. The leftmost open fundamental intervals of depths $n-1$ and $n$ that are included in an interval  $\left(c,1\right)$, are
$$K^\circ(m_1,\cdots,m_{n-1})\mbox{ and }K^\circ(m_1,\cdots,m_{n-1},m_{n}+1)\mbox{, respectively}.$$
Note that the left end points of these intervals are $\frac{P_{n-1}}{Q_{n-1}}$ and $\frac{P_n+P_{n-1}}{Q_n+Q_{n-1}}$, respectively. Then, it can be easily seen that the open fundamental intervals of the depth $n$ that are included in an interval $\Big(c,\frac{P_{n-1}}{Q_{n-1}}\Big)$, are $K^\circ(m_1,\cdots,m_{n-1},k)$ for $k\geq m_{n}+1.$

The rightmost open fundamental intervals of depths $n-1$ and $n$ that are included in an interval $(0,c)$, are
$$K^\circ(m_1,\cdots,m_{n-2},m_{n-1}+1)\mbox{ and }K^\circ(m_1,\cdots,m_n), \mbox{ respectively}.$$
Note that the right end points of these intervals are $\frac{P_{n-1}+P_{n-2}}{Q_{n-1}+Q_{n-2}}$ and $\frac{P_n}{Q_n}$, respectively.
Then, the open fundamental intervals of the depth $n$ that are included in an interval $\Big(\frac{P_{n-1}+P_{n-2}}{Q_{n-1}+Q_{n-2}},c\Big)$, are $K^\circ(m_1,\cdots,m_{n-1},k)$ for $1\leq k\leq m_{n}.$
In sum, we obtain the statements.

The arguments for odd $n$ are similar.
\end{proof}

For an open fundamental interval $K=K^\circ(m_1,\cdots,m_n)$, we define 
$${\qb}_K:=\qb_{m_n}\circ\cdots\circ \qb_{m_1}.$$
Later, we shall only need fundamental intervals of even length (see Remark \ref{rem:even:depth}).  Having in mind
the identity (\ref{decomp:fund:int}), let us set
$${\rm A}_n:=\begin{cases}
\phantom{blaaaaank}{\rm A}_n'&\mbox{ if $n$ is even}\\
\{(K,k)\,|\, K\in {\rm A}_n',k\geq 1\}&\mbox{ if $n$ is odd}
\end{cases}
$$
where $(K,k):=K^\circ(m_1,\cdots,m_n,k)$ for $K=K^\circ(m_1,\cdots,m_n)$.
Then, we set
$${\bf Q}_J:=\bigg\{{\qb}_{K}\,\bigg|K\in \bigcup_{n\geq1}{\rm A}_n\bigg\}.$$
In particular, we have $\mathbf{Q}_K=\{\qb_K\}$ for an open fundamental interval $K$ if $K$ is of even depth and $\Qbb_K=\{\qb_k\circ\qb_K\,|\,k\geq 1\}=\Qbb\circ\qb_K$ if $K$ is of odd depth. Therefore, all the branches in $\Qbb_J$ are of even depth.

Let us set
$${U}_J:=\bigcup_{n\geq1}\bigcup_{K\in {\rm A}_n}K\mbox{ and }{V}_J:=J\setminus U_J.$$
Note that $V_J$ is a countable set of rational numbers that consists of the endpoints of each $K\in\bigcup_{n\geq 1}{\rm A}_n$  except the boundaries of $J$.
Let us set
$$\partial \Qbb_J:=\{\qb_r\,|\,r\in V_J\}.$$

Let $s\in\Cb$ and $\eta\in \Cb^{\cha{\Gamma}}$ with sufficiently large $\Re(s)$ and small $\max|\eta|$.
We define an {\it interval operator} and an auxiliary operator for $\aG$ as
\begin{align}\label{DJ:s:w}
\mathcal{D}_{s,\psi}^{J}:= \sum_{{\bf q} \in \mathbf{Q}_J} \Bc^{\qb}_{s,\psi}\mbox{ and }\mathcal{K}_{s,\psi}^J:=\sum_{\qb\in \partial \Qbb_J}\Bc^{\qb}_{s,\psi},\mbox{ respectively}.
\end{align}
Similarly, we define an {\it interval operator}  and an auxiliary operator for $\bG$ as
\begin{align}\label{DJhat:s:w}
\widehat{\mathcal{D}}_{s,\psi}^{J}:= \sum_{{\bf q} \in \mathbf{Q}_J} \widehat{\Bc}^{\qb}_{s,\psi}\mbox{ and }\widehat{\mathcal{K}}_{s,\psi}^J:=\sum_{\qb\in \partial \Qbb_J}\widehat{\Bc}^{\qb}_{s,\psi},\mbox{ respectively}.
\end{align}
We also define an {\it interval operator} and an auxiliary operator for $\cG$ as
\begin{align}\label{DJtilde:s:w}
	\widetilde{\mathcal{D}}_{s,\psi}^{J}:= \sum_{\qb\in\Qbb_J} \widetilde{\Bc}^\qb_{s,\psi}\mbox{ and }\widetilde{\mathcal{K}}^J_{s,\psi}:=\sum_{\qb\in\partial \Qbb_J}\widetilde{\Bc}^\qb_{s,\psi},\mbox{ respectively}.
\end{align}
These operators are well-defined for $\Re(s)>\frac{1}{2}$ as follows:

\begin{prop}\label{prop:DJ:conv}
The series in (\ref{DJ:s:w}), (\ref{DJhat:s:w}), and (\ref{DJtilde:s:w}) are uniformly convergent for $\Re(s)\geq\sigma_0$ for any $\sigma_0>\frac{1}{2}$. Hence, they are analytic in the region $\Re(s)>\frac{1}{2}$. 
\end{prop}

\begin{proof}
We first consider the interval operator in (\ref{DJ:s:w}). Let $J=(a,b)$. Let $\frac{p_n}{q_n}$ and $\frac{P_n}{Q_n}$ be the convergents of $a$ and $b$, respectively. From Proposition \ref{prop:An}, we obtain 
 for $\Re(s)=\sigma>\frac{1}{2}$ and a bounded function $\Psi$ that
\begin{align*}
\sum_{K\in {\rm A}_n}\Bc_{s,\psi}^{\qb_K}\Psi\ll&\sum_{k=1}^\infty \frac{1}{(q_{n-1}k+q_{n-2})^{2\sigma}}+\frac{1}{(Q_nk+Q_{n-1})^{2\sigma}}\\
\ll& \frac{1}{2\sigma-1}\left(\frac{1}{q_n^{2\sigma}}+\frac{1}{Q_n^{2\sigma}}\right).
\end{align*}
Hence, we obtain
$$\mathcal{D}_{s,\psi}^J\Psi\ll \frac{1}{2\sigma-1}\left(\sum_{n}\frac{1}{q_n^{2\sigma}}+\frac{1}{Q_n^{2\sigma}}\right).$$
The latter sum is a finite sum, or a convergent series for $\sigma>\frac{1}{2}$ since $q_n,Q_n\geq n$. 

For (\ref{DJhat:s:w}), it suffices to observe that for a bounded function $\Phi$ and a branch $\qb$, we have
$$\|\widehat{\mathcal{B}}^\qb_{s,\psi}\Phi\|_0\leq \|{\mathcal{B}}^\qb_{\sigma,\eta_0}{\bf 1}\|_0\cdot \|\Phi\|_0$$
where $\eta_0$ is the constant function $\max|\Re\psi|$ and $\mathbf{1}=1\otimes 1$. By following the previous calculation for $\mathcal{D}^J_{s,\psi}$, we obtain the statement. A similar argument is applied to (\ref{DJtilde:s:w}).

For the auxiliary operators, we observe that for $r\in V_J$, there are at most two $K$, $K'$ in $\cup_n {\rm A}_n$ such that $r$ is a common endpoint of them.  Then, $\qb_r=\qb_K$ or $\qb_{K'}$. Hence, we can obtain the statement from an observation that $\|\mathcal{K}^J_{s,\psi}\Psi\|_0\leq 2\|\mathcal{D}^J_{\sigma,\eta}|\Psi|\|_0$ for a bounded function $\Psi$. Discussion for the other operators are similar. 
\end{proof}

\begin{rem}\label{rem:even:depth}
	One can define the operators $\mathcal{D}^J_{s,\psi}$, $\widehat{\mathcal{D}}^J_{s,\psi}$, $\mathcal{K}^J_{s,\psi}$, $\widehat{\mathcal{K}}^J_{s,\psi}$ by using ${\rm A}_n'$ as well. However, we have no choice but to use ${\rm A}_n$ to define the operators for $\cG$ since Proposition \ref{prop:digit:iterate} holds only for branches of even depth.
\end{rem}


\subsection{Key relations for Dirichlet series}

In this subsection, we present an underlying connection between the transfer operators of the skewed Gauss dynamical systems and the Dirichlet series. 

Recall $\psi=\eta+i\zeta$. Observe that for $\Re(s)>1+\frac{c\max|\eta|}{2}$ (for the definition of $c$, see (\ref{perron:am})), the Dirichlet series for $\aG$ can be written as
\begin{align}\label{L:Psi:series:rep}
L^\aG_{\Psi,J}(s,\psi) &= \sum_{r \in \Qb\cap J} \frac{  \Psi(r^*,\Gamma g(r))\exp(\aG_\psi(r))}{Q(r)^{s}}.
\end{align}
One can easily obtain a similar expressions for $\bG$ and $\cG$. 

In Theorem \ref{bound:dolgopyat}, it will be shown that quasi-inverses $(\mathcal{I}-\Lc_{s,\psi})^{-1}$, $(\mathcal{I}-\widehat{\Lc}_{s,\psi})^{-1}$, and $(\mathcal{I}-\mathcal{M}_{s,\psi})^{-1}$ are well-defined as geometric series of the operators when $(\sigma,\eta)$ is close enough to $(1,{\bf 0})$; in the remaining part of this section, let us assume this condition. Then, a portion of the sum (\ref{L:Psi:series:rep}) can be described as follows:

\begin{prop}\label{iter:repr}
For an open fundamental interval $K$ and a bounded function  $\Psi$ on ${I}\times\Gamma\backslash\GL_2(\Zb)$, we obtain 
\[
\mathcal{D}^K_{s,\psi}(\mathcal{I}-\Lc_{s,\psi})^{-1}\mathcal{F}_{s,\psi} \Psi(0,\Gamma)=\sum_{r\in \Qb\cap K}\frac{\Psi(r^*,\Gamma g(r))\exp(\aG_\psi(r))}{Q(r)^{2s}}.
\]
For a bounded function $\Phi$ on ${I}\times\Gamma\backslash\SL_2(\Zb)$, we also have
\begin{align*}
\widehat{\mathcal{D}}^K_{s,\psi}(\mathcal{I}-\widehat{\Lc}_{s,\psi})^{-1}\widehat{\mathcal{F}}_{s,\psi} \Phi(0,\Gamma)&=\sum_{r\in \Qb\cap K}\frac{\Phi(r^*,\Gamma \widehat{g}(r))\exp(\bG_\psi(r))}{Q(r)^{2s}}\mbox{ and}\\
\widetilde{\mathcal{D}}^K_{s,\psi}(\mathcal{I}-\mathcal{M}_{s,\psi})^{-1}{\widetilde{\mathcal{F}}}_{s,\psi} \Phi(0,\Gamma)&=\sum_{r\in \Qb\cap K}\frac{\Phi(r^*,\Gamma \widetilde{g}(r))\exp(\cG_\psi(r))}{Q(r)^{2s}}.
\end{align*}
\end{prop}

\begin{proof}
Let $n\geq 0$ be an integer. From Proposition \ref{prop:digit:iterate}, we obtain
\begin{align*}
\mathcal{D}^K_{s,\psi}\Lc_{s,\psi}^{n}\mathcal{F}_{s,\psi} \Psi(0,\Gamma)=\sum_{\qb\in \mathbf{F}\circ  \Qbb^{\circ n} \circ \mathbf{Q}_K}\Bc_{s,\psi}^\qb\Psi(0,\Gamma).
\end{align*}
Since $r\mapsto \qb_r$ is a one-to-one correspondence between $K\cap \Qb$ and $\mathbf{F}\circ\Qbb^\infty\circ\Qbb_K$, and $Q(r)=Q(r^*)$, we obtain the statement by Corollary \ref{Bqr:Dapsi}. A proof for the second statement is similar. 
For the third statement, with Proposition \ref{prop:digit:iterate},  we have
$$\widetilde{\mathcal{D}}^K_{s,\psi}\mathcal{M}_{s,\psi}^{n}\widetilde{\mathcal{F}}_{s,\psi} \Phi(0,\Gamma)={\sum}_{\qb}\widetilde{\Bc}_{s,\psi}^{\qb}\Phi(0,\Gamma)$$
where $\sum_\qb$ means the summation over $\qb\in \mathbf{F}\circ(\Qbb^{\circ 2n}\circ\Qbb_K)\bigsqcup (\mathbf{F}\circ\Qbb)\circ (\Qbb^{\circ 2n}\circ\Qbb_K).$  By the correspondence $\Qb\cap K\rightarrow \mathbf{F}\circ \Qbb^{\infty} \circ \Qbb_{K}$, we obtain the statement from Corollary \ref{Bqr:Dapsi}. \end{proof}

Finally, we settle the following explicit expressions for the Dirichlet series:

\begin{thm} \label{keyrelation}
Let $\Psi$ be a bounded function on ${I}\times\Gamma\backslash\GL_2(\Zb)$, $\Phi$ a bounded function on $I_\Gamma$, and $\Re(s)>1+\frac{c\max|\eta|}{2}$. We have
\begin{align*}
L^{\aG}_{\Psi,J}(2s;\varphi)&= \mathcal{K}_{s,\psi}^J\Psi(0,\Gamma)+\mathcal{D}_{s,\psi}^J(\mathcal{I}-\mathcal{L}_{s,\psi})^{-1} \Fc_{s,\psi}\Psi(0,\Gamma),\\ 
L^\bG_{\Phi,J}(2s;\psi)&= \widehat{\mathcal{K}}_{s,\psi}^J\Phi(0,\Gamma)+\widehat{\mathcal{D}}_{s,\psi}^J(\mathcal{I}-\widehat{\mathcal{L}}_{s,\psi})^{-1} \widehat{\Fc}_{s,\psi}\Phi(0,\Gamma),\mbox{ and}\\
L^{\cG}_{\Phi,J}(2s;\psi)&=\widetilde{\mathcal{K}}_{s,\psi}^J\Phi(0,\Gamma)+\widetilde{\mathcal{D}}_{s,\psi}^J(\mathcal{I}-\mathcal{M}_{s,\psi})^{-1}\widetilde{\mathcal{F}}_{s,\psi}\Phi(0,\Gamma). 
\end{align*}
\end{thm}

\begin{proof}
From the correspondence ${V}_J\rightarrow \partial\Qbb_J$,  similarly as the last proof, we get
\begin{align*}
\mathcal{K}_{s,\psi}^{J}\Psi(0,\Gamma)=\sum_{r\in V_J}\frac{ \Psi(r^*, \Gamma g(r))\exp(\aG_\psi(r))}{Q(r)^{2s}}.
\end{align*}
We also have similar expressions for $\bG$ and $\cG$. Note that there are one-to-one correspondences 
$\Qb\cap U_J\rightarrow \big(\mathbf{F}\circ \Qbb^{\infty} \circ \mathbf{Q}_J\big)$ and $V_J\rightarrow \partial \Qbb_J$,  given by $r\mapsto \qb_r.$
Now Proposition \ref{iter:repr} and the disjoint union $J=U_J \bigsqcup V_J$ enable us to conclude the proof.
\end{proof}


\begin{rem}\label{atkin:lehner:rem}
Instead of using the interval operator, one might want to choose $\Psi$ as a product of a smooth approximation of $J$ and the function $\varphi$ to study $\Omega_{M,\varphi,J}$. A problem is that there is no known relation between $\cG_\psi(r)$ and $\cG_\psi(r^*)$, in general. However, there is a relation between $\mG_f^\pm(r)$ and $\mG_f^\pm(r^*)$, so-called the Atkin--Lehner relation (See Mazur--Rubin \cite{mazur:rubin}). We had tried this direction and were only able to obtain a partial and unsatisfactory result. One advantage of introducing the interval operator is that the Atkin--Lehner relation is dispensable.
\end{rem}

\section{Spectral analysis of transfer operator}\label{sect:spec:transf}

In this section, we present a dynamical analysis on the transfer operator associated to the skewed Gauss dynamical system.

\subsection{Basic settings and properties}\label{subsec:basic:prop}

For the remaining part of the paper, we write $s:=\sigma+it \in \Cb$ with $\sigma,t\in\Rb$. 
In order to discuss all the modular partition functions simultaneously, let us set 
$$\mathfrak{g}:=\aG,\bG,\mbox{ or }\cG.$$
We use the symbol $\Hc_{s,\psi}$ to represent
$$\Hc_{s,\psi}:=\Lc_{s,\psi},\widehat{\Lc}_{s,\psi},\mbox{or }{\mathcal{M}}_{s,\psi}$$
according to the choice of $\mathfrak{g}$. Since $\bG_\psi$ and $\cG_\psi$ are also functions on the inverse branches of $\widehat{\Tb}$ and $\widetilde{\Tb}^2$, respectively (Remark \ref{rem:br:bGG} and \ref{rem:cr:cGG}), let us use the symbol $\bf B$ to represent the branches:
\begin{align*}
\mathbf{B}:=\Qbb,\widehat{\Qbb},\mbox{ or }\widetilde{\Qbb^{\circ 2}}
\end{align*}
according to the choice of $\mathfrak{g}$. Note that for $\Psi\in C^1({X})$, the transfer operator can be written as
$$\Hc_{s,\psi}\Psi=\sum_{\pb\in\mathbf{B}}\exp[\mathfrak{g}_\psi(\pb)]|\partial\pi_1\pb|^{s}\Psi\circ\pb.$$
Note here that $\mathfrak{a}_\psi(\qb)=\psi\circ\pi_2\qb$ and $\mathfrak{b}_\psi(\qb)={\psi}\circ\pi_2\widehat{\qb}$ for a $\qb\in\Qbb$. For $\pb=\pb_2\circ\pb_1\in \widetilde{\Qbb^{\circ 2}}$, we get $\cG_\psi(\pb)={\psi}\circ\pi_2\widetilde{\pb}_1+\psi\circ\pi_2\widetilde{\pb}.$


Recall that a space on which $\Hc_{s,\psi}$ acts is defined as $C^1({{X}}) = \{ \Psi: {{X}} \rightarrow \Cb \ | \ \Psi \ \mbox{and } \partial\Psi \ \mbox{are \emph{continuous}} \}$ 
where the derivative $\partial$ on $C^1({{X}})$ is defined by the partial derivative with respect to the first coordinate $\partial\Psi(x,v) := \frac{\partial}{\partial x} \Psi(x,v) .$ The space $C^1({{X}})$ is just a finite union of $C^1(I)$ and its elements are the linear combinations of tensor type $(f \otimes g)(x,v):=f(x)g(v)$ for a function $f$ on $I$ and a function $g$ on the set of right cosets of $\Gamma$. It is a Banach space with the norm 
$$\|\Psi\|_1 =\|\Psi\|_0 + \| \partial \Psi\|_0.$$ It is easy to show that the operator acts boundedly: For $\Psi \in C^1({{X}})$ and $(\sigma,\eta)$ in a small compact real neighborhood $B\subseteq \Rb\times\Rb^{\cha{\Gamma}}$ of $(1, \bf 0)$, we have
\begin{align}\label{bnded:op:H}
\| \Hc_{\sigma,\eta}\Psi\|_1 \ll_B  \|\Psi\|_1.
\end{align}

\subsection{Geometric properties of skewed Gauss dynamical system}

We study the spectrum of our transfer operator in the later section. Remark that this will be settled by the metric properties of the set $\Qbb$ of inverse branches of $\Tb$ based on the following geometric properties of the Gauss dynamical system.

Let us set the {\it contraction ratio} as
$$\rho:=\begin{cases}
1/2&\mbox{ if }\mathfrak{g}=\aG\mbox{ or }\bG\\
1/4&\mbox{ if }\mathfrak{g}=\cG
\end{cases}.
$$
We remark that $\rho$ from Proposition \ref{main:dynamics} is given by the contraction ratio.

\begin{prop}[Baladi-Vall\'ee \cite{bv}] \label{prop:gaussmap}
	For any branch $\qb\in{\bf B}^{\circ n}$ for $n \geq 1$, we have:
	\begin{enumerate}[leftmargin=*]
		\item \emph{(Uniform contraction)}  $$\|\partial \pi_1\qb\|_0 \ll \rho^n.$$
		\item \emph{(Bounded distortion)} $$\ds\left\| \frac{\partial^2\pi_1\qb}{\partial \pi_1\qb} \right\|_0 \ll1.$$
	\end{enumerate}
\end{prop}  

\begin{proof}
The is a mere translation of results in Baladi-Vall\'{e}e \cite[Section 2.2]{bv} or Naud \cite[Lemma 3.5]{naud} in terms of our notations.
\end{proof}

We recall the UNI property for the Gauss dynamical system established by Baladi--Vall\'ee \cite{bv}. For any $n \geq 1$ and for two inverse branches $\pb$ and $\qb$ of $\Tb^n$, the temporal distance is defined by
\begin{equation*}
\Delta(\pb,\qb):=\inf_{{X}}|\partial\Pi_{\pb,\qb}| 
\end{equation*}
where the map $\Pi_{\pb,\qb}$ on ${X}$ is given by 
\begin{equation*}
\Pi_{\pb,\qb}:=\log \frac{|\partial\pi_1\pb|}{|\partial\pi_1\qb|}.
\end{equation*}
What they showed can be written as:

\begin{prop}[Baladi--Vall\'ee {\cite[Lemma 6]{bv}}] \label{uni}
The skewed Gauss dynamical systems satisfy the UNI condition:
\begin{enumerate}[leftmargin=*]
\item\label{uni:a} Let $m$ be the product of Lebesgue measure on ${I}$ and the counting measure on the right cosets of $\Gamma$. For $0<a<1$, $n \geq 1$, and $\pb \in \Qbb^{\circ n}$, we have  
 \[ m \bigg( \bigcup_{\qb\in\Qbb^{\circ n} \atop \Delta(\pb,\qb) \leq \rho^{an}} \qb(X)   \bigg) \ll \rho^{an} \] 
 where the implicit constant is independent of $a$, $n$, and $\pb$.
\item\label{uni:b} One has 
$$\sup_{\pb,\qb\in\Qbb^{\infty}}\|\partial^2\Pi_{\pb,\qb}\|_0 < \infty.$$

\end{enumerate}
\end{prop}



\subsection{Dominant eigenvalue and spectral gap of positive operator}

We describe the spectrum of positive transfer operator $\Hc_{\sigma,\eta}$ acting on ${C^1}({{X}})$. 
We begin by stating the following sufficient condition for quasi-compactness due to Hennion.

\begin{thm}[Hennion \cite{hennion}] \label{hennion}
Let $\Hc$ be a bounded operator on a Banach space $X$, endowed with two norms $\|\cdot\|$ and $\|\cdot\|'$ satisfying $\Hc(\{\phi \in X: \|\phi \| \leq 1 \})$ is conditionally compact in $(X,\|\cdot\|')$. Suppose that there exist two sequences of real numbers $r_n$ and $t_n$ such that for any $n \geq 1$ and $\phi \in X$, one has the inequalities
\begin{equation} \label{eq:hennion}
\|\Hc^n \phi \| \leq t_n \|\phi\|'+ r_n \|\phi\|.
\end{equation}
Then the essential spectral radius of $\Hc$ is at most $\liminf_{n \rightarrow \infty} r_n^{1/n}$.
\end{thm}

Inequalities of the form (\ref{eq:hennion}) are often called the \emph{Lasota--Yorke} type in the theory of dynamical system. This enables us not only to show quasi-compactness of $\Hc_{\sigma,\eta}$ on ${C^1}({{X}})$, but also to obtain an explicit estimate for the iterates, which is a crucial ingredient for the uniform spectral bound in $\S$\ref{sect:dolgopyat}.

The main estimate for norms of our transfer operators is controlled by geometric behavior of inverse branches of the Gauss map. For example, we have:

\begin{prop} \label{lasota:yorke}
Let $B$ be the neighborhood of $(1,{\bf 0})$ for (\ref{bnded:op:H}). Choosing $B$ small enough, if needed, for $(\sigma,\eta)\in B$: 
\begin{enumerate}[leftmargin=*]
\item \label{ineq:partial:Hc} For any $n \geq 1$ and $\Psi \in {C^1}({{X}})$, we have
\[ \| \partial \Hc_{\sigma,\eta}^n \Psi \|_0 \ll_{B}  |\sigma| \|\Psi\|_0 +\rho^n \|\partial \Psi\|_0 .\]
\item The operator $\Hc_{\sigma,\eta}$ on ${C^1}({{X}})$ is quasi-compact.
\end{enumerate}
\end{prop}

\begin{proof} 

By Proposition \ref{prop:digit:iterate}, we notice that for $n \geq 1$, the iteration is of the form
\[ \Hc_{\sigma,\eta}^n \Psi = \sum_{\pb \in \mathbf{B}^{\circ n}} \exp[\mathfrak{g}_\eta(\pb)] |\partial\pi_1\pb|^\sigma  \Psi \circ \pb. \]
Then the differentiation gives 
\begin{align*}
\partial \Hc_{\sigma,\eta}^n \Psi &= \sum_{\pb} \exp[\mathfrak{g}_\eta(\pb)] \left(\sigma |\partial\pi_1\pb|^{\sigma} \frac{|\partial^2\pi_1\pb|}{|\partial\pi_1\pb|} \cdot \Psi \circ \pb + |\partial\pi_1\pb|^{\sigma} \partial \pi_1\pb \cdot \partial \Psi \circ \pb \right).
\end{align*}
By the uniform contraction in Proposition \ref{prop:gaussmap}, we obtain the first statement.

Notice that the embedding of $({C^1}({{X}}), \|\cdot\|_1)$ into $({C^1}({{X}}), \|\cdot\|_0)$ is a compact operator, since $\Gamma$ is of finite index in $\SL_2(\Zb)$. Hence by Theorem \ref{hennion} and the first statement, we have the second one. 
\end{proof}

We collect spectral properties of $\Hc_{\sigma,\eta}$. We mainly refer to Baladi \cite[Theorem 1.5]{baladi} for a general theory.

\begin{prop} \label{op:dominant}
Let $B$ be as before. For $(\sigma, \eta) \in {B}$, let us set
$$\lambda_{\sigma,\eta}:=\lim_{n\rightarrow\infty}\|\Hc_{\sigma,\eta}^n{\bf 1}\|_0^{\frac{1}{n}}.$$
Then, we have:
\begin{enumerate}[leftmargin=*]
\item The value $\lambda_{\sigma,\eta}$ is the spectral radius of $\Hc_{\sigma,\eta}$ on $C^1(X)$ with $\|\cdot\|_1$.
\item \label{pos:eigen}The operator $\Hc_{\sigma,\eta}$ has a positive eigenfunction $\Phi_{\sigma,\eta}$ with the eigenvalue $\lambda_{\sigma,\eta}$. In particular, $\Phi_{1,\bf 0}(x,v)=\frac{1}{\log 2(x+1)}$ and $\lambda_{1,\bf 0}=1$.
\item The eigenvalue $\lambda_{\sigma,\eta}$ is of maximal modulus, positive, and simple.
\item There is an eigenmeasure $\mu_{\sigma,\eta}$ of the adjoint of $\Hc_{\sigma,\eta}$ such that it is a Borel probability measure with $\int_{{X}}\Phi_{\sigma,\eta}d\mu_{\sigma,\eta}=1$ after normalizing $\Phi_{\sigma,\eta}$ suitably. In particular, $\mu_{1,\bf 0}$ is equivalent to Lebesgue measure.
\end{enumerate}
\end{prop}

\begin{proof}
Even though it is almost same as discussion in Baladi \cite[Theorem 1.5]{baladi}, let us give a sketch of proof for reader's convenience. See also Baladi--Vall\'ee \cite{bv} and Parry--Pollicott \cite{parry:pollicott2} for more details.

(1) Using Proposition \ref{lasota:yorke}.(\ref{ineq:partial:Hc}), it can be shown that the spectral radius of $\Hc_{\sigma,\eta}$ on $C^1(X)$ with $\|\cdot\|_1$ is less than or equal to $\lambda_{\sigma,\eta}$. We also have
\[ 
\lim_{n \rightarrow \infty} \|\Hc_{\sigma,\eta}^n\|_1^{1/n} \geq  \lim_{n \rightarrow \infty}   \|\Hc_{\sigma,\eta}^n \mathbf{1} \|_1^{1/n}  \geq  \lim_{n \rightarrow \infty} \|\Hc_{\sigma,\eta}^n \mathbf{1} \|_0^{1/n},    
\]
which implies the statement.

(2) Let $\lambda_0=\lambda_{\sigma,\eta}$, $\lambda_1$, $\cdots$, $\lambda_\ell$ be the distinct eigenvalues of maximal modulus. Then, by spectral projection, there exist $\Psi$, $\Psi_j\in C^1({X})$ such that ${\bf 1}=\Psi+\sum_{j=0}^\ell\Psi_j$, $\|\Hc_{\sigma,\eta}^n\Psi\|_1=o(\lambda_{0}^n)$, and $\Psi_j$ is in the generalized eigenspace for $\lambda_j$. By observing the Jordan normal form of $\Hc_{\sigma,\eta}$ on the generalized eigenspace, it can be shown that there exists an integer $k>0$ such that for each $j=0,\cdots,\ell$, the following limits exist
$$\lim_{n\rightarrow\infty}\frac{1}{\lambda_j^n n^k}\Hc_{\sigma,\eta}^n\Psi_j=:\Phi_j.$$
Then $\Hc_{\sigma,\eta}\Phi_j=\lambda_j\Phi_j$ for each $j=0,\cdots,\ell$  and at least one of $\Phi_j$ is not trivial. In sum, we obtain
$$0\leq \frac{\Hc_{\sigma,\eta}^n{\bf 1}}{\lambda_{0}^nn^k}=o(1)+\sum_{j}\left(\frac{\lambda_j}{\lambda_{0}}\right)^n\frac{\Hc_{\sigma,\eta}^n\Psi_j}{\lambda_j^n n^k}.$$
From this inequality, using a version of orthogonality relation: 
$$\frac{1}{M}\sum_{n=1}^M\left(\frac{\lambda_j}{\lambda_0}\right)^n=\begin{cases}
~~1&\mbox{ if }\lambda_j=\lambda_0\\
o(1)&\mbox{ otherwise}
\end{cases},
$$
we can deduce that the function $\Phi_0$ is non-negative and non-trivial.

Suppose that $\Phi_0(x_0,v_0)=0$ for some $(x_0,v_0)\in {X}$. For each $n\geq 1$, we have
$$0=\Hc_{\sigma,\eta}^n\Phi_0(x_0,v_0)=\sum_{\pb\in\mathbf{B}^{\circ n}}\exp[\mathfrak{g}_\eta(\pb)]|\partial\pi_1\pb(x_0,v_0)|^\sigma \Phi_0(\pb(x_0,v_0)).$$
Since the weights $\exp[\mathfrak{g}_\eta(\pb)]|\partial\pi_1\pb(x_0,v_0)|^\sigma$ are positive, $\Phi_0(\pb(x_0,v_0))=0$ for all $\pb\in\mathbf{B}^{\circ n}$.  The density result of Proposition \ref{op:alg:simple}.(\ref{mixing:dense})  and continuity of $\Phi$ result in a contradiction. Hence, we obtain the statement with $\Phi_{\sigma,\eta}:=\Phi_0$.

It is a classical result that $\Phi_{1,\bf 0}(x,v)=\frac{1}{\log 2(x+1)}$ is an eigenfunction of $\Lc_{1,\bf 0}$ and $\widehat{\Lc}_{1,\bf 0}$ with the eigenvalue $1$ of maximal modulus. For $\mathcal{M}_{s,\wbb}$, recall that $\mathcal{M}_{1,\bf 0}=\widehat{\Lc}_{1,\bf 0}^2$.

(3) The first two claims come from the definition of $\lambda_{\sigma,\eta}$. For the geometric simplicity, let $\Psi$ be  a eigenfunction for $\lambda_0$ and set $t:=\min\{\frac{\Psi(x,v)}{\Phi_0(x,v)}\,|\,(x,v)\in{X}\}$. Since $t=\frac{\Psi(x_0,v_0)}{\Phi_0(x_0,v_0)}$ for some $(x_0,v_0)\in {X}$ by the continuity, we can conclude that $\Psi=t\Phi_0$ using the last density argument.

For the algebraic simplicity, let us assume that for a non-trivial $\Psi\in C^1({X})$, one has $(\Hc_{\sigma,\eta}-\lambda_0\mathcal{I})^2\Psi={\bf 0}$ and $\Phi:=(\Hc_{\sigma,\eta}-\lambda_0\mathcal{I})\Psi\neq\bf 0$. Then, we have $\Hc_{\sigma,\eta}^n\Psi=\lambda_0^n\Psi+n\lambda_0^{n-1}\Phi$, with which we deduce a contradiction from 
$$\|\Hc_{\sigma,\eta}^n\Psi\|_0\leq\|\Hc_{\sigma,\eta}^n\Phi_{\sigma,\eta}\|_0\|\Phi_{\sigma,\eta}^{-1}\Psi\|_0=\lambda_0^n\|\Phi_{\sigma,\eta}^{-1}\Psi\|_0.$$


(4) Extending the functionals on the eigenspace for $\lambda_{\sigma,\eta}$ to $C^1({X})$ using spectral projection, we obtain a positive eigen Radon measure of the adjoint, which corresponds to a Borel probability measure on ${X}$. Normalizing suitably, we obtain the statement.
\end{proof}

We show the uniqueness of the eigenvalue of maximal modulus.

\begin{prop} \label{ruelle:perron}
Let $\mathfrak{g}=\bG$ or $\cG$. Then, for $(\sigma,\eta)\in B$,
the eigenvalue $\lambda_{\sigma,\eta}$ is unique, i.e., $\Hc_{\sigma,\eta}$ has no other eigenvalue on the circle of radius $\lambda_{\sigma,\eta}$. 
\end{prop}

\begin{proof}
Since it is almost same as in Baladi \cite[Theorem 1.5.(5)]{baladi}, let us give a sketch of the proof. First we need to show that 
\begin{align}\label{iter:limit:temp}
\lim_{n\rightarrow \infty}\left\|\frac{1}{\lambda_{\sigma,\eta}^n}\Hc_{\sigma,\eta}^n\Psi-\Phi_{\sigma,\eta}\int_{{X}}\Psi d\mu_{\sigma,\eta}\right\|_0=0
\end{align}
for all $\Psi\in C^1({X})$. Our version of the density result in Proposition \ref{op:alg:simple}.(\ref{mixing:dense}) together with Proposition \ref{prop:cong:Tadm} enables us to show that a continuous accumulation point of the sequence $\lambda_{\sigma,\eta}^{-n}\Hc_{\sigma,\eta}\Psi$ is actually the function $\Phi_{\sigma,\eta}\int_{{X}}\Psi d\mu_{\sigma,\eta}$. The limit is verified after applying the Arzel\`a--Ascoli theorem to an equicontinuous family $\{\lambda_{\sigma,\eta}^{-n}\Hc_{\sigma,\eta}^n\Psi\,|\,n\geq 1\}$. Then the desired statement of proposition follows easily from the expression (\ref{iter:limit:temp}).
\end{proof}

\hide{----
\begin{rem}
As $\Tb$ is not topological mixing, one can expect easily that the above proof is no longer valid for $\Lc_{s,\eta}$. Nonetheless, based on numerical calculations, we speculate that the uniqueness of the eigenvalue is still valid.
\end{rem}
-------}

\section{Dolgopyat--Baladi--Vall\'{e}e bound in vertical strip}\label{sect:dolgopyat}

Main objective in this section is the uniform polynomial bound for the iterations of $\Hc_{s,\psi}$, namely the Dolgopyat--Baladi--Vall\'ee estimate. 
Consequently, along with the results from $\S$\ref{sect:spec:transf}, we complete the proof of Proposition \ref{main:dynamics} at the end of this section.

 Dolgopyat \cite{dolgopyat} first established the result of such type for the plain transfer operators associated to certain Anosov systems with a finite Markov partition, which depends on a single complex parameter $s$. Let us roughly overview his ideas for the proof:
\begin{enumerate}
\item Due to the spectral properties of transfer operator, the main estimate can be reduced to $L^2$-norm estimate, which is decomposed into a sum of oscillatory integrals over the pairs of the inverse branches. This sum is divided into two parts. 

\item Relatively \emph{separated
pairs} of inverse branches consist in one part, in which the oscillatory integrals can be simply dealt with the use of Van der Corput Lemma.

\item In order to control the other part that consists of \emph{close pairs}, the dynamical system must satisfy the Uniform Non-Integrability(UNI) condition, which explains that there are a few such pairs.

\end{enumerate}

This groundbreaking work has been generalised into other dynamical systems. In particular, Baladi--Vall\'ee \cite{bv} modified the UNI condition to obtain Dolgopyat-type estimate for the weighted transfer operator associated to the Gauss map with countably many inverse branches. Our proof goes in a similar way. In fact, focusing on the classical continued fractions, we mainly follow a more concise exposition of Naud \cite{naud}.

\subsection{Reduction to $L^2$-estimates}

Consider the normalised operator defined by
\begin{align}\label{Lsw:normalization}
\underline{\Hc}_{s,\psi}\Psi:= \lambda_{\sigma, \eta}^{-1} \Phi_{\sigma, \eta}^{-1} \Hc_{s,\psi} (\Phi_{\sigma, \eta} \cdot \Psi)
\end{align}
for $\Psi\in C^1(X)$.
Then $\underline{\Hc}_{\sigma,\eta}$ on ${C^1}({{X}})$ has a spectral radius $1$ and fixes the constant function $\bf 1$, i.e., $\underline{\Hc}_{\sigma,\eta} \mathbf{1}=\mathbf{1}$. 

For $t \neq 0$, we set and use the norm
\[ \| \Psi \|_{(t)}:= \|\Psi\|_0+\frac{1}{|t|} \|\Psi \|_0 \mbox{ for } \Psi\in C^1(X)\]
which is equivalent to $\| \cdot \|_1$. One of the main interest in this section is to estimate $\|\Hc_{s,\psi}^n\|_{(t)}$ or equivalently $\|\underline{\Hc}_{s,\psi}^n\|_{(t)}$.

We start the calculation to obtain the bound of Dolgopyat-Baladi-Vall\'ee by reducing our main estimate to $L^2$-type estimate. For the reduction we need:

\begin{lem} \label{relating:lebesgue}
Let $(\sigma,\eta) \in {B}$ where $B$ is chosen small enough so that $\sigma>3/4$. For all $n \geq 1$, we get
\[ \|\underline{\Hc}_{\sigma,\eta}^{n}\Psi\|_0^2 \ll_{B} A_{\sigma,\eta}^{2n} \cdot \big\| \underline{\Hc}_{1,\mathbf{0}}^{n}|\Psi|^2\big\|_0  \]
for $A_{\sigma,\eta}=\lambda_{\sigma,\eta}^{-1}\sqrt{\lambda_{2 \sigma-1, 2\eta}}>0$.
\end{lem}

\begin{proof} \label{cs:leb}
By the Cauchy-Schwarz inequality, we have
\begin{align*} 
 \|\underline{\Hc}_{\sigma,\eta}^{n}\Psi\|_0^2 
\leq \lambda_{\sigma,\eta}^{-2n}  \big\| \Hc_{2\sigma-1, 2\eta}^n \Phi_{2\sigma-1, 2\eta}\big\|_0 \cdot\big\| \mathcal{H}_{1,\mathbf{0}}^{n}|\Psi|^2 \big\|_0.
\end{align*} 
The desired result comes from  $\mathcal{H}_{1,\mathbf{0}}^{n}|\Psi|^2 \ll_B \underline{\Hc}_{1,\mathbf{0}}^{n}|\Psi|^2$
and 
$\Hc_{2\sigma-1, 2\eta}^n \Phi_{2\sigma-1, 2\eta} \ll_{B}\lambda_{2\sigma-1, 2\eta}^n$.
\end{proof}
A crucial observation based on the spectral gap is that the projection operator $\underline{\mathcal{P}}_{1,\bf0}$ associated with the dominant eigenvalue $1$ satisfies: $\underline{\Hc}_{1,\bf 0}=\underline{\mathcal{P}}_{1,\bf 0}+\underline{\mathcal{N}}_{1,\bf 0}$ and the subdominant spectral radius $R_1$ of $\underline{\mathcal{N}}_{1,\bf 0}$ is strictly less than $1$. In particular, we have
$\underline{\mathcal{P}}_{1,\bf 0}\Psi=\int_{{X}}\Psi\,dm$
and hence $\underline{\Hc}_{1,\bf 0}^n\Psi=\int_{{X}}\Psi\,dm+O(R_1^n).$ 
In sum, from Lemma \ref{relating:lebesgue}, we get
\begin{align}\label{L2:red} 
 \|\underline{\Hc}_{s,\psi}^{n+k}\Psi\|_0^2 
\ll_B A_{\sigma,\eta}^{2n} \Big(\int_{{X}}\big|\underline{\Hc}_{s,\psi}^k\Psi\big|^2dm+O(R_1^n) |t| \big\|\Psi\big\|_{(t)}^2 \Big).
\end{align}

\subsection{Estimating $L^2$-norms}

The normalised operator satisfies the Lasota--Yorke inequality, which comes from a direct computation similar to the proof of Proposition \ref{lasota:yorke}.

\begin{prop} \label{lasota:yorke:normal}
Let $B$ be as before. For $(s,\psi)$ with $(\sigma,\eta) \in {B}$ and all $n \geq 1$, we have $\|\underline{\Hc}_{s,\psi}^{n}\Psi\|_1 \ll_B |s| \|\Psi\|_0+ \rho^n \|\Psi\|_1.$
\end{prop}

The following $L^2$-estimate is the heart of $\S$\ref{sect:dolgopyat}, in which the UNI property in Proposition \ref{uni} plays an essential role together with the Lasota--Yorke inequality.


\begin{prop} \label{L2:estimate}
For suitable constants $\alpha, \beta>0$, large $|t| \geq \frac{1}{\rho^2}$, and for $(s,\psi)$ with $(\sigma,\eta) \in {B}$, we have
 \[ \int_{{{X}}} \Big|\, \underline{\Hc}_{s,\psi}^{\lceil \alpha \log |t| \rceil}\Psi \Big|^2 dm \ll_B \rho^{\beta \lceil \alpha \log |t| \rceil} \|\Psi\|_{(t)}^2.  \]
\end{prop}


\begin{proof}

First we express the integrand as 
\begin{align*}
 |\underline{\Hc}_{s,\psi}^{n}\Psi|^2 = \frac{1}{\lambda_{\sigma,\eta}^{2n}} \sum_{({\bf p}, \qb) \in \mathbf{B}^{\circ n} \times \mathbf{B}^{\circ n}}  |\partial\pi_1\pb|^{it} |\partial\pi_1\qb|^{-it} \cdot R^\sigma_{\bf p, \qb}
\end{align*}
where we set  
\begin{align*}
g_\psi(\bf p,q)&:= \exp[\mathfrak{g}_\psi(\pb)+\mathfrak{g}_{\cl{\psi}}(\qb)],\\
R^\sigma_{\bf p, \qb} &:= \Phi_{\sigma,\eta}^{-2}\cdot g_\psi(\pb,\qb)|\partial\pi_1\pb|^\sigma |\partial\pi_1\qb|^\sigma \cdot (\Phi_{\sigma,\eta} \Psi) \circ {\bf p} \cdot (\Phi_{\sigma,\eta}\overline{\Psi}) \circ \qb 
\end{align*}
in order to simplify the notation.
Thus we have 
\begin{align}\label{sum:p:q} 
\int_{{{X}}} |\underline{\Hc}_{s,\psi}^{n}\Psi|^2dm=  \frac{1}{\lambda_{\sigma,\eta}^{2n}} \sum_{(\bf p,q)}
 \int_{{X}} \exp \left[ it \Pi_{\pb,\qb} \right] R_{\bf p, q}^\sigma\,dm.
\end{align}
Here recall that $\Pi_{\bf p,q}= \log |\partial \pi_1 \bf{p}|-\log |\partial \pi_1 \bf{q}|$.


Since $R_{\pb,\qb}^\sigma$ are bounded, the sum is dominated by the oscillatory integrals which are controlled by the behavior of the phase function $\Pi_{\pb,\qb}$, hence essentially by the geometric properties of the skewed Gauss map. We divide the sum (\ref{sum:p:q}) into two parts: one with close pairs, i.e., with small  $\Delta(\pb,\qb)$ and another with relatively separated pairs, i.e., with relatively large $\Delta(\pb,\qb)$. In other words, the integral (\ref{sum:p:q}) is written as $\int_{{{X}}} |\underline{\Hc}_{s,\psi}^{n}\Psi|^2dm= I^{(1)} + I^{(2)}$, where
\begin{align*} 
I^{(1)} &:= \frac{1}{\lambda_{\sigma,\eta}^{2n}} \sum_{\Delta(\pb,\qb) \leq \varepsilon}
 \int_{{X}} e^{it \Pi_{\pb,\qb}} R_{\bf p,q}^\sigma\,dm
 \mbox{  and }\\ I^{(2)} &:=   \frac{1}{\lambda_{\sigma,\eta}^{2n}} \sum_{\Delta(\pb,\qb)>\varepsilon}
 \int_{{X}} e^{it \Pi_{\pb,\qb}}  R_{\bf p,q}^\sigma\,dm. 
\end{align*}


Let us consider the integral $I^{(1)}$. We set $\nu_{\sigma, \eta}:=\Phi_{\sigma,\eta} \mu_{\sigma,\eta}$, which is fixed by the normalised adjoint operator.
We need the following results that are mere reformulations of Naud \cite[Lemma 4.2]{naud}.
\begin{lem}\label{nu:sigma:x:lem}
\begin{enumerate}[leftmargin=*]
\item \label{order:mag:ratio}For all $\pb\in{\bf B}^{\circ n}$, we get 
$\ds\frac{\|\partial\pi_1\pb\|_0^\sigma}{\lambda_{\sigma,\eta}^n}\asymp_B \nu_{\sigma,\eta}(\pb(X)).$
\item Let ${\bf A}$ be a subset of ${\bf B}^{\circ n}$ and $Y=\bigcup_{\qb\in{\bf A}}\qb(X)$. Then, $\nu_{\sigma,\eta}(Y)\ll_B A_{\sigma,\eta}^{2n}m(Y)^{1/2}.$
\end{enumerate}
\end{lem}

 

Obviously we have 
\begin{align}
	I^{(1)} \ll_B \frac{\| \Psi \|_0^2}{\lambda_{\sigma,\eta}^{2n}} \sum_{\Delta(\pb,\qb) \leq \varepsilon}\|\partial\pi_1\pb\|_0^\sigma\|\partial\pi_1\qb\|_0^\sigma  \int_{{X}} g_\eta(\pb,\qb) dm\label{L2:I1}.
\end{align}
Hence, by Lemma \ref{nu:sigma:x:lem} we have
\begin{align*}
I^{(1)}&\ll_B \| \Psi \|_0^2 \sum_{\Delta(\pb,\qb) \leq \varepsilon}\nu_{\sigma,\eta}(\pb(X))\nu_{\sigma,\eta}(\qb(X))  \int_{{X}} g_\eta(\pb,\qb)dm\\
&\ll_B\| \Psi \|_0^2 \sum_{\pb \in {\bf B}^{\circ n}} \nu_{\sigma,\eta}(\pb({X}))  \bigg( \sum_{{\qb \in {\bf B}^{\circ n}} \atop \Delta(\pb,\qb) \leq \varepsilon} \nu_{\sigma,\eta}(\qb({X})) \bigg).
\end{align*}
For any $0<a<1$, taking $\varepsilon=\rho^{an}$, we finally have 
\begin{align*}
|I^{(1)}| \ll_B \|\Psi\|_0^2\lambda_{\sigma,\eta}^{-n}\|\Hc_{\sigma,\bf 0}^n{\bf 1}\|_0\rho^{an/2}\ll_B \rho^{an/2} A_{\sigma,\eta}^{2n}  \| \Psi \|_0^2
\end{align*}
by the UNI condition of Proposition \ref{uni}.(\ref{uni:a}) and Lemma \ref{nu:sigma:x:lem}.(\ref{order:mag:ratio}). 

The main point of estimating $I^{(2)}$ is to deal with the oscillatory integrals with the phase function $\Pi_{\pb,\qb}$. By Proposition \ref{uni}.(\ref{uni:b}) and a version of Van der Corput Lemma (see Baladi-Vall\'ee \cite[p.359]{bv}), we obtain 
\begin{align*}
|I^{(2)}| \ll \sum_{\Delta(\pb,\qb) >\varepsilon} \frac{\| R_{\bf p,q}^\sigma \|_1}{|t|} \left( \frac{1}{\varepsilon}+\frac{1}{\varepsilon^2} \right) \ll_{B} \| \Psi \|_{(t)}^2 \frac{(1+\rho^n|t|)}{|t|} \left( \frac{1}{\varepsilon}+\frac{1}{\varepsilon^2} \right) 
\end{align*}
by observing the Lasota--Yorke type estimate for $R_{\bf p,q}^\sigma$. Then again choosing the scale $\varepsilon=\rho^{an}$ with $n=\lceil \alpha \log |t| \rceil$ for some $\al$ and $a$ satisfying $|t| \leq \rho^{-a n}$, we have $|I^{(2)}| \ll \rho^{(1-2a)n} \| \Psi \|_{(t)}^2$.

Hence by the above choices of $\epsilon$, $n$, $a$ and $\al$, we finally have the complete estimate for $I^{(1)} + I^{(2)}$ with a constant $\beta=1-2a>0$.
\end{proof}


\subsection{Uniform polynomial growth}


Finally, the following Dolgopyat--Baladi--Vall\'ee estimate can be deduced from the $L^2$-type estimate in Proposition \ref{L2:estimate}.

\begin{thm} \label{bound:dolgopyat}
For $0<\xi<\frac{1}{5}$, there is a small neighborhood $B' \subseteq B$ such that for all complex pair $(s,\psi)$ with $(\sigma,\eta)\in B'$, an integer $n \geq 1$, and $|t| \geq \frac{1}{\rho^2}$, we have 
\[ \| \Hc_{s,\psi}^{n}\|_{(t)} \ll_{B,\xi} (r \lambda_{\sigma,\eta})^n |t|^\xi \]
for some $0<r<1$. In particular, the quasi-inverse $(\mathcal{I}-\Hc_{s,\psi})^{-1}$ is well-defined and analytic when $(\sigma,\eta)\in B'$.
\end{thm}

\begin{proof}

Set $n_0=n_0(t):=\lceil \alpha \log |t| \rceil$. From (\ref{L2:red}), for $n_1=n_1(t) \geq n_0$, we have
\begin{align*} 
\|\underline{\Hc}_{s,\psi}^{n_1}\Psi \|_0^2 &\ll_{B}  A_{\sigma,\eta}^{2(n_1-n_0)} \left(  \int_{{{X}}} |\underline{\Hc}_{s,\psi}^{n_0}\Psi|^2 + R_1^{n_1-n_0} |t| \|\Psi\|_{(t)}^2 \right) \\ &\ll_{B} A_{\sigma,\eta}^{2(n_1-n_0)} \left( \rho^{\beta n_0} + R_1^{n_1-n_0} |t|  \right) \|\Psi\|_{(t)}^2 .
\end{align*}
We take $n_1=\lceil\widetilde{\alpha} n_0\rceil$ for some $\widetilde{\alpha}>1$ to be large enough to have $R_1^{n_1-n_0} |t|=O(\rho^{\beta n_0})$ and choose $B'$ to be small enough so that $A_{\sigma,\eta}^{n_1-n_0}\ll_{B'} \rho^{-\beta n_0 /2}$. Then, we get
\[ \|\underline{\Hc}_{s,\psi}^{n_1}\Psi \|_0 \ll_{B'} \rho^{\widetilde{\beta} n_1} \|\Psi\|_{(t)} \]
for a suitable $\widetilde{\beta}>0$. Repeated application of Lasota-Yorke inequality from Proposition \ref{lasota:yorke:normal} enables us to write
\[
 \|\underline{\Hc}_{s,\psi}^{2n_1}\Psi \|_1 \ll |s| \|\underline{\Hc}_{s,\psi}^{n_1}\Psi \|_0 + \rho^{n_1} \|\underline{\Hc}_{s,\psi}^{n_1}\Psi\|_1 \\ \ll \rho^{\widetilde{\beta} n_1} |t|  \|\Psi\|_{(t)}
\]
and hence we have $\|\underline{\Hc}_{s,\psi}^{2n_1}\|_{(t)} \ll \rho^{\widetilde{\beta} n_1}$. For a fixed $t$ with $|t| \geq \frac{1}{\rho^2}$, writing any integer $n=(2 n_1)q+m$ with $m<2 n_1$, we obtain
\[ \|\underline{\Hc}_{s,\psi}^{n}\|_{(t)} \leq  \|\underline{\Hc}_{s,\psi}^{m}\|_{(t)} \|\underline{\Hc}_{s,\psi}^{2 n_1}\|_{(t)}^q  \ll \rho^{\widetilde{\beta} q n_1} \leq \rho^{\widetilde{\beta} n /2} \rho^{-\widetilde{\beta} n_1}   \] 
since for a large $|t|$, we have $\|\underline{\Hc}_{s,\psi}^{m}\|_{(t)} \ll 1$. This leads to the assertion by choosing $\xi=\widetilde{\beta}\widetilde{\alpha}$ and $r=\rho^{\widetilde{\beta}/2}$.

More detailed computation in Baladi--Vall\'ee \cite[\S3.3, Eq.(3.21)--(3.23)]{bv} shows the closed forms of $\widetilde{\alpha}$ and $\widetilde{\beta}$, which determine that the constant $\xi$ can be taken to be any value between $0$ and $1/5$.
\end{proof}

\hide{---
\begin{rem}
More detailed computation in Baladi--Vall\'ee \cite[\S3.3]{bv} shows that actually the constant $\xi$ can be taken between $0$ and $1/5$.
\end{rem}
---}

\section{Coboundary conditions}\label{sec:pressure}

First we collect some preliminary results to prove the main steps of Proposition \ref{main:dynamics}. We follow a similar argument of Baladi-Vall\'ee \cite[Proposition 1]{bv}. 

From Baladi--Vall\'ee \cite[Proposition 0.6a]{bv}) together with Remark \ref{rem:M:L:sq}, we get
\begin{prop}\label{derv:lambda:s:0}
$\ds\frac{\partial \lambda_{s, \bf 0}}{\partial s}\Big|_{s=1}=-\frac{\pi^2}{12 \rho\log 2}.
$
\end{prop}

For a fixed real number $h$, define a piecewise differentiable cost function $\Upsilon\in L^1({{X}})$ such as 
$$\Upsilon(x,v):=\begin{cases}
2h\log|x|+{\psi}(v)&\mbox{ if }\mathfrak{g}=\bG\\
2h\log|x|+{\psi}(v)+{\psi}(v\twobytwotiny{-m_1(x)}{1}{1}{0}{\rm j})&\mbox{ if }\mathfrak{g}=\cG
\end{cases}.$$
From now on, let us set 
$$\mathbf{S}:=\widehat{\Tb}\mbox{ or }\widetilde{\Tb}^2$$ according to the choice of $\mathfrak{g}=\bG$ or $\cG$, respectively. 
The following two results will be useful when discussing the property of the pole $s(\wbb)$.

\begin{prop}\label{var:cob:prop}
For a $\psi\in \Cb^{\cha{\Gamma}}$, we have
$$\frac{d^2}{dw^2}\lambda_{1+h w,w\psi}\Big|_{w=0}=\lim_{n\rightarrow\infty}\frac{1}{n}\int_{{X}}\left(\sum_{k=0}^{n-1}\Upsilon\circ\mathbf{S}^k\right)^2\Phi_{1,\bf 0}\,dm.$$
\end{prop}

\begin{proof}
We set $\kappa(w):=\lambda_{1+hw,w\psi}$  and $\Psi(w):=\Phi_{1+hw,w\psi}$. Note that $\kappa(0)=1$ and $\kappa'(0)=0$. 
From Proposition \ref{birkhoff:prop}, we obtain
$$\kappa(w)^n\Psi(w)=\mathcal{H}_{1+hw,w\psi}^n\Psi(w)=\mathcal{H}_{1,0}^n\left[\exp\Big[w\sum_{k=0}^{n-1}\Upsilon\circ{\bf S}^k\Big]\Psi(w)\right].$$
Differentiating this twice and setting $w=0$, we have
\begin{align*}
n\kappa''(0)\Psi(0)+\Psi''(0)
=\mathcal{H}_{1,0}^n\left[(\sum_{k=0}^{n-1}\Upsilon\circ{\bf S}^k)^2\Psi(0)+2(\sum_{k=0}^{n-1}\Upsilon\circ{\bf S}^k)\Psi'(0)+\Psi''(0)\right].
\end{align*}
Hence, we get
$$\kappa''(0)=\frac{1}{n}\int_{{X}}\left((\sum_{k=0}^{n-1}\Upsilon\circ{\bf S}^k)^2\Psi(0)+2(\sum_{k=0}^{n-1}\Upsilon\circ{\bf S}^k)\Psi'(0)\right)dm.$$
It can be shown that the second term satisfies
\begin{align}\label{2nd:van:upsilon}
\int_{{X}}(\frac{1}{n}\sum_{k=0}^{n-1}\Upsilon\circ{\bf S}^k)\Psi'(0)dm=o(1).
\end{align}
In fact, we have
$$\int_{{X}}(\frac{1}{n}\sum_{k=0}^{n-1}\Upsilon\circ{\bf S}^k)\Psi'(0)dm=\int_{{X}}\Upsilon \left(\frac{1}{n}\sum_{k=0}^{n-1}\Hc_{1,0}^k\Big[\Psi'(0)\Big]\right)dm$$
and since the spectral radius of $\mathcal{N}_{1,0}$ is strictly less than $1$, we get
$$\frac{1}{n}\sum_{k=0}^{n-1}\Hc_{1,0}^k\Big[\Psi'(0)\Big]=\frac{1}{n}\sum_{k=0}^{n-1}\mathcal{P}_{1,0}^k\Big[\Psi'(0)\Big]+o(1)=C\Phi_{1,0}+o(1)$$ 
for a constant $C$. Hence L.H.S. of (\ref{2nd:van:upsilon}) equals
$$C\int_{{X}}\Upsilon \Phi_{1,0}dm+o(1)=C \kappa'(0)+o(1)=o(1).$$
This finishes the proof.
\end{proof}

\begin{prop}\label{var:h:cohomo}
There is $\Theta\in C^1(X)$ such that for $\widetilde{\Upsilon}:=\Upsilon+\Theta\circ{\bf S}-\Theta$, we have
\begin{align*}
	\frac{d^2}{dw^2}\lambda_{1+h w,w\psi}\Big|_{w=0}=\int_{{X}}\widetilde{\Upsilon}^2\Phi_{1,\bf 0}\,dm.
\end{align*}
In particular, the last quantity is zero if and only if $h=0$ and $\psi$ is a coboundary over $\Rb$.
\end{prop}

\begin{proof}
Recall that  $\int_{{X}}\Upsilon\Phi_{1,\bf 0}\, dm=0$. Hence, we get $\|\Hc_{1,0}^n\Upsilon\Phi_{1,\bf 0}\|_0\ll R_1^n$ for the subdominant eigenvalue $R_1<1$ as $\Hc_{1,\bf 0}[\Upsilon\Phi_{1,\bf 0}]\in C^1({X})$ and obtain a  function $$\Theta:=\Phi_{1,\bf 0}^{-1}(\mathcal{I}-\Hc_{1,\bf 0})^{-1}\Hc_{1,\bf 0}[\Upsilon\Phi_{1,\bf 0}]$$  
that is well-defined in $C^1({X})$. In a similar way as (\ref{2nd:van:upsilon}), it can be shown that
$$\int_{{X}}\Big(\frac{1}{n}\sum_{k=0}^{n-1}\widetilde{\Upsilon}\circ{\bf S}^k\Big)\Phi_{1,\bf 0}dm=o(1).$$
Since $\Theta\circ {\bf S}-\Theta$ is bounded, we can conclude that
\begin{align*}
\frac{d^2\lambda}{dw^2}(1,0)=\lim_{n\rightarrow\infty}\frac{1}{n}\int_{{X}}\Big(\sum_{k=0}^{n-1}\widetilde{\Upsilon}\circ {\bf S}^k\Big)^2\Phi_{1,\bf 0}\,dm.
\end{align*}
Since  $\Hc_{1,0}[\Psi_1\Psi_2\circ{\bf S}]=\Hc_{1,0}[\Psi_1]\Psi_2$ for any $\Psi_1,\Psi_2\in C^1({{X}})$, one can show that $\Hc_{1,\bf 0}[\widetilde{\Upsilon}\Phi_{1,\bf 0}]=0$. 
Hence, for $k>j$, we have
\begin{align*}
\int_{{X}}\widetilde{\Upsilon}\circ{\bf S}^k\,\widetilde{\Upsilon}\circ{\bf S}^j\Phi_{1,\bf 0}\,dm =\int_{{X}}\Hc_{1,\bf 0}^{k-j-1}\left[\widetilde{\Upsilon}\circ {\bf S}^{k-j-1}\Hc_{1,\bf 0}[\widetilde{\Upsilon}\Phi_{1,\bf 0}]\right]\,dm
=0.
\end{align*}
Hence, we obtain the first statement. 

For the second statement, observe that the integral is zero if and only if $\widetilde{\Upsilon}=0$, i.e., $\Upsilon=\Theta-\Theta\circ {\bf S}$. Since  $\Theta$ is bounded, the latter statement is equivalent to the conditions that  $h=0$ and $\psi$ is a coboundary over $\Rb$.
\end{proof}


\begin{rem}\label{cob:cond2}
	Observe that 	$\psi$ is a $\mathfrak{g}$-coboundary over $\Bbbk$ if and only if there exists a $\beta\in \Bbbk^{\cha{\Gamma}}$ such that  $\mathfrak{g}_\psi(\pb)=\beta-\beta\circ\pi_2\pb$ for any $\pb\in \mathbf{B}$.
\end{rem}
The following result will be one of crucial ingredients in the next section.

\begin{prop}\label{op:far0}
	Let $\Hc_{s,\psi}=\widehat{\Lc}_{s,\psi}$ or $\mathcal{M}_{s,\psi}$. Let $t$ be a real number. Then, $1$ is an eigenvalue of $\Hc_{1+it, i \zeta}$ if and only if $t=0$ and $\zeta$ is a $\mathfrak{g}$-coboundary over $\Rb/2\pi\Zb$.
\end{prop}

\begin{proof}
	First let us assume that there is a $\Psi\in C^1({X})$ with $\|\Psi\|_0=1$ such that $\Hc_{1+it,i\zeta}\Phi_{1,\bf 0}\Psi=\Phi_{1,\bf 0}\Psi$. 
	Suppose that $|\Psi|$ attains a maximum at $(x_0,v_0)$. Setting 
	\begin{align*}
		a_\qb&:=\frac{1}{\Phi_{1,\bf 0}(x_0,v_0)}|\partial\pi_1\qb(x_0)|\Phi_{1,0}\circ\qb(x_0,v_0)\mbox{ and }\\ 
		b_\qb&:=\frac{1}{\Psi(x_0,v_0)}\Psi\circ\qb(x_0,v_0)\exp[i\mathfrak{g}_\zeta(\qb)(v_0)]|\partial\pi_1\qb(x_0)|^{it},
	\end{align*}
	we have $\sum_{\qb\in{\bf B}^{\circ n}}a_\qb b_\qb=1$ for all $n\geq1$.
	Since $\sum_{\qb\in{\bf B}^{\circ n}}a_\qb=1$ and $|b_\qb|\leq 1$, we obtain $b_\qb=1$ for all $\qb$.  In other words,
	we have
	$$\exp[i\mathfrak{g}_\zeta(\qb)(v_0)]|\partial \pi_1\qb(x_0)|^{i t} \Psi\circ\qb(x_0,v_0) = \Psi(x_0,v_0)$$
	for any $\qb \in {\bf B}^{\infty}$. 
	Proposition \ref{op:alg:simple} enables us to show $|\Psi|\equiv 1$, the constant function. Then we repeat the above process for any $(x,v)\in X$. In sum, we conclude
	\begin{align}\label{C2:conj}
		\exp[i\mathfrak{g}_\zeta(\qb)]|\partial \pi_1\qb|^{i t} \Psi\circ \qb = \Psi
	\end{align}
	for all  $\qb\in {\bf B}^{\infty}$.

	From (\ref{C2:conj}), we have 
	$$|t|\cdot\|\partial \Pi_{\pb,\qb}\|_0=\|\partial \pi_1\pb\cdot\partial (\log \Psi)\circ\pb-\partial \pi_1\qb\cdot\partial (\log \Psi)\circ\qb\|_0$$
	for all $\pb,\qb\in{\bf B}^{\circ n}$ and $n\geq 1$. By the uniform contraction in Proposition \ref{prop:gaussmap}, we obtain
	 $|t|\cdot\|\partial \Pi_{\pb,\qb}\|_0\ll \rho^n\leq \rho^{an}$ 
	for all $0<a<1$ and $\pb,\qb\in{\bf B}^\infty$.
	Hence, we conclude that $t=0$, otherwise the last inequality violates the UNI property (a) in Proposition \ref{uni}. In sum, there exists a $\Psi\in C^1({X})$ such that $|\Psi|\equiv 1$ and
	\begin{align}\label{cocycle:ybb}
		\exp[i\mathfrak{g}_\zeta(\pb)] = \frac{\Psi}{\Psi\circ\pb}\mbox{ for }\pb\in {\bf B}.
	\end{align}
	
	Since $\mathfrak{g}_\zeta(\pb)$ is independent of $x$, by differentiating both sides of (\ref{cocycle:ybb}) with respect to $x$, we get $|\partial \pi_1\pb|\cdot|\partial \Psi \circ\pb|=|\partial \Psi|$ for any $\pb$. As $\|\partial \pi_1\pb\|_0$ can be arbitrarily small, we get $\partial\Psi\equiv0$, i.e., $\Psi$ is a function only on $\cha{\Gamma}$. This implies $\Psi=\exp[i \beta]$
	for a $\beta\in (\Rb/2\pi\Zb)^{\cha{\Gamma}}$. From Remark \ref{cob:cond2}, we conclude that $\zeta$ is a coboundary over $\Rb/2\pi\Zb$. 
	
	Conversely, we assume (\ref{cocycle:ybb}). Then it can be easily seen that $\Phi_{1,\bf 0}\Psi$ is an eigenfunction for $\Hc_{1,i\zeta}$. Hence, we finish the proof.
\end{proof}

\begin{rem}\label{sp:case:lv:1}
Let $\Gamma=\SL_2(\Zb)$. Then, $\psi:\SL_2(\Zb)\backslash\SL_2(\Zb)\rightarrow \Cb$ is just a variable $\psi=w$ and $\zeta=\tau\in \Rb$. Hence, $\zeta$ is a coboundary over $\Rb/2\pi \Zb$ if and only if $\tau$ is zero in $\Rb/2\pi\Zb$ if and only if $\tau$ is an integral multiple of $2\pi$. This is the result of \cite[Proposition 0]{bv} with $L=1$.
\end{rem}

\section{Proof of Proposition \ref{main:dynamics}}\label{pf:prop:2.2}

Combining all the previous results, we prove Proposition \ref{main:dynamics}. The proofs are quite similar to ones of Lemma 8 and 9 in Baladi-Vall\'ee \cite{bv}. For $\wbb\in\Cb^{d}$, set
$$\Hc_{s,\wbb}:=\Hc_{s,\wbb\cdot{\bm \psi}}\mbox{ and }\lambda(s,\wbb):=\lambda_{s,\wbb\cdot{\bm \psi}}.$$
For each statement, we specify regions $W_1$, $W_2$, and $W_3$ around $\bf 0$ and in the end take the intersection to get the desired $W$.

\subsection{Statement (1)} We first prove the case ${\bf v}={\bf 0}$. By Theorem \ref{keyrelation}, it is enough to discuss the behavior of $(\mathcal{I}-\Hc_{s,\wbb})^{-1}$ since all the interval and auxiliary operators are analytic by Proposition \ref{prop:DJ:conv}; and so are the final operators. To obtain the statements (a) and (b), following the proof of Baladi-Vall\'ee \cite[Lemma 8]{bv}, we split the region into three pieces I, II, and III according to $t=\Im s$.

\emph{(\rom{1}) When $|t|$ is small:}
Let $|\wbb_0|=1$ be fixed. As discussed in Baladi-Vall\'{e}e \cite{bv}, Kato \cite{kato}, and Sarig \cite{sarig}, when $(s,w)$ is subjected to a small perturbation near $(1,0)\in \Cb^2$, one can show that the operators $\Hc_{s,w\wbb_0}$, (\ref{DJ:s:w}), (\ref{DJhat:s:w}), and (\ref{DJtilde:s:w}) with $\psi=w\wbb_0\cdot {\bm \psi}$ are all analytic. Furthermore, the properties of spectral gap, uniqueness, and simplicity of the eigenvalue in Proposition \ref{op:dominant} and \ref{ruelle:perron} are extended to a complex parameter family $\Hc_{s,w\wbb_0}$. As $\wbb_0$ is arbitrary, a standard argument in the theory of several complex variables ensures that all the operators mentioned above are analytic for the general variable $\wbb\in\Cb^{d}$ instead of $w\wbb_0$. Similarly we also have:

\begin{prop} \label{perturb:analyticity}
	There exists a complex neighborhood $U$ of $(1, \bf 0)$ such that for all $(s,\wbb) \in U$, the operator $\Hc_{s,\psi}$ has a spectral gap with the decomposition $\Hc_{s,\wbb}=\lambda_{s,\wbb} \mathcal{P}_{s,\wbb}+\mathcal{N}_{s,\wbb}$, where $\lambda_{s,\wbb}$, $\mathcal{P}_{s,\wbb}$, $\mathcal{N}_{s,\wbb}$ are analytic on $U$ and $R(\mathcal{N}_{s,\wbb})<|\lambda_{s,\wbb}|$.
	Further, the corresponding eigenfunction $\Phi_{s,\wbb}$ and its derivative $\partial \Phi_{s,\wbb}$ are well-defined and analytic on $U$.
\end{prop}

Note that $\frac{\partial}{\partial s}\lambda_{s, \bf 0}\Big|_{s=1}\neq 0$ by Proposition \ref{derv:lambda:s:0}. 
By the implicit function theorem, we have an analytic map $s$ from the neighborhood $W_1$ of $\bf 0$ to $\Cb$ such that for some $\delta_1>0$ and $t_0>0$,  $\lambda({s(\wbb),\wbb})=1$ with $|\Re s(\wbb)-1|\leq\delta_1$ and $|\Im s(\wbb)|< t_0$ for all $\wbb\in W_1$. Obviously, $s({\bf 0})=1$.


\emph{(\rom{2}) When $t_0\leq |t|\leq \frac{1}{\rho^2}$:} In a similar way as Baladi-Vall\'ee \cite[Lemma 8]{bv}, one can conclude with the help of Proposition \ref{op:far0} that there exists $\delta_2>0$ and a neighborhood $W_2$ of $\bf 0$ such that  the distance between $1$ and the spectrum of $\Hc_{s,\wbb}$ is positive on region $|\Re s-1|\leq \delta_2$ for all $\wbb\in W_2$. Hence, $(\mathcal{I}-\Hc_{s,\wbb})^{-1}$ is analytic and bounded on the region.

\emph{(\rom{3}) When $|t|\geq \frac{1}{\rho^2}$:} Using Theorem \ref{bound:dolgopyat}, we can find $\delta_3>0$ and a neighborhood $W_3$ of $\bf 0$ such that $(\mathcal{I}-\Hc_{s,\wbb})^{-1}$ is analytic on the region $|\Re s-1|\leq \delta_3$ with $|\Im s|\geq \frac{1}{\rho^2}$ for all $\wbb\in W_3$.

Now take $\alpha_1$ as the minimum of $\delta_1$, $\delta_2$, and $\delta_3$. For any $0<\widehat{\alpha}_1<\alpha_1$, choose a neighborhood $W$ of $\bf 0$ small enough so that $W\subseteq W_1\cap W_2\cap W_3$ and $\Re s(\wbb)>1-(\alpha_1-\widehat{\alpha}_1)$.

To obtain the statement (c),  we fix $\wbb_0\neq \bf 0$ and set $s(w):=s(w\wbb_0)$. 
Let us set $\psi_0=\wbb_0\cdot {\bm \psi}$. We also set
$\Hc_{s,w}:=\Hc_{s,w\psi_0}$ and $\lambda(s,w):=\lambda_{s,w\psi_0}$ for $w\in\Cb$. 

Since $\lambda(s(w),w)=1$ for a small $|w|$, we have
\begin{align}\label{ch:rule:s:0}
s'(0)=-\frac{\partial\lambda}{\partial w}(1,0)\Big\slash \frac{\partial \lambda}{\partial s}(1,0)
\end{align}
We also note that
\begin{align}\label{twice:diff:lambda}
\frac{\partial \lambda}{\partial s}(1,0)s''(0)=\frac{d^2}{dw^2}\lambda(1+s'(0)w,w)\Big|_{w=0}.
\end{align}
From Proposition \ref{derv:lambda:s:0} and \ref{var:h:cohomo}, we obtain that $\psi_0$ is not a $\mathfrak{g}$-coboundary over $\Rb$ if and only if $s''(0)\neq 0$. Since $\wbb_0$ is arbitrary, we can conclude that the Hessian of $s(\wbb)$ at $\wbb=\bf 0$ is non-singular if and only if $\psi_i$ are linearly independent over $\Rb$ modulo $\Bc_\mathfrak{g}(\Gamma,\Rb)$.

Let us consider the statement (d). Let $\mathcal{R}_{\wbb}$ be the residue operator of the quasi-inverse at $s=s(\wbb)$. Since $\mathcal{R}_{\wbb}$ is the residue operator of $(1-\lambda_{s,\wbb})^{-1}\mathcal{P}_{s,\wbb}$ at $s=s(\wbb)$, by Theorem \ref{keyrelation}, the residue of the Dirichlet series is
\begin{align}\label{res:ev:temp}
\mathcal{E}_{s,\wbb}^J\mathcal{R}_{\wbb}\mathcal{G}_{s,\wbb} \Psi(0,\Gamma)=-\frac{\mathcal{E}_{s,\wbb}^J\Phi_{s,\wbb}(0,\Gamma)}{\frac{\partial}{\partial s}\lambda_{s,\wbb}}\left(\int_{{{X}}}\mathcal{G}_{s,\wbb}\Psi d\mu_{s,\wbb}\right)
\end{align}
where $\mathcal{E}_{s,\wbb}^J:=\widehat{D}_{s,\wbb}^J$ or $\widetilde{D}_{s,\wbb}^J$ and $\mathcal{G}_{s,\wbb}=\widehat{\mathcal{F}}_{s,\wbb}$ or $\widetilde{\mathcal{F}}_{s,\wbb}$ according to $\mathfrak{g}=\bG$ or $\cG$, respectively.
Let $\wbb=\bf 0$. For the evaluation of an integration in (\ref{res:ev:temp}) at $\wbb={\bf 0}$, observe first that we have $\int_{{{X}}}\Psi \Phi_{1,\bf 0}^{-1}d\mu_{1,\bf 0}=\int_{{{X}}}\Psi\,dm$ since $\mu_{1,{\bf 0}}=dm$. For $\mathfrak{g}=\cG$, note also that $\int_X \widetilde{\mathcal{F}}_{1,\bf 0}\Psi d\mu_{1,\bf 0}= 2\int_X \widehat{\mathcal{F}}_{1,\bf 0}\Psi d\mu_{1,\bf 0}$. Hence the integration in (\ref{res:ev:temp}) equals 
$$\frac{1}{2\rho\log2}\sum_{m\geq 2}\int_{{{X}}}\frac{1}{(m+x)^2}\Psi\left(\frac{1}{m+x},v\right)dxdv=\frac{1}{2\rho\log 2}\int_{(0,\frac{1}{2})\times\Gamma\backslash\SL_2(\Zb)}\Psi\, dm.$$ 
Note also that for an open fundamental interval $K=K^\circ(m_1,\cdots,m_\ell)$, we obtain
$${\mathcal{E}}_{1,{\bf 0}}^K\Phi_{1,\bf 0}(0,\Gamma)=\frac{[0;1,m_\ell,\cdots,m_1]}{Q([0;m_1,\cdots,m_\ell])^2\log 2 }=\frac{|K|}{\log 2}.$$
Hence we also obtain the same expression for an interval $J$. In total, we obtain the desired expression for the residue. 

Now consider a general $\bf v$. Let $\wbb$ be written as $\wbb={\bf u}+i{\bf v}$ for ${\bf u}\in W$, a neighborhood of ${\bf 0}$. With the observation in Remark \ref{cob:cond2}, one can easily show that $\Hc_{s,\wbb}\Psi=e^{i\beta}\Hc_{s,{\bf u}}[e^{-i\beta}\Psi]$ for all $\Psi$ and same expressions for the other operators $\mathcal{E}^J_{s,\wbb}$ and $\mathcal{G}_{s,\wbb}$. Therefore, we get $L_{\Psi,J}(s,\wbb)=e^{i\beta(\Gamma)}L_{e^{-i\beta}\Psi,J}(s,{\bf u})$. Hence, all the necessary properties of $L_{\Psi,J}(s,\wbb)$ follow from ones of $L_{e^{-i\beta}\Psi,J}(s,{\bf u})$. This concludes the proof of statement (1).

\subsection{Statement (2)} For a given ${\bf v}\neq {\bf 0}$, choose a neighborhood $W_1$ of $i{\bf v}$ small enough so that $\wbb\cdot{\bm \psi}$ is not a $\mathfrak{g}$-coboundary over $\Rb/2\pi\Zb$ for all $\wbb\in W_1$. Then, using Proposition \ref{op:far0}, the proof goes exactly same as the proof for (a) and (b). 

\subsection{Statement (3)} We split the region into three pieces I, II, and III as before. In the region I, whether ${\bf v}\cdot {\bm\psi}\in \Bc_{\mathfrak{g}}(\Gamma,\Rb)$ or not, i.e., the series is meromorphic or not,  $L_{\Psi,J}(s,\wbb)$ is bounded on $\Re s=1\pm \alpha_1$. On the region II, the series is bounded as it is analytic. On the region III, we now apply Dolgopyat--Baladi--Vall\'ee bound.

This finishes the proof of Proposition \ref{main:dynamics}. \qed


\begin{thebibliography}{wwww}


\bibitem{ash:stevens} A. Ash and G. Stevens, Modular forms in characteristic $\ell$ and special values of their $L$-functions, Duke Math. J. 53 (1986), 84--868.

\bibitem{bv} V. Baladi and B. Vall\'ee, Euclidean algorithms are Gaussian, J. Number Theory. 110 (2005), 331--386.

\bibitem{baladi} V. Baladi, Positive transfer operators and Decay of correlations, Advanced Series in Non-linear Dynamics, World Scientific Publishing. 16 (2002).



\bibitem{bettin:drappeau} S. Bettin and S. Drappeau. Limit laws for rational continued fractions and value distribution of quantum modular forms. Proc. Lond. Math. Soc. (3), 125(6):1377–1425, 2022.


\bibitem{b.f.k.m.m.s} V. Blomer, \'E. Fourvry, E. Kowalski, P. Michel, D. Mili\'cevi\'c, and W. Sawin, The second moment theory of families of $L$-functions, Memoirs AMS 282 no.1394 (2023).

\bibitem{broise} A. Broise, Transformations dilatantes de l'intervalle et th\'eor\`emes limites, Ast\'erisque. (1996), no. 238, 1--109. 

\bibitem{burungale:hida} A. Burungale and H. Hida, Andr\'e-Oort conjecture and nonvanishing of central $L$-values over Hilbert class fields, Forum Math. Sigma 4 (2016), e20, 26 pp. 


\bibitem{burungale:sun} A. Burungale and H.-S. Sun, Quantitative non-vanishing mod $p$ of Dirichlet $L$-values, Math. Ann. 378 (2020), no. 1-2, 317--358.

\bibitem{cv} E. Cesaratto and B. Vall\'ee, Gaussian behavior of Quadratic irrationals, Acta Arith. 197 (2021), 159--205.

\bibitem{const:nord} P. Constantinescu and A. C. Nordentoft. Residual equidistribution of modular symbols and cohomology classes for quotients of hyperbolic $n$-space. Trans. Amer. Math. Soc., 375(10):7001–7034, 2022.

\bibitem{d:h:k:l} N. Diamantis, J. Hoffstein, E. M. Kiral, and M. Lee, Additive twists and a conjecture by Mazur, Rubin and Stein, J. Number Theory. 209 (2020), 1--36.

\bibitem{dolgopyat} D. Dolgopyat, On decay of correlations in Anosov flows, Ann. of Math. (2) 147 (1998), no. 2, 357--390.

\bibitem{ferrero:washington} B. Ferrero and L. C. Washington, The Iwasawa invariant $\mu_p$ vanishes for abelian number fields, Ann. of Math. (2) 109 (1979), no. 2, 377--395.

\bibitem{greenberg:stevens} R. Greenberg and G. Stevens, $p$-adic $L$-functions and $p$-adic periods of modular forms, Invent. Math. 111 (1993), 407--447.



\bibitem{hida} H. Hida, Non-vanishing modulo $p$ of Hecke $L$-values, Geometric aspects of Dwork theory, Vol. I, II, Walter de Gruyter, Berlin (2004),  735--784

\bibitem{hk} C. Heuberger and S. Kropf, On the higher dimensional quasi-power Theorem and a Berry-Esseen Inequality, Monatsh. Math. 187 (2018), 293--314.

\bibitem{hennion} H. Hennion, Sur un th\'eor\`eme spectral et son application aux noyaux lipchitziens, Proc. Amer. Math. Soc. 118 (1993), no. 2, 627--634.

\bibitem{hensley} D. Hensley, The number of steps in the Euclidean algorithm, J. Number Theory. 49 (2)  (1994), 142--182.

\bibitem{iwaniec:kowalski} H. Iwaniec and E. Kowalski, Analytic number theory, A.M.S. Colloquium Publications, 53. A.M.S., Providence, RI (2004).

\bibitem{kim:sun} M. Kim and H.-S. Sun, Modular symbols and modular $L$-values with cyclotomic twists, submitted.

\bibitem{kato} T. Kato, Perturbation Theory for Linear Operators, Springer, Berlin (1980). 

\bibitem{lee:pal} J. Lee and B. Palvannan, An ergodic approach towards an equidistribution result of Ferrero–Washington, J. Th\'eor. Nombres Bordeaux, 36 (2024), no. 3, 805--833.


\bibitem{lee:sun} J. Lee and H.-S. Sun, Another note on ``Euclidean algorithms are Gaussian" by V. Baladi and B. Vall\'ee, Acta Arith. 188 (2019), 241--251. 

\bibitem{lhote} L. Lhote, Mod\'elisation et approximation de sources complexes, M\'emoire de DEA, Universit\'e de
Caen (2002).

\bibitem{manin} Y. Manin, Parabolic points and zeta functions of modular curves, Izv. Akad. Nauk SSSR Ser. Mat. 36 (1972), 19--66.

\bibitem{manin:marc} Y. Manin and M. Marcolli, Continued fractions, modular symbols, and noncommutative geometry, Selecta Math. (N.S.) 8 (2002), no. 3, 475--521.

\bibitem{mayer} D. Mayer, On the thermodynamic formalism for the Gauss map, Comm. Math. Phys. 130 (1990), no. 2, 311--333.

\bibitem{mazur} B. Mazur, Relatively few rational points, Lecture note at Caltech.

\bibitem{mazur:rubin} B. Mazur and K. Rubin, The statistical behavior of modular symbols and arithmetic conjectures, Lecture note at Toronto, \url{http://www.math.harvard.edu/~mazur/ papers/heuristics.Toronto.12.pdf}.

\bibitem{merel} L. Merel, Universal Fourier expansions of modular forms, in: {On Artin's conjecture for odd 2-dimensional representations}, 59--94, Lecture Notes in Math. 1585, Springer, Berlin, (1994).


\bibitem{naud} F. Naud, Selberg’s zeta function and Dolgopyat’s estimates for the modular surface, Lecture note at IHP, \url{http://baladi.perso.math.cnrs.fr/naudcourse.pdf}.

\bibitem{nord} A. Nordentoft, Central values of additive twists of cuspidal $L$-functions, J. reine angew. Math. 776 (2021), 255--293.


\bibitem{petridis:risager} Y. Petridis and M. Risager, Arithmetic statistics of modular symbols, Invent. Math. 212 (2018), no. 3, 997--1053.


\bibitem{parry:pollicott2} W. Parry and M. Pollicott, Zeta functions and the periodic orbit structure of hyperbolic dynamics, Ast\'erisque. 187-188, 268 (1990).

\bibitem{ruelle} D. Ruelle, Zeta-functions for expanding maps and Anosov flows, Invent. Math. 34 (1976), 231--242.

\bibitem{sarig} O. Sarig, Introduction to the transfer operator method, Lecture note at Bonn.

\bibitem{shimura} G. Shimura, Introduction to the arithmetic theory of automorphic functions, Publications of the Mathematical Society of Japan, No. 11, Princeton University Press. (1971).

\bibitem{sinnott} W. Sinnott, On a theorem of L. Washington, Ast\'erisque, 147-148 (1987), 209--224.

\bibitem{spea:will} B. K. Spearman and K. S. Williams, Values of the Euler phi function not divisible by a given odd prime, Ark. Mat. 44 (2006), 166--181.


\bibitem{stevens} G. Stevens, The cuspidal group and special values of L-functions, Trans. Amer. Math. Soc. 291 (1985), no. 2, 519--550.

\bibitem{sun} H.-S. Sun, A proof of the conjecture of Mazur-Rubin-Stein, Bull. Korean Math. Soc. 58 (2021), no. 1, 163--170.

\bibitem{titchmarsh} E. C. Titchmarsh, The theory of the Riemann zeta-function, Second
edition, Clarendon Press, Oxford (1986).


\bibitem{vatsal} V. Vatsal, Canonical periods and congruence formulae, Duke Math. J. 98 (1999), no. 2, 397--419.

\bibitem{vatsal2} V. Vatsal, Special values of anticyclotomic $L$-functions. Duke Math. J. 116 (2003), no. 2, 219--261.

\bibitem{vatsal:imc} V. Vatsal, Special values of $L$-functions modulo $p$. International Congress of Mathematicians. Vol. II, 501–514, Eur. Math. Soc., Zürich (2006).

\bibitem{washington} L. Washington, The non-$p$-part of the class number in a cyclotomic $\Zb_p$-extension, Invent. Math. 49 (1978), 87--97.

\end{thebibliography}
\end{document}